\documentclass[a4paper,12pt]{preprint}
\usepackage{amsfonts,amssymb, times}
\usepackage{amsmath,amsthm}
\usepackage{mhequ}
\usepackage{color}
\newtheorem{theorem}{Theorem}[section]
\newtheorem{lemma}[theorem]{Lemma}
\newtheorem{proposition}[theorem]{Proposition}
\newtheorem{corollary}[theorem]{Corollary}

\theoremstyle{definition}
\newtheorem{definition}{Definition}[section]

\theoremstyle{remark}
{\newtheorem{remark}[theorem]{Remark}

\numberwithin{equation}{section}
\def\sectionmark#1{\markboth{}{}}

\begin{document}

\def\ssm{\smallsetminus}
\newcommand{\tmname}[1]{\textsc{#1}}
\newcommand{\tmop}[1]{\operatorname{#1}}
\newcommand{\tmsamp}[1]{\textsf{#1}}
\newenvironment{enumerateroman}{\begin{enumerate}[i.]}{\end{enumerate}}
\newenvironment{enumerateromancap}{\begin{enumerate}[I.]}{\end{enumerate}}

\newcounter{problemnr}
\setcounter{problemnr}{0}
\newenvironment{problem}{\medskip
  \refstepcounter{problemnr}\small{\bf\noindent Problem~\arabic{problemnr}\ }}{\normalsize}
\newenvironment{enumeratealphacap}{\begin{enumerate}[A.]}{\end{enumerate}}
\newcommand{\tmmathbf}[1]{\boldsymbol{#1}}

\def\paral{/\kern-0.55ex/}
\def\parals_#1{/\kern-0.55ex/_{\!#1}}
\def\bparals_#1{\breve{/\kern-0.55ex/_{\!#1}}}
\def\n#1{|\kern-0.24em|\kern-0.24em|#1|\kern-0.24em|\kern-0.24em|}
\def\j{\jmath}
\newcommand{\A}{{\bf \mathcal A}}
\newcommand{\B}{{\bf \mathcal B}}
\def\C{\mathbb C}
\def\D{\mathcal D}
\newcommand{\dom}{{\mathcal D}om}
\newcommand{\pathR}{{\mathcal{\rm I\!R}}}
\newcommand{\Nabla}{{\bf \nabla}}
\newcommand{\E}{{\mathbb E}}
\newcommand{\Epsilon}{{\mathcal E}}
\newcommand{\F}{{\mathcal F}}
\newcommand{\G}{{\mathcal G}}
\def\g{{\mathfrak g}}
\newcommand{\HH}{{\mathcal H}}
\def\h{{\mathfrak h}}
\def\k{{\mathfrak k}}
\newcommand{\I}{{\mathcal I}}
\def\LL{{\mathbb L}}
\def\law{\mathop{\mathrm{ Law}}}
\def\m{{\mathfrak m}}
\newcommand{\K}{{\mathcal K}}
\newcommand{\p}{{\mathfrak p}}
\newcommand{\PP}{{\mathbb P}}
\newcommand{\R}{{\mathbb R}}
\newcommand{\Rc}{{\mathcal R}}
\def\T{{\mathcal T}}
\def\M{{\mathcal M}}
\def\N{{\mathcal N}}
\newcommand{\pnabla}{{\nabla\!\!\!\!\!\!\nabla}}
\def\X{{\mathbb X}}
\def\Y{{\mathbb Y}}
\def\L{{\mathcal L}}
\def\1{{\mathbf 1}}
\def\half{{ \frac{1}{2} }}
\def\vol{{\mathop {\rm vol}}}
\def\euc{{\mathop {\rm eul}}}

\newcommand{\term}{{T}}
\newcommand{\tm}{{t}}
\newcommand{\stm}{{s}}
\newcommand{\pole}{{y_0}}

\def\le{\leq}
\def\ge {\geq}
\def\ad{{\mathop {\rm ad}}}
\def\Conj{{\mathop {\rm Ad}}}
\def\Ad{{\mathop {\rm Ad}}}
\newcommand{\const}{\rm {const.}}
\newcommand{\eg}{\textit{e.g. }}
\newcommand{\as}{\textit{a.s. }}
\newcommand{\ie}{\textit{i.e. }}
\def\s.t.{\mathop {\rm s.t.}}
\def\esssup{\mathop{\rm ess\; sup}}
\def\Ric{\mathop{\rm Ric}}
\def\div{\mathop{\rm div}}
\def\ker{\mathop{\rm ker}}
\def\Hess{\mathop{\rm Hess}}
\def\Image{\mathop{\rm Image}}
\def\Dom{\mathop{\rm Dom}}
\def\id{\mathop {\rm id}}
\def\Image{\mathop{\rm Image}}
\def\Cyl{\mathop {\rm Cyl}}
\def\Conj{\mathop {\rm Conj}}
\def\Span{\mathop {\rm Span}}
\def\trace{\mathop{\rm trace}}
\def\ev{\mathop {\rm ev}}
\def\supp{{\mathrm supp}}
\def\Ent{\mathop {\rm Ent}}
\def\HS{\mathop{\rm HS}}
\def\tr{\mathop {\rm tr}}
\def\graph{\mathop {\rm graph}}
\def\loc{\mathop{\rm loc}}
\def\so{{\mathfrak {so}}}
\def\su{{\mathfrak {su}}}
\def\u{{\mathfrak {u}}}
\def\o{{\mathfrak {o}}}
\def\pp{{\mathfrak p}}
\def\gl{{\mathfrak gl}}
\def\hol{{\mathfrak hol}}
\def\z{{\mathfrak z}}
\def\t{{\mathfrak t}}
\def\<{\langle}
\def\>{\rangle}
\def\span{{\mathop{\rm span}}}
\def\diam{\mathrm {diam}}
\def\inj{\mathrm {inj}}
\def\Lip{\mathrm {Lip}}
\def\Iso{\mathrm {Iso}}
\def\Osc{\mathop{\rm Osc}}
\renewcommand{\thefootnote}{}
\def\supp{\mathrm {Supp}}
\def\V{\mathbb V}
\def\vol{{\mathop {\rm vol}}}
\def\cut{{\mathop {\rm Cut}}}
\def\Lip{{\mathop {\rm Lip}}}
\def\ric{\mathop {\rm ric}}
\def\le{\leq}
\def\ge{\geq}
\def\paral{/\kern-0.55ex/}
\def\parals_#1{/\kern-0.55ex/_{\!#1}}
\def\bparals_#1{\breve{/\kern-0.55ex/_{\!#1}}}
\def\n#1{|\kern-0.24em|\kern-0.24em|#1|\kern-0.24em|\kern-0.24em|}
\def\f{\frac}
\title{ Hessian formulas and estimates for parabolic Schr\"odinger  operators}
\author{Xue-Mei Li}
\date{}
\institute{Mathematics Institute, The University of Warwick, Coventry CV4, 7AL}
%\footnote{printed \today}
\maketitle

\begin{abstract}
We study the Hessian of the fundamental solution to the parabolic problem for weighted Schr\"odinger operators of the form $\half \Delta+\nabla h-V$  proving a second order Feynman-Kac formula and obtaining Hessian estimates. 
For manifolds with a pole, we use the Jacobian determinant of the exponential map to offset the volume growth of the Riemannian measure and use the semi-classical bridge as a delta measure at $y_0$ to obtain exact Gaussian estimates.

These estimates are in terms of bounds on $\Ric-2\Hess h$, on the curvature operator,
and on the cyclic sum of the gradient of the Ricci tensor.

\end{abstract}


\tableofcontents

\section{Introduction}
Let $M$ be a complete smooth Riemannian manifold of dimension $n$ greater than $1$, $\Delta$ the Laplace-Beltrami operator on $M$,
 $\Delta^h=\Delta+2L_{\nabla h}$ the Bismut-Witten (or weighted) Laplacian, where $h$ is a real valued smooth function on $M$,  $\nabla$ denotes the gradient operator and also stands for the Levi-Civita connection. The manifold is also assumed to be connected.  For a real valued
 bounded and H\"older continuous function $V$ on $M$, 
 let $P_t^{h,V}f$ denotes the  solution to the equation \begin{equation}\label{schroeq}
\frac{\partial }{\partial t}f_t = \left(\frac{1}{2}\Delta^h- V\right) f_t,  \; \hbox{ for } t>0;  \qquad
\lim_{t\downarrow 0} f_t=f.
\end{equation}
where $V$ is a real-valued bounded and locally H\"{o}lder continuous function on $M$.  We denote by $p_t^{h,V}$ the fundamental solution
of (\ref{schroeq}). If $h=0$ or/and $V=0$ the corresponding subscripts will be  dropped.  
There is a vast literature on parabolic Schr\"odinger operators on a complete Riemannian manifold, see 
the books \cite{Cycon-Froese-Kirsh-Simon, Davies-Safanov, Simon-82} and the reference therein, and \cite{ Molchanov, Azencott, Donnelly, Elworthy-Truman-81, Aizenman-Simon, Li-Yau}. 
We study the Hessian of the solutions and will pay particular attention to the second order derivative of $ \log p^{h}_t$, the latter plays a role in the $L^2$ analysis of the space of continuous paths on a Riemannian manifold, which we elaborate later, which was also studied in connection with parabolic Harnack inequalities. 

By an $h$-Brownian motion $(x_t, t\ge 0)$ we mean a Markov process with generator $\f 12 \Delta^h$. An $h$-Brownian bridge $(z_s^t, s\ge 0)$ is the $h$-Brownian motion conditioned to reach a point $y_0$ at a given time $t$. The former gives a probabilisitic representation for $P_t^h$ and for the Feynman-Kac semigroup via Feynman-Kac's path integration formula,
\begin{equation}
P_t^{h,V} f(x_0)=\E\left[ f(x_t)e^{-\int_0^t V(x_s)ds}\right];
\end{equation}
 while the latter plays the role of the $\delta$-measure at $y_0$ leading to the following formula for the Feynman-Kac kernel
\begin{equation}\label{fk-bridge}
p_t^{h,V}(x_0,y_0)= p_t(x_0, y_0)\E \left[e^{-\int_0^t V(z^t_s) ds}\right].
\end{equation}

Feynman-Kac formula / Feynman-path integral is an effective tool in a range of studies. For their use in connection with stochastic processes see the book \cite{Freidlin-book}, and more recently  \cite{Sturm-93,Albeverio-Mazzucchi,Gueneysu, Albverio-Kawabi-Rockner, Li-Thompson, Rincon}.  

 The results in this article are obtained under conditions on the bilinear form $\Ric-2\Hess h$ where $\Ric$ denotes the Ricci curvature and $\Hess h$ denotes the second derivative of $h$  on the curvature operator $\Rc$, and on a curvature $\Theta^h$, the cyclic sum of 
$\nabla \Ric^\sharp$.  
 For $v_1, v_2, v_3 \in T_{x_0}M$,  \begin{equation}\label{Theta}
\begin{aligned}
\left\< \Theta(v_2) v_1, v_3\right\>&=  \left( \nabla_{v_3}  {\Ric}^{\sharp}\right) (v_1, v_2) -  \left( \nabla_{v_1}  {\Ric}^{\sharp}\right) (v_3, v_2)
 -   \left( \nabla_{v_2}  {\Ric}^{\sharp}\right) (v_1, v_3),\\
 \Theta^h(v_2, v_1)&=\f 12 
\Theta(v_2)( v_1)+  \nabla^2 (\nabla h)( v_2,v_1)
+ \Rc( \nabla h, v_2)(v_1).
 \end{aligned}
\end{equation}
If $M$ is a symmetric space the gradient of the curvature vanishes, so $\Theta$ describes a  symmetry property of
 the manifold.

We will first obtain a second order Feynman-Kac formula, i.e. a formula for the Hessians, $\Hess P_t^{h,V}f$ and  $\Hess  p_t^{h,V}$,
in which \underline{neither} the initial function $f$ \underline {nor}  the potential function is differentiated, and then use them for estimates.
For this we first introduce a {\it doubly damped} stochastic parallel translation equation, a stochastic equation driven by $\f 12 \Ric^\sharp -\nabla^2 h$, $\Rc$  and $\Theta^h$,  and use it
to develop a probabilistic representation for the Hessian of the solution without potential.
Such type of formulas were previously proved in  \cite{Elworthy-Li} using the solutions of a gradient stochastic differential equation (SDE) together with its linearised and the twice linearised equations; a local Hessian formula was presented in \cite{ArnaudonPlankThalmaier} where $V=0$;  also a formula for $\Delta P_tf$ was given \cite{Elworthy-Li-form} 
not involving the derivative of the Ricci curvature. 
%See also \cite[B. Driver and A. Thalmaier]{Driver-Thalmaier}for first order consideration on vector bundles.
Differentiation formulas of all orders for heat kernels was obtained in  \cite{Norris-93},
generalising  a formula in \cite{Bismut} for $\nabla \log p_t$ in terms of the Brownian bridge, see also
\cite{Driver-Thalmaier}.

The stochastic damped parallel translation process $(W_t)$ is given by an ODE driven by the zero order part of the wighted Laplacian on differential 1-forms, $\Delta^h=-\nabla^*\nabla -(\f 12 \Ric^\sharp -\nabla^2 h)$. They are also the local conditional expectations of the derivative flows,  with respect to the filtration of $(x_t)$, the latter are
solutions of the linearised gradient SDE. Thus $dP_tf(v)=\E df(W_t(v))$. Using the solutions $W_t^{(2)}$ to the
doubly damped stochastic parallel translation equation, we prove that
$$\nabla df(v_1, v_2)=\E \nabla df(W_t(v_2), W_t(v_1))+\E df(W_t^{(2)}(v_2, v_1)),$$
holds under conditions \underline{\bf C1(a)}, \underline{\bf C3}, and \underline{\bf C5(g)} in \S\ref{section-main-results}.
The Hessian of $P_t^h$ does not satisfy the commutative relations 
 enjoyed by $d$ and  $\Delta^h$,
  $de^{\f 12 \Delta^h} f=e^{\f 12 \Delta^h} (df)$ and $\Delta^h e^{\f 12 \Delta^h} f=e^{\f 12 \Delta^h} (\Delta^h f)$,
  explaining the involvement of $\Theta^h$ in the doubly damped stochastic parallel translation whose solutions 
  are in fact the local conditional expectations of the spatial derivatives of~$W_t$. 

To avoid differentiating the potential term, we use the variation of constant formula and has the following formal expression:
$$\Hess P_t^{h,V}f(v_1, v_2) =\Hess P_t^h f(v_1, v_2)+\int_0^t \Hess \left( P^h_{t-r}(V P_r^{h,V} f)\right) (v_1, v_2) dr.$$
Every derivative shifted from the initial value $f$ to the heat kernel, causes a loss of integrability in time by the order of  $\f 1 {\sqrt t}$.
 In particular  $|\Hess P_{t-r}^h|$ is of order  $\f 1{t-r}$ and the integral given above has a formal singularity at $r=t$.  See
 Theorem \ref{second-order} which is proved under condition \underline{\bf C1}.
 Another type of singularity will present itself when we work on the Feynman-Kac kernel $p_t^{h,V}$.

 To obtain information on the fundamental solution $p_t^{h,V}$, one might wish to use the Brownian bridge. 
 This does not work very well. For example, the usual formula for $\nabla p_t$ using the Brownian
 bridge is circular as the Brownian bridge is a Markov process with Markov generator 
 $\f 1 2 \Delta+\nabla \log p_{T-t}(\cdot, y_0)$ where $T$ being the terminal time and $y_0$ the terminal value.
The extension of the operator to differential 1-forms is $\f 12 \Delta+\nabla^2 \log p_{T-t}(\cdot, y_0)$.
However it is possible to obtain information on the Hessian of the Schr\"odinger 
kernel  from that on $ p_t^h$. Although our aim is to obtain precise estimates for the weighted schr\"odinger kernel
 on manifolds with a pole, we did obtain Hessian estimates for the $P_t^{h,V}f$ and on $\Hess p_t^h$ on
   general complete  Riemannian manifolds,  generalising those in
\cite{Sheu,Norris-93,HsuEstimates,Aida,Stroock2000,Stroock-Turetsky-98, XDLi-Hamilton}.
 These are proved  under conditions \underline {\bf C2} and \underline {\bf C1(c)} from \S\ref{section-main-results}.
%For these estimates and formulas we use extensively the exponential integrability of 
%  the damped and doubly damped stochastic parallel translations and that of the square of the distance of the Brownian 
%  motion from a fixed point.\\

 Our main results are for manifolds with a pole, for which we use the semi-classical bridge introduced by Elworthy and Truman \cite{Elworthy-Truman-81}  to gauge the estimates. The semi-classical bridge before the terminal time can be considered as an approximation to the delta measure. 
Our method will depart slightly from  that in \cite{Elworthy-Truman-81}, taking advantage of the equivalence of the semi-classical bridge with the $h$-Brownian motion before the terminal time and navigating  the singularity at the terminal time, this latter requires careful analysis. For the second order derivative of $P_t^{h}$ the singularity problems becomes pronounced.
  We are also able to avoid to differentiate the
damped stochastic parallel translations along the semi-classical bridge. Instead we use the damped stochastic parallel translation equations for the $h$ Brownian motion,
and run them along semi-classical bridge paths.

   %See \underline{\bf C4} in \S\ref{section-Hessian-estimates} for our set of conditions.

 A point $y_0$ is a pole for $M$ if the exponential map at $y_0$ is a diffeomorphism from the tangent space $T_{y_0}M$ at $y_0$ to $M$.
If $M$ is a manifold with negative sectional curvature then every point of the manifold is a pole.
For $\R^n$,  $p_t(x,y)=(2\pi t)^{-\f n2} e^{-\f{|x-y|^2}{2t}}$, $ \nabla_\cdot \log p_t(\cdot, y_0)=-\f {d\nabla d(x_0, y)} {t}$,
and for $I$  the $n\times n$ identity matrix, $\f{\nabla d p_t}{p_t}=-\f 1t +\f {d^2(x_0, y)} {t^2}$.
The simplest heat kernel for the hyperbolic space is in dimension $3$, for which
 $$p_t(x,y_0)= (2\pi t)^{-\f 32} e^{-\f t 2 }e^{-\f {d ^2(x,y_0)}{2t}}J_{y_0}^{-\f 12 }(x),   \quad J_{y_0}(x)=(\f {\sinh d(x,y_0)}{d(x, y_0)})^2.$$
 But computations for its derivatives are already messy.
 For a generic manifolds, we  expect an estimate of the order $\f {d(x_0, y)} {t}+\f 1 {\sqrt t}$ for the gradient and 
 $(\f {d(x_0, y)} {t}+\f 1{ \sqrt t})^2$ for the Hessian. 
% We are concerned with  the precise conditions under which these estimate hold as well as the
% constants in front of these estimates. 
 Observe that for a Brownian motion, $\f {d(x_t, y_0)}t\sim \f 1 {\sqrt t}$
 and so the above estimates are in effect of order $\f 1 {\sqrt t}$ and $\f 1 t$ respectively.  
 However for the Hessian, the Euclidean exact formula has a lot of advantage.

If $y_0$ is a pole, the heat kernel and its derivatives are expected to be also
 Gaussian like, but we will need to discount the growth of the volume at a point away from $y_0$. It is reasonable to base the study on an approximate  ansatz :
 $$k_t(\cdot, y_0)=(2\pi t)^{-\f n 2}e^{-\f {d^2(\cdot, y_0)}{2t}}J_{y_0}^{-\f 12},$$
 where   $J_{y_0}$ is the Jacobian determinant of the exponential map at $y_0$, also called the Ruse invariant \cite{Walker, Elworthy-book}.
 For the hyperbolic space, $J = \left(\frac{1}{d}\sinh {d}\right)^{n-1}$. 
Our ansatz for the Hessian of $p_t$ are  of the same form,
\begin{equation}\label{Hessian-idea}
\nabla dp_T^{h,V}= e^{h(y_0)-h(x_0)}k_t(x_0,y_0) A_T.
\end{equation}

The correction term $A_T$ will be given by the semi-classical bridge $\tilde x_t$, introduced in \cite{Elworthy-Truman-81}, with terminal time $T$ and terminal value $y_0$. It  is a strong Markov process with time dependent Markov generator  $\f 12 \Delta+ \nabla \log k_{T-t}(\cdot, y_0)$ where $0<t\le T$.
The excellent property it enjoys is that its radial part is the $n$-dimensional Bessel bridge.
The estimates will crucially depend on that of the Bessel bridge.

  Upper and lower bounds of the following form 
$ C \f 1 {V(x,{\sqrt t})} e^{-\f {d^2(x,y)} {Ct}}$ where $V(x, {\sqrt t})$ is the volume of the geodesic ball at $x$ of radius $\sqrt t$,
was first proven for manifolds of bounded geometry, see
\cite{Donnelly} and also  \cite{Cheeger-Gromov-Taylor}. It is remarkable that this extends
to manifolds whose curvature is bounded between two constants and to manifolds of non-negative curvature  \cite{CLY, Li-Yau}.
For manifold with a pole the upper bound follow readily from the
 `elementary formula' of Elworthy and Truman, for $\Phi=\frac{1}{2}J_{y_0}^{\frac{1}{2}}\Delta J_{y_0}^{-\frac{1}{2}}$,
\begin{equation}
p_T^V(x_0,y_0)=k_T(x_0, y_0)\E \left(e^{\int_0^T (\Phi-V)(\tilde x_s) ds}\right).
\end{equation}
noted that for the hyperbolic space, $\Phi = -\frac{(n-1)^2}{8}+\frac{(n-1)(n-3)}{8}\left(d^{-2}- \sinh^{-2}({d})\right)$.

Unlike the standard representation by the Brownian bridge, (\ref{fk-bridge}), this representation becomes amenable for the heat kernel estimates as soon as the Riemannian metric on $M$ is not trivial. The elementary formula was generalised in \cite {Watling1} and used by M. Ndumu \cite{Ndumu91} to obtain  very nice heat kernel estimates.   Using elementary formulas for the study of the heat kernel were pioneered by M. Ndumu \cite{Ndumu-09} and S. Aida \cite{Aida-Gaussian-estimates}.  In \cite{Aida-Gaussian-estimates}
the estimates are obtained under assumptions on the derivatives of Ruse invariant and on the derivatives of the distance function to the pole, also the gradient of the Ricci curvature are assumed to be bounded. 
Very recently 
the semi-classical measure on the pinned path space is studied in \cite{Li-ibp-sc} and gradient estimates for the Schr\"odinger operator $\half \Delta^h-V$ are obtained in \cite{Li-Thompson}, where the Ruse invariant is not differentiated.

 The ansatz (\ref{Hessian-idea}), which we will deduce under condition \underline{\bf C4},
  would lead easily to an exact Gaussian bound for the Hessian of $p_t$, and we will prove the correction term 
 $A_t$ has the appropriate short time asymptotics.

%This study can be considered as a sequel to
% and to \cite{Li-Thompson}, in the former an integration by parts formula is obtained for the semi-classical bridge measure
%and  in the latter gradient estimates for both $P_t^{h,V}f$ and  $ p_T^h(\cdot, y_0)$ are obtained.We remark that the first-order Feyman-Kac formula in \cite{Li-Thompson} leading to a good estimate on $dP_t^{h,V}f$ appears to be the one involving $dV$.

  We outline the paper.
In \S\ref{preliminary} we introduce the notation, discuss the intuitive ideas, motivations and related work.
The main results and {\it geometric conditions} will be presented in \S \ref{section-main-results}. In \S \ref{section-gradient}  moment estimates and the strong 1-completeness for gradient SDEs are obtained. In \S \ref{section-doubly-damped} we prove the primitive Hessian formula, define and study
   the doubly damped stochastic parallel translation equation estimating the norm of its solutions and proving their exponential integrability, and estimate the relevant terms in the Hessian formulas given in \S \ref{section-Hessian}. The immediate consequences of the Hessian formulas
   are the basic Hessian estimates for general manifolds
    \S\ref{section-Hessian-estimates}.
  In \S\ref{section-pole}
  we justify the `exact Gaussian formulas' for the Hessians in terms of the semi-classical bridge, leading to 
  exact Gaussian estimates,  \S\ref{section-exact-Gaussian}, for $\nabla^2 p_t$.

 \section{Preliminaries}\label{preliminary} 
We begin with explaining the intrinsic ideas behind the damped stochastic  parallel translation,
with heat equations on differential one forms and  the derivative flow of a gradient SDE.
We then discuss estimates on the moments of the derivative flows and the continuous dependence of the solution flow of an SDE on its initial value.  Finally we remarks on Hessian estimates for the logarithms of the heat kernels.

 Throughout this paper $\{ \Omega,\mathcal{F},\mathcal{F}_t,\mathbb{P}\rbrace$ denotes the underlying probability space with filtration $\F_t$ satisfying the usual hypothesis.
Let $\lbrace B_t^i \rbrace$ be a family of independent one-dimensional Brownian motions and we write $B_t=(B_t^1, \dots, B_t^m)$. The  notation $\circ dB_s$ indicates  Stratonovich integration. 

  \subsection{Heat equation for differential 1- form and damped parallel translation}
Let $\Ric_x: T_xM\times T_xM\to \R$ denote the Ricci curvature, 
 ${\Ric_x}^{\sharp}: T_xM\to T_xM$  be the linear map given by $\< {\Ric}^\sharp_x u,v\> = {\Ric_x}(u,v)$. Similarly $(\nabla^2 h)^{\sharp}=\nabla_\cdot \nabla h$.
  We also denote by ${\Ric_x}^{\sharp}$ the corresponding operator on differential 1-forms. 
Let $\delta^h$ denote the $L^2$ adjoint of the differential $d$ with respect to the weighted volume measure
 $e^{2h}dx$, then $\Delta^h=-(d+\delta^h)^2$. Furthermore $d+\delta^h$ is essentially self-adjoint on the space of smooth functions  (and on smooth differential forms) with compact supports, on which $d$ and $\Delta^h$ commute \cite{Li-thesis}. In particular 
 $dP_t^{h}f$ solves the following equation on differential 1-forms with initial condition the exact differential 1-form $df$,
 $$\frac{\partial }{\partial t}\phi_t = \frac{1}{2}(\nabla)^{h,*} \nabla \phi_t +\left( -\f 12 {\Ric}^{\sharp}+ (\nabla^2 h)^{\sharp}\right )\phi_t,$$
 where $(\nabla)^{h,*}$ is the adjoint of $\nabla$ with respect to $e^{2h}dx$.

 It is therefore not surprising that probability representations for  $dP_t^hf$ involves  the solution to the following damped stochastic parallel equation\begin{equation}\label{Wt-pre}
\f {DW_t} {dt}=\left( -\f 12{\Ric}^\sharp _{x_t}+ (\nabla^2 h)^{\sharp} \right) (W_t),
\end{equation}
Its solution with initial value $v_0$ will be denoted by $W_t(v_0)$. It is a stochastic process of finite variation.  Here the equation is defined path by path, for each $\omega$,
$\f {DW_t(\omega)} {dt}=\parals_t\f {d}{dt} \parals_t^{-1} $
and $\parals_t\equiv \parals_t (x_\cdot(\omega))$ is the stochastic parallel translation along the
the $h$-Brownian paths.
The stochastic process $W_t$ takes values in $\LL(T_{x_0}M;T_{x_t}M)$.

The following probabilistic representation, $$e^{\f 12 t \Delta}df(v)=\E df(W_t(v)),$$  was given in H. Airault \cite{Airault}, initially for $h=0$ and for compact manifolds.  The formula holds in fact for Brownian motion with a general drift.

% Set
%\begin{equation}\label{rhoh}
%\rho^h(x) =\inf_{v\in ST_xM} \{\Ric(v,v)-2 \Hess(h)(v,v)\},\end{equation}
%then $\|W_t\|^2$ is bounded by $e^{-\int_0^t \rho^h(x_s) ds}$. If $\rho^h$ is bounded from below, the $h$-Brownian motion does not explode, Airault's representation holds and $|dP_tf|_\infty \le C |df|_\infty$.
%We will take this idea further to deal with second order problems and obtain an primitive Hessian formula ,
%see Proposition \ref{second-diff-formula}. 

\subsection{Linearised SDE, the derivative flow, and  strong 1-completeness }

For $r \ge 2$, let  $\{X_i\}_{i=1}^m$ be $C^{r+1}$ vector fields, $X_0$ a $C^r$ vector field. Let $(F_t(x), t<\xi(x))$ denote the maximal solution to the SDE  \begin{equs}
\label{sde}
dx_t=\sum_{i=1}^m X_i(x_t)\circ dB_t^i +X_0(x_t)dt=X(x_t)\circ dB_t+X_0(x_t)dt.
\end{equs}
 with initial value $x$. For a fixed point $x_0\in M$ we write $x_t=F_t(x_0)$ for simplicity. 
 The SDE  is said to be {\it complete or conservative} if for every $x$,
$\xi(x)=\infty$ almost surely, which does not necessarily mean that  for a common set $\omega$ of measure zero, 
$F_t(x)$ has infinite life time for every $x$.  The completeness property is determined by the generator of the SDE.
Assume the SDE is complete. Another useful commutation relation, should it hold, is
$$dP_t f(v)=\E df(TF_t(v)),$$
where $TF_t(v)$ is the {\it derivative flows} of (\ref{sde}), it is solution to
 the linearised SDE:
$$Dv_t=\sum_{i=1}^m \nabla X_i(v_t)\circ dB_t^i + \nabla X_0(v_t) dt.$$
Let  $\F_t^{x_0}=\sigma\{x_s: s\le t\}$ and $v\in T_{x_0}M$. Then $W_t(v)$ is the local conditional expectation of $TF_t(v)$, and if the latter is integrable, $$W_t(v)=\E\{T_{x_0}F_t(v)|\F_t^{x_0}\}.$$

We say an SDE is {\it strongly complete} or it has a {\it global smooth solution flow} if $(t,x)\mapsto F_t(x, \omega)$ is continuous on $\R_+\times M$ almost surely.
For one dimensional manifolds, the two concepts are equivalent. 
It is common to believe that well posedness implies strong completeness,  or the first exit time of the solution from a set with smooth boundary is  continuous with respect to the initial value. See e.g. \cite{Lack-flow} for discussions on counter examples. 
Consider $\dot z=z^2$ on the complex plane, a solution of which
has finite life time if and only if its initial value is a non-zero real number. Indeed for $\delta \not =0$ and $x_0\not =0$, consider $z(t)=\f {x_0(1+i\delta)}{1-x_0(1+i\delta)t}$. 
Its norm is bounded in time and so for sufficiently
large $a>0$, its first exit time from the ball of radius $a$ is infinity, unless $\delta =0$; while the solution from $x_0$ will exit the ball in finite time.
For an SDE, even if the solutions from any given point does not explode, the solution must restart from  random initial
points for which  the uncountability of the exceptional zero measure sets must be taken into consideration.
 If the system is strongly complete then for every $t$, $x\mapsto F_t(x, \omega)$ is $C^r$. 
 It is sufficient to use a weaker notion of strong completeness.
\begin{definition}\cite{Li-flow}
\label{def-1-complete}
The SDE is strongly 1-complete if  for any smooth curves $\sigma: [0,1]\to M$, $F_t(\sigma(s), \omega)$ is continuous in $(s,t)$ almost surely for all $t\in [0,T]$ and $s\in [0,1]$ for some $T>0$.  
\end{definition}
This is a weaker concept than the existence of a global smooth solution flow.
Between strong 1-completeness and strong completeness there is a range of notions: 
the strong $p$-completeness for $p=2, \dots, n-1 $. Roughly speaking the flow takes a $p$-dimensional sub-manifold  to a $p$-dimensional sub-manifold. If $p=n-1$, this is equivalent to strong completeness. 
If $M=\R^n$, $X_i$  are globally Lipschitz continuous and the stochastic integral  are in It\^o form,  then the strong completeness follows from the fixed point theorem. For compact manifolds, the complete SDE can be lifted to a complete  SDE on the group of diffeomorphism over $M$ and the strong completeness follows.
 Because of the difficulty in localising, we introduce the derivative flow method to obtain strong completeness for non-compact manifolds and for SDEs on $\R^n$ without Lipschitz conditions.

  We illustrate the use of this in the lemma below is taken from \cite[Cor 9.2]{Li-flow}, we make a slight improvement on this,  replacing the continuity of
 $|T_{\sigma(s)}F_t|$ in $L^1$ there by a more intrinsic condition.  
 \begin{lemma} \label{exchange-lemma}
Suppose that (\ref{sde})  is strongly 1-complete with $\E|TF_t|$ finite. Suppose that 
  $s\mapsto|W_t(\dot\sigma(s)) |$ is continuous in $L^1$.  If $f$ is $BC^1$ then 
for $v\in T_{x_0}M$ and $x_0 \in M$, $d(P_t^hf)(v)=\E [df(W_t(v))]$.
%If $\E [\sup_{s\le t}|W_t|] $ is finite then, $\E|TF_\tau |$ is finite for any stopping time $\tau$.
\end{lemma}
\begin{proof}
Take a normal geodesic curve $\sigma\colon  [0,\ell]\to M$  with initial value $x_0$ and initial velocity $v$. 
By the strong 1-completeness, for almost all $\omega$, $F_t(\sigma(s))$ is  differentiable with respect to $s$ and
$d(P_tf)(v)=\lim_{s\to 0}  \E\left(\f 1 s \int_0^s df
\left(T_{\sigma(r)}F_t\left(\dot \sigma(r),\omega\right) \right)dr\right)$.
%=\lim_{s\to 0} \E\left( \f {f\left(F_t\left(\sigma(s), \omega\right)\right)
%-f\left(F_t(x, \omega)\right)} {s}\right)\\
If  $T_{\sigma(r)}F_t(\dot \sigma(r))$ is integrable, then $\E \{TF_t(\dot \sigma(r))|\F_t^{\sigma(r)}\}$ is $W_t(\dot \sigma(r))$,
from which we see that $d(P_tf)(v)=
\lim_{s\to 0} \f 1 s \int_0^s  \E[df (W_t(\dot \sigma(r))]dr$.
%Since $\E[df (W_t(\dot \sigma(r))]=(P_t^h df)(\dot \sigma(r))=d(P_t^h f)(\dot \sigma(r))$ is integrable,
By the strong 1-completeness $W_t(\dot \sigma(r))$ is continuous in $r$.
If it is continuous in $L^1$ and $f \in BC^1$ we may take $r\to 0$ to obtain  $\E df(W_t)$
on the right hand side. 
\end{proof}
\begin{remark}
If $\sup_{s\in [0, \ell]} \E|W_t(\dot \sigma(s))|^{1+\delta}$ is finite for some $\delta>0$, then $s\mapsto|W_t(\dot\sigma(s))|$ is continuous in $L^1 $,
 which holds in particular under one of the conditions from \underline{\bf C5(a)} to \underline{\bf C5(d)}. See Lemma \ref{lemma:strong-1}.
\end{remark}

 The derivative flow describes the evolution of
the distance, $d(F_t(x), F_t(y))$, between two solutions with initial values $x$ and $y$ respectively. 
We denote by $v_t:=TF_t(v_0)$  its  solution with initial value $v_0\in T_{x_0}M$.
Then, c.f \cite{Li-moment},
 \begin{equation}
\begin{array}{ll}
|v_t|^p=|v_0|^p +p\sum_{i=1}^m\int_0^t|v_s|^{p-2} \<\nabla_{v_s} X^i, v_s\> dB_s^i
+\f p2 \int_0^{t}|v_s|^{p-2}H_p(x_s)(v_s, v_s)ds,
\end{array}
\label{eq: vt}
\end{equation}
where for $v\in T_xM$,  
\begin{equation}\label{hp}
 \begin{array}{ll}
&H_p(x)(v,v)= \sum_{i=1}^m \<\nabla^2 X^i(X^i,v), v\>
+\sum_1^m \left\<\nabla_{\nabla_v X^i} X^i,v\right\> \\
&\qquad\qquad +\sum_1^m |\nabla X^i(v)|^2
+(p-2)\f 1 {|v|^2}\sum_1^m \<\nabla _vX^i, v\>^2+
2\<\nabla _vX_0, v\>.
\end{array}
\end{equation}
\normalsize
set \begin{equation}
\label{h1}
\overline H_p(x)=\sup_{|v|=1, v\in T_xM}H_p(v,v)
\end{equation}
If  $\overline H_1$ is bounded above, then $\sup_{x\in D} \E \sup_{s\le t} |T_xF_s|$ is finite for every compact set $D$ and the SDE is strongly 1-complete. 
If for a number $p \ge 1$ and a function $f:M\to \R$ satisfying that
\begin{equation}
\label{criterion}
\sup_{x\in D}\E\exp\left({6p^2 \int_0^t f(F_s(x))\chi_{s<\xi(x)}ds}\right)<\infty,
\end{equation}
the following two conditions hold, (a) $\sum_{i=1}^m|\nabla X_i|^2\le f$,  (b) $\overline H_p\le 6f$.   Then
 $\sup_{x\in D} \E \sup_{s\le t} |T_xF_s|^p$ is finite for every compact set $K$ and the SDE is strongly 1-complete. 
For a weaker criterion  see \cite[Thm. 3.1 \& 5.3]{Li-flow}. For $\sup_{x\in D} \E |T_xF_t|^p$ to be finite, we may use localising techniques to remove condition (a). See also \cite{Chen-Li-flow}.

\subsection{Hessian Estimates}  Finally we comment on $ \log p_t$.
The gradient and the Hessian of $p_t$ have been considered in the study of parabolic Harnack inequalities, 
see \cite{Li-Yau,Hamilton,Aizenman-Simon}, also the book \cite{Jost}. For bounded strict elliptic second order differential operators on $\R^n$, they are obtained in \cite{Sheu}. For compact manifolds the relevant small time gradient  estimate was given and used  in \cite{Driver92}.
The remarkable Hessian estimates $|\nabla^2 \log p_t| \le C(\f 1 {\sqrt t}+\f d t)^2$ and indeed `off cut-locus' estimates for derivatives of all order are given in \cite{Malliavin-Stroock}
with generalisation to \cite{HsuEstimates,Aida,Stroock2000,Stroock-Turetsky-98}, and more recently,  \cite{XDLi-Hamilton},
see also \cite{Engoulatov,Coulhon-Zhang}.

These estimates are naturally relevant 
 in the study of loop spaces (continuous loops), where the Brownian bridge measure
 is used as a reference measure playing the role of the Lebesque measure on $\R^n$.
   The basic questions concerning the Brownian bridge measures
 include the tail behaviour of the measure, Poincare and Logarithmic Sobolev inequalities (LSI), which we discuss further later.
 One also wish to obtain a Hodge decomposition theorem of the space of  $L^2$ spaces of differential forms and to seek a Hodge deRham theorem linking the topology of the infinite dimensional pinned path space with the cohomology defined by the Hodge decomposition,
  \cite{Shigekawa,Fang-Franchi-97,Elworthy-Li-icm, Elworthy-Li07,Aida}.

If $M=\R^n$, the Brownian bridge measure is Gaussian for which  L. Gross \cite{Gross-91} proved his celebrated logarithmic Sobolev inequality,
this extends to pinned path spaces over Lie groups,  \cite{Airault-Malliavin,Driver-Lohrenz,Fang-99}. Although there are progress, for example see   \cite{Gong-Rockner-Wu, Arnaudon-Simon, Gong-Ma},   a counter example of a manifold for which Poincar\'e inequality does not hold on its loop space was  given in \cite{Eberle}. Further positive results were established in  \cite{Chen-Li-Wu,Aida},  the first is for the hyperbolic space and the latter was for asymptotically Euclidean manifolds.  It is likely that the validity of the Poincar\'e inequality for the natural metric induced by
stochastic parallel translation measures how far the Brownian bridge deviates from a Gaussian measure. In particular we have the conjecture.

{\bf Conjecture. } Poincar\'e inequality on the loop space, with reference measure  the Brownian bridge measure, does not hold on the spheres.

%  we believe that the Poincar\'e inequality for the Brownian bridge measure, with norm induced by the stochastic parallel transport is unlikely to
% hold for a large class of loop spaces, e.g. for loop space over a sphere.  On the other hand, we believe that
% the semi-classical bridge will be a reference measure on loop spaces over  manifolds with a pole.
% 

 \section{Main results}\label{section-main-results}
 Let $(x_t)$ be an $h$-Brownian with initial value $x_0$, $\parals_t\equiv \parals_t(x_\cdot):T_{x_0}M\to T_{x_t}M$ the stochastic parallel translation, and $W_t: T_{x_0}M\to T_{x_t}M$ the damped stochastic parallel translation so that  for almost every $\omega$,
$\parals_t^{-1}W_t(\omega)$ is a matrix valued stochastic process solving a linear ODE given by $-\f 12 \parals_t^{-1}Ric^{\sharp} \parals_t
+\parals_t^{-1} \nabla^2 h({\parals_t\cdot} )$.
 
 Let $v_1, v_2 \in T_{x_0}M$, fix $W_t(v_1)$, $W_t(v_2)$.
We introduce the doubly damped stochastic parallel translation equation,  along the $h$-Brownian paths:
\begin{equation}\label{doubly-damped}
\begin{aligned}
 {D} W_t^{(2)}(v_1, v_2)
=&\left(- \f 12 {\Ric}^{\sharp}_{x_t}   +\nabla_\cdot \nabla h\right)  \left(W_t^{(2)} (v_1, v_2)\right)dt +\f 12 
\Theta^h(W_t(v_2))( W_t(v_1)) dt\\
&+\Rc( d\{x_t\}, W_t(v_2)  )W_t(v_1),\\
W_0^{(2)}(v_1, v_2)=&v_2, \qquad  v_1, v_2\in T_{x_0}M,
\end{aligned}
\end{equation}
where $d\{x_t\}$ denotes integration with respect to the martingale part of  $(x_t)$, see Definition \ref{def-martingale-part}.
\begin{definition}
The  solution to the doubly damped covariant stochastic integral equation, denoted by $W_t^{(2)}(v_1, v_2)$,
 is called the doubly damped stochastic parallel translation.\end{definition}
Unlike the damped stochastic parallel translation, which is  a process of finite variation and whose norm has a pointwise upper bound
 $\|W_t\|^2\le e^{-\int_0^t \rho^h(x_s) ds}$,
 the doubly damped one has a non-trivial martingale part. 
 The estimates on $\|W_t^{(2)}\|$ are given in \S\ref{section-doubly-damped}.
 Its exponential integrability  requires growth conditions on the curvature $\Rc$ and $\Theta$. These estimates are  also essential
for validating taking the time of the semi-classical Brownian bridge to the terminal time in the forthcoming formula for the hessian of the
Feynman-Kac formula for manifolds with a pole, see e.g.\S\ref{section-estimate-2}, Lemmas  \ref{exponential-estimates} and \ref{estimate-10}.

\subsection{Geometric Conditions} 
\label{section-conditions}
 We state the set of conditions to which we  refer throughout this paper.  Let $\Rc$ denote the curvature and $\Theta$ is the cyclic sum of the covariant derivatives of the Ricci curvature defined by (\ref{Theta}).
  Let $ST_xM$ denotes the unit tangent space at $x$, $v_i\in T_xM$, and let $\{E_i\}$ be an o.n.b. of the tangent space $T_xM$.   Denote by $\|\Rc(\cdot, v_2)v_1\|$ the Hilbert-Schmidt norm of the linear map $\Rc(\cdot, v_2)v_1: T_xM\to T_xM$,
 $\|\Rc(\cdot, v_2)v_1\|:=\left(\sum_{i=1}^n |\Rc(E_i, v_2)(v_1)|^2\right)^{\f 12}$. Set
  \begin{equation}
\label{norms}
\begin{aligned}
\|\Rc_x\|&= \sup_{v_1, v_2\in ST_xM} \|\Rc(\cdot, v_2) v_1\|, \quad \|\Theta^h\|
= \sup_{v_1, v_2\in ST_xM}\left\{ | \Theta^h(v_2)v_1|\right\}.\end{aligned}
\end{equation}
We then fix an isometric immersion of the manifold in an Euclidean space $\R^m$ and denote by
 $\alpha_x$ its second fundamental form. Set
 \begin{equation}\begin{aligned}
H_1(v,v)&=-\Ric(v,v)+2 \Hess(h)(v,v)+|\alpha_x(v,\cdot)|_{\HS}^2- |\alpha_x(v,v)|^2\\
H_2(v,v)&=-\Ric(v,v)+2 \Hess(h)(v,v)+|\alpha_x(v,\cdot)|_{\HS}^2.
\end{aligned}
\end{equation}
 We first observe that $|\alpha_x(v,\cdot)|_{\HS}^2\ge  |\alpha_x(v,v)|^2$
and if  $m=n+1$, then $ H_1=-\Ric+2 \Hess(h)$. Set for $p=1,2$,
$$\overline H_p(x)=\sup_{|v|=1, x\in T_xM} H_p(v,v), \quad \rho^h(x) =\inf_{v\in T_xM, |v|=1} \{\Ric(v,v)-2 \Hess(h)(v,v)\}.$$
We remark that $H_1$ or $ H_2$ does not appear in the Hessian formula nor in the estimates. The controls over $\overline H^p$ leading to $L^p$ boundedness of the derivative flows on any compact time interval and the bounds are
 locally uniform in space.\\

Condition \underline{\bf C1} is used to obtain precise formulas;
 Condition \underline{\bf C2} is used to obtain exponential integrability of the doubly damped process;
 Conditions \underline{\bf C2} and  \underline{\bf C1(c)} are used to obtain Hessian estimates on a general complete Riemannian manifold;
  {\bf  \underline{C4}} is used to obtain estimates on manifolds with a pole.\\
  
Let  $K\in \R$, $c,C,\delta,\delta_1, \delta_2$ are positive constants. Let $\alpha_2$ be a sufficiently small positive number. 
Let $r$ denotes the distance function from a given point.
\begin{enumerate}
\item [{\bf  \underline{C1}.}]
\begin{itemize}
\item [(a)]  $\rho^h\ge -K$; \item[(b)]  $\sup_{s\le t}\E( \|W_s^{(2)}\|^2)<\infty$;
\item[(c)] for all  $f\in BC^2(M;\R)$, $v_1, v_2 \in T_{x_0}M$,
\begin{equation}\label{primitive-formula-intro}
\Hess (P_tf)(v_2, v_1) =\E \left[\nabla df(W_t(v_2), W_t(v_1))\right]+\E \left[df (W_t^{(2)}(v_2, v_1))\right].
\end{equation} 

\end{itemize}
\item  [{\bf  \underline{C2}.}]
$|\rho^h| \le K$,  $\|\Rc_x\|\le \|\Rc\|_{\infty}$,  and $\|\Theta^h\|^2\le c+ \delta r^2$ for $\delta$ sufficiently small.
\item [{\bf  \underline{C3}.}]  $\|\Theta^h\|+\|\Rc\| \le ce^{C r}$ for $C<\alpha_2$.

\item [{\bf  \underline{C4}.}] Let $y_0$ be a pole.  \begin{enumerate}
\item [(a)] $\rho^h\ge -K$,  
$|\nabla \log J|+|\nabla h|\le c e^{\delta_1 r^2}$ where $\delta_1<\alpha_2$, and {\bf  \underline{C3}}.
\item[(b)]
$\Phi^h\le C+\delta_2 r^{2}$ where   
 $\delta_2$ is sufficiently small and $$\Phi^h= - \f 1 2 |\nabla h|^2-\f 1 2\Delta h+\frac{1}{2}J_{y_0}^{\frac{1}{2}}\Delta J_{y_0}^{-\frac{1}{2}}.$$
\item[(c)] \underline {\bf C1(c)}.
\end{enumerate}
\item [ \underline {\bf C5.}]
\begin{enumerate}
\item[(a)]  $\rho^h\ge -C(1+r^2)$ and   $\<\nabla h, \nabla r\> \le c(1+r)$;
\item[(b)]   $\underline\Ric\ge -C(1+r^2)$ and   $\<\nabla h, \nabla r\>\le c(1+r)$;
\item[(c)] There exist Borel measurable functions $f, g: \R_+\to \R_+ $ s.t. 
$$|h|\le f\circ r, \quad \rho^h\ge -g\circ r, \quad (fg)(s) \le c(1+s^2).$$
\item[(d)] $\rho^h\ge - K$.
\item [(e)]$\overline  H_1\le c(1+\ln r)$
\item[(f)]  $\overline H_1\le c(1+ r^2)$.
\item[(g)] $\overline H_2\le c+\delta r^{2}$, $\delta$ is sufficiently small.
  
\end{enumerate}

\end{enumerate}
\begin{remark}
 \begin{enumerate}
 \item [(i)] Conditions \underline{\bf C1(a)} + \underline{\bf C3} imply  \underline{\bf C1(b)} hold ( Lemma \ref{moments-10}).
\item [(ii)] Any one of the conditions from \underline{\bf C5(a)} to \underline{\bf C5(d)} implies that the $h$-Brownian motions do not explode and 
finite moments of all order of its radial process.

If in addition \underline {\bf C5(e)} holds, 
the gradient SDE ( c.f. \S\ref{section-gradient}), with generator $\f 12 \Delta^h$,   is strongly 1-complete. If  \underline {\bf C5(d)}+\underline {\bf C5(f)} hold, the corresponding gradient SDE  is strongly 1-complete
and the square of its radial process  is exponentially integrable for a small exponent. See Theorem \ref{theorem-strong}.

% (by which we mean that its diffusion coefficients are given by the vector fields induced from an isometric embedding and its drift  vector field is $\nabla h$)
 
If  \underline {\bf C5(g)} holds, the gradient SDE has square integrable derivative flow, with bound locally uniform in time and in space.
 \end{enumerate}
\end{remark}

\noindent
{\bf Proposition \ref{second-diff-formula}.} 
Assume \underline{\bf C1(a)}+\underline{\bf C3}. If furthermore  a gradient SDE  is strongly 1-complete with square integrable derivative flow, locally uniform in time and in space,  then  \underline{\bf C1(c)} holds. In particular 
 \underline{\bf C1(a)} + \underline{\bf C3}+ \underline{\bf C5(g)} $\Longrightarrow$ \underline{\bf C1(c)}.
\begin{remark}
The assumption for  the strong 1-completeness is purely technical, and so are conditions on $\overline H^p$, which validates differentiating the stochastic flow of the gradient SDE with respect to the initial data. 
Gradient SDEs are  defined and studied in \S\ref{section-gradient}.
\end{remark}
Let $V$ be a positive bounded H\"{o}lder continuous function.
Set \begin{equation}\label{NV}
\begin{aligned} N_{t} &=\f  4{t^2}\int_{\f t 2}^t \<d\{x_s\}, W_s(v_1)\>\int_0^{\f t2} \<d\{x_s\}, W_s(v_2)\>, \\
Q_t&=\f 2 t  \int_0^{t/2} \<d\{ x_s\}, W_s^{(2)}(v_1,v_2)\>,\\
\V_{a,t}&=(V(x_a)-V(x_0)) e^{-\int_{a}^t (V(x_s) -V(x_0))ds},  \quad \V_t:=e^{-\int_{0}^t V(x_s) ds}.\end{aligned}
\end{equation}
For a general manifold  we have a Hessian formula in terms of the above terms, leading to estimates on $\nabla^2 \log p_t$.

Below we will present  the version of the second order Feynman-Kac formula on  manifolds with a pole.  Let $(x_t)$ denote the semi-classical bridge $\tilde x_t$ with initial value $x_0$.  We define the processes just $\tilde W_t$, $\tilde W_t^{(2)}$, $\tilde N_t$, $\tilde M_t$, and $ \tilde V_{a,t}$ as in
 (\ref{Wt-pre}) and (\ref{doubly-damped}) and (\ref{NV}) with $x_t$ replaced by $\tilde x_t$.

 \begin{remark}
If $(x_t)$ is a Brownian motion with a drift, solving an SDE, then $W_t$ is the conditional expectation of the solution to the linearised SDE, the derivative flow.  The linearised SDE will involve  the  derivative of the driving vector fields which would have seemed to be, at first glance, 
 the appropriate tangent processes
along the semi-classical bridge.  
 Our $\tilde W_t$ process along $(\tilde x_t)$ will \underline {not} be deduced from  the derivative flow of the corresponding SDE for the semi-classical bridge, in particular we do \underline {not} differentiate $\nabla \log k_t$, nor do we differentiate the distance function.  The same remark applies to 
$\tilde W_t^{(2)}$. \end{remark}

\noindent
 {\bf Theorem \ref{2nd-order-kernel}.} 
Assume \underline{\bf C4.}  Let  $T>0$ and
$ \beta_T^h=e^{\int_0^T \Phi^h(\tilde{x}_s)ds}
$. Then
 \begin{equation}\begin{aligned}
&   {e^{h(x_0)-h(y_0)+V(x_0)\,T} }\f {\nabla d p_T^{h,V}(v_1, v_2)}{k_T(x_0, y_0)}\\
&=  \E\left[\beta_T ( \tilde N_{T}+\tilde Q_T) \right]
+\int_0^T  
\E \left[ \beta_T \tilde \V_{T-r, T}\left( \tilde Q_{t-r}+\tilde N_{T-r}\right) \right] dr. \end{aligned}
\end{equation}

Estimates for the Hessians of $p_t$ follows easily from  Lemmas \ref{lemma8.4}, \ref{lemma8.5}, which states the following hold  under \underline{\bf C4(a)(b)}.
\begin{equation*}\begin{aligned}
\left|\tilde N_T\right |_{L^p(\Omega)}
&\le c(p,n)\left(a_1(\alpha, p,T)+e^{ |K| T}\f {d^{2}(x_0, y_0)}{T^{ 2}} +a_2(K,T)  \f 1 T\right),\\
 \left|\tilde Q_T\right|_{L^p(\Omega)}&\le c(p,n)\left(b_4(2p)\f {d(x_0, y_0) }{T}+ {b_4(2p)} \f 1 {\sqrt T} +{b_3(2)} \f 1{\sqrt{t} } +A\right),
\end{aligned}
\end{equation*}
where the constants are explicitly given in terms of $J_{y_0}$, $\|\Theta^h\|$, $K$ and $\|\Rc\|$, see Definitions \ref{constants-1}
and \ref{constants-2}.

 %Analogous estimates for $p_t^{h,V}$ are also given.

%By Lemmas \ref{exponential-estimates} and  \ref{estimate-10} we will see that, for small time, the first term makes a  contribution of  order $\f 1 {\sqrt t}$ and the second term makes a contribution of  $\f 1t$.\\
%
%{\bf Theorem \ref{Hessian-estimates-thm1}.}
%Assume  \underline{\bf C2}+ \underline{\bf C1(c)}. Assuem $V(x_0)=0$. If $f$ is a non-negative function,  set $\HH_t(f)=\E \left[\f {f(x_t)} {P_t^h f(x_0)}\log \left(\f {f(x_t)} {P_t^h f(x_0)}\right)\right]$. Then
%$$\begin{aligned}
%& \left|{\nabla dP_t^{h,V} f}\right|_{x_0}
% \le{}C {P_t^h f(x_0)}\left(  \f 1{\sqrt t } + \f 1t \HH_t(f)  +  \f 1t  \right)\\
% &\qquad \quad+ c|f|_\infty \left[ \sqrt{|V|_\infty} \int_0^t  \f {P_{t-r}|V|(x_0)} {t-r}dr
% + |V|_\infty \sqrt t\sup_{s \le t}\E | W_{s}^{(2)}(v_1,v_2)|^2\right]. \end{aligned}$$ 
% The constants are explicitly given.
%
%
%\noindent
% {\bf  Corollary \ref{corollary-kernel-estimates-1}. } 
%Assume  \underline{\bf C2} +  \underline{\bf C1(c)}.  Let $H_t(x_0, y_0)=\sup_{y\in M} \left( \log  \frac{p^{h}_t(y,y_0)}{p^{h}_{2t}(x_0,y_0)}\right)$.
%Then, for explicit constants,
%$$\begin{aligned}
%\left|\f{\nabla d p_{2t}^{h} }{p_{2t}^h }\right|_{x_0}
%\le{}& \f 1 {\sqrt t } \left( c_1 H_t(x_0, y_0)
% +c_2+c_3A\right)+\f 1 t c_2(t,K)H_t.
%\end{aligned}
%$$
%If  $p_t^h$ satisfies a Harnack inequality, or equivalently if a good heat kernel estimate hold, we have the standard estimate of the form $\f 1t+ \f {d^2}{t^2}$ with explicit constant. \\ 
% 

\section{Gradient SDEs:  integrability and strong 1-completeness}
\label{section-gradient}

If $i: M\to \R^m$ is an isometric embedding we define $X(x)(e)$ to be the gradient of the function $\<i(x), e\>$ where $e\in \R^m$ and 
$Y(x)$  the adjoint of $X(x): \R^m\to T_xM$. If $\{e_i\}$ is an orthonormal basis of $\R^m$ we define $X_i(x)=X(x)(e_i)$ then the
following is called a gradient Brownian system,
 \begin{equation}
\label{gradient}dx_t=\sum_{i=1}^m X_i(x_t)\circ dB_t^i +\nabla h(x_t)dt,
\end{equation}
In this case,  $\nabla_v X^i=A_x \left(v,Y(x)e_i\right)$, 
$Y(x):\R^m\to \nu_x$ is the orthogonal projection to the normal bundle,  $A_x$ is the shape operator given by the relation $\left\<A_x(v_1,w), v_2\right\>=\left\<\alpha_x(v_1,v_2), w\right\> $
and $\alpha_x: T_xM\times T_xM \to \nu_x$ is the second fundamental form of the embedding. If $|\alpha_x(v,\cdot)|_{HS}^2$ denotes the Hilbert-Schmidt norm of $\alpha_x(v,\cdot)$, i.e. $\sum_{j=1}^m |\alpha_x(v,f_j)|_{\nu_x}^2$ for $\{f_i\}$ an o.n.b. of $T_xM$.
Then
$$\sum_1^m \left \<\nabla _vX^i, v\right\>^2=|\alpha_x(v,v)|_{\nu_x}^2, \quad 
\sum_1^m |\nabla_v X^i|^2=\sum_{j=1}^m |\alpha_x(v,f_j)|_{\nu_x}^2.$$ 
The map $H_p: TM\times TM\to \R$, defined in (\ref{hp}), is given by  geometric quantities:
$$H_p(v,v)=-\Ric(v,v)+2 \Hess(h)(v,v)+|\alpha_x(v,\cdot)|_{HS}^2+(p-2) |\alpha_x(v,v)|^2.$$
 In particular for a co-dimension one sub-manifold of an Euclidean space, $H_1=-\Ric+2 \Hess(h)$.  
%if $\rho^h$ is bounded from below), D. Bakry showed that the $h$-gradient Brownian system is complete. 
%If furthermore $\int_0^\infty \E e^{-\int_0^t 
% \rho^h(x_s) ds}dt$ is finite, it was shown in \cite{Li-Myers} that the weighted measure $e^{2h}$ is finite, so the gradient system has a finite invariant measure,  and the fundamental group $\pi_1(M)$ is finite.
 Let  $x_0$ be a point in the manifold,  $C, K$ are constants, and $r$ denotes the distance function from a given point, e.g. $y_0$.
Recall $\overline H_1$ is the supremum of the bilinear form $H_1$, c.f.  (\ref{h1}).  

%The following extends \cite[Corollary 8.3 ]{Li-flow}:

\begin{lemma}\label{lemma:strong-1}
Assume the dimension of $M$ is greater than $1$.
\begin{enumerate}
\item  Under one of the following conditions,\begin{enumerate}
\item   $\underline \Ric\ge -C(1+r^2)$ and $dr(\nabla h)\le C(1+r)$;
 \item $\rho^h\ge -C(1+r)$, \end{enumerate}
  the $h$-Brownian motion $(x_t)$ is complete and for any $t>0$,
$$ \sup_{s\le t}\E[ d^p(x_s, y_0)]\le  c_1(p)[d^p(x_0, y_0)+t]e^{c_2(p)t}.$$

 If furthermore $\underline H_1\le c(1+\ln r)$, then the gradient Brownian system (\ref{gradient}) is strongly 1-complete.

\item  Suppose that $\rho^h\ge -K $.  There exists a number $\alpha_1>0$ such that for any compact set $D$,
for any $t>0$ and for any $\theta$ satisfying $\theta t< \alpha_1$, $$\sup_{y_0\in D}\sup_{s\le t} \E \left(e^{\theta d^2(x_s, y_0)} \right) <\infty.$$
Also, for any $p>0$, $ \E [\sup_{s\le t} d(x_s, y_0)^p ]\le c_1(p) e^{c_2(p)t} t^{\f p2}$ for some constants $c_1(p), c_2(p)$.
 If furthermore $\overline H_1\le C(1+ d^2(\cdot, y_0))$,  the gradient SDE  (\ref{gradient}) is strongly 1-complete.
\end{enumerate}
\end{lemma}
 \begin{proof}
 Fix $y_0\in M$,  we write $r_t=d(x_t, y_0)$ and $r(x)=d(x, y_0)$.
For $p\ge 2$, we apply It\^o's formula to $r$. On $\{t<\zeta(x_0)\}$, where $\zeta(x_0)$ is the life time of the $h$-Brownian motion from $x_0$,
the following holds,
 \begin{eqnarray*}
 r_t^p&=&d^p(x_0, y_0)+ p\sum_{i=1}^m\int_0^t r_s^{p-1}\<\nabla r, X_i\>_{x_s} dB_s^i+\f 1 2 {p(p-1)}\int_0^t r_s^{p-2} ds \\
&&+\f 12 p\int_0^t r_s^{p-1}\;
(\underline \Delta  r+2 L_{\nabla h} r) (x_s) ds -L_t^{\cut} \end{eqnarray*}
where $L_t^{\cut}$ is a non-negative term, vanishing off the cut locus,  $\underline {\Delta }r$ is the Laplacian of the distance function off the cut locus of $y_0$ and vanishes on the cut locus. 
See \cite{Cranston-kendall-March}, especially for $h=0$.
This can be obtained also by taking a smooth approximation
$r_\epsilon$ of $r$ and applying to them It\^o's formula, and the following distributional inequality  \cite{Yau76}. For positive test function $f$,
$$\int_M \Delta r f\le \int_{M\setminus\{\cut(y_0)\}} r \Delta f,$$
and so the measure $\Delta r(1_{\{\cut(y_0)\}})$ is non-positive and $L_t^{\cut}\ge 0$.   Note that $\Delta r$ is of order $\f 1 r$
neat zero, and $r^{p-1}\Delta r$ vanishes for $p>2$.

Let us take $p\ge 1$. For part (1a), we apply the standard Laplacian comparison theorem to $ \Delta r$.
On the whole manifold where $r(x)$ is smooth, $\Delta r(x)$ is less or equal to  $\sqrt{ {(n-1)}K}\cot\left(r\sqrt {\f K{n-1}}\right) $ if $\Ric\ge K$ and $K$ is positive,
$\f {n-1} r$ for $K=0$ and  $\sqrt{ {-K(n-1)}}\coth\left(r\sqrt {\f {-K}{(n-1}}\right) $ if $K<0$.  Taking the infimum of the lower bound of the Ricci curvature over an exhausting sequence of relatively compact sets, noting that $\Delta r$ is of the order $\f 1 r$ near $r=0$ and
 otherwise it grows at the order of the square root of the infimum of the Ricci curvature, we see that there 
 exists a constant $c$ such that
$$r^{p-1}(\underline \Delta r+2 L_{\nabla h}r)\le c+cr^p.$$
We do not need to worry the exhausting sets where the lower bound of the Ricci curvature is non-negative.
The conclusion of part (1a) follows from a localising procedure which removes the local martingale part in the formula for $r_t^p$,  Young's inequality to bound $r_t^{p-2}\le c(p)r_t^p+c(p)$, Gronwall's inequality, and Fatou's lemma to conclude finiteness of the moments of the radial process from any initial point: for some constants $c_i(p)$, 
\begin{equation}\label{distance-p}
 \sup_{s\le t}\E[ d^p(x_s, y_0)]\le  c_1(p)(d^p(x_0, y_0)+t)e^{c_2(p)t}.
\end{equation}
In particular the SDE is complete.

For part (1b) we apply Bochner's formula to  $r$. Denote $\partial_r$ covariant differentiation w.r.t. $\nabla r$.
Since $|\nabla r|=1$,
$$\partial_r (\Delta r)=-|\Hess r|^2-\Ric(\nabla r, \nabla r).$$
Using the identity $\partial_r \<\nabla h, \nabla r\> =2 \Hess(h)(\nabla r, \nabla r)$,  we see that
$$\partial_r (\Delta^ h r)=-|\Hess r|^2+(-\Ric+2 \Hess(h) ) (\nabla r, \nabla r)\le -\rho^h.$$
If $\rho^h\ge K$, along a geodesic segment $\gamma$ from $y_0$,  $\Delta^h r\le -Kr -\Delta^h r(\gamma(0))$.
From the assumption that $\rho^h\ge - C(1+r)$, we have again $r^{p-1}\Delta^h r\le c+cr^p$.
%We apply to the original equation a sequencing of stopping times to neutralise the local martingale term, 
%followed by Gronwall's inequality and  Fatou's lemma etc to see that
%$$ \sup_{s\le t}\E[ d^p(x_s, y_0)]\le (d^p(x_0, y_0)+c_1(p)t)e^{c_2(p)t},$$
The same argument as before leads to (\ref{distance-p}), concluding also non-explosion.

By Theorem 8.5 in \cite{Li-flow} the gradient flow is strongly 1-complete if 
\begin{equation}
\label{strong-gradient}
\sup_{x\in K}\E\left(e^{c \int_0^t\overline H_1(F_s(x))ds}\right)<\infty,
\end{equation}
for $c$ positive.  Since $r(x_t)$ has uniform $p$th -moments for $x_0$ in a compact
subset,
$$\E\left(e^{c \int_0^t\overline H_1(F_s(x))ds}\right)\le \E \f 1 t\int_0^t e^{ct \overline H_1(F_s(x))} ds<\infty.$$
The finiteness follows from $\overline H_1(F_s(x))\le C+C\ln  r(F_s(x))$, proving the strong 1-completeness.

Finally we assume that $\rho^h\ge K$. Let $2a=\inf_{x\in D} \inj (x)$ where $\inj(x)$ is the injetivity radius at $x$ and $D$ a compact set containing $x_0$.  
Away from $0$, we have seen in part (1b) that 
$\partial_r (\Delta^ h r)\le K$, and so along a geodesic $\gamma$  from $x_0$,
$\Delta^h r (\gamma(s)) \le -Kr(\gamma(s))-\Delta^h r(\gamma(\epsilon))$. 
On a set close to $x_0$, $\Delta^h r=\f {n-1} r+dr(\nabla \log (Je^{2h}) )$ where $J$ is the Jacobian determinant of the exponential map $\exp_{y_0}$.
Since $\sup_{x\in D}|\nabla \log (Je^{2h})|_x$ is bounded, there exists a constant $c$ and a one dimensional Brownian motion $\beta_t$ which may depend on $x_0$, such that for all $x_0\in D$,
$$r_t \le r(x_0)+\beta_t+ \int_0^t  \left(\f {n-1}{2r_s} +c \right)ds -L_t,$$
which we compare  with the following equation 
$$d R_t =d \beta_t+ \f {n-1}{2R_t} dt+cdt, \quad R_0=r(x_0).$$
The process is the radial process of the SDE $dz_t=dB_t+c\f {z_t}{|z_t|}dt$ on $\R^n$ with $z_0$ a point such that $|z_0|=r(x_0)$. By comparing with
the Bessel process we see that the paths of $|z_t|$ do not hit zero with probability $1$.
We make a Girsanov transform to remove the drift. Set
$$M_t=e^{ -c\int_0^t\left  \<dB_s, \f {B_s}{|B_s|}\right\> -\f {c^2}{2}t},$$
Let $\alpha_1>0$ be a number such that $\E [e^{2\alpha_1 |B_t|^2}]$ is finite. Since $M_t$ has finite second moment,  $\E(M_t)^2\le e^{c^2t}$, we see that
$$\E [e^{\alpha_1 R_t^2}]=\E[ e^{\alpha_1 |z_t|^2}] =\E \left( e^{\alpha_1 |z_0+B_t|^2}\; e^{ -c\int_0^t\left  \<dB_s, \f {B_s}{|B_s|}\right\> -\f {c^2}{2}t}\right)$$
is finite. By the comparison theorem, $\sup_{x\in D}\E [e^{\alpha_1 r_t^2}]$ is finite for any compact set $D$.
Similarly,  
$$\E \left[\sup_{s\le t} r_s^p\right] =\E\left[ \sup_{s\le t} R_s^p \right]\le \E \left[\sup_{s\le t} |z_0+B_s|^p\right] e^{c(p)t}.$$
Consequently, if we take $y_0=x_0$ and $z_0=0$, and apply Burkholder-Davies-Gundy inequality to obtain
that $$ \E [\sup_{s\le t} d(x_s, x_0)^p ]\le c(p) e^{c(p)t} t^{\f p2}.$$

 Since $H_1\le C(1+r^q)$, $\E e^{ct \overline H_1(F_s(x))} <\infty$ for sufficiently small $t$, say $t<t_0$. The strong 1-completeness
 follows from (\ref{strong-gradient}), first on a small time interval $[0, t_0]$, and then extended to all finite times by the strong Markov property.
 \end{proof}

For better estimates on $\Delta^h r$ we make use of $-|\Hess r|^2$, observing that
$ -|\Hess r|^2\le -\f{(\Delta r)^2 }{n-1}$, and compare $\Delta^hr$ with the explicit solution of the ODE $m'=-\f {m^2}{n-1}-K$ with $m(0)=\infty$.
This leads to the Laplacian comparison theorem in  \cite{Wei-Wylie}, generalising the standard Ricci comparison theorem. Let $r$ denote the distance function from a given point
$y_0$.
\begin{theorem}
 \cite{Wei-Wylie} Suppose that $\rho^h\ge K$. 
 \begin{enumerate}
 \item If $\<\nabla h, \nabla r\> \le 2a$ where $a$ is a positive number. Then along a minimal geodesic segment from $y_0$, 
 $$\Delta^h r\le \left\{ \begin{array}{ll}
a+ \sqrt{ {(n-1)}K}\cot\left(r\sqrt {\f K{n-1}} \right),  \quad & K>0 \;\; \& \; \; r \sqrt{\f K{n-1}} \le \f \pi 2,\\
a+ \f {n-1} r, &K=0,\\
 a+ \sqrt{ {(n-1)}{(-K)}}\coth\left(r\sqrt {\f {-K}{n-1}} \right),  \quad & K<0.
  \end{array}
 \right.$$
 \item  if   $|h|\le k/2$, then for an explicit number $r_0>0$,
$$\Delta^h r\le \left\{ \begin{array}{ll} \sqrt{ {(n+4k-1)}K}\cot\left(r\sqrt {\f K{n+4k-1}} \right),  \quad & K>0,  \;\; \& \; \; r\le r_0\\
 \f {n+4k-1} r, &K=0\\
 \sqrt{ {(n+4k-1)}{(-K)}}\coth\left(r\sqrt {\f {-K}{n+4k-1}} \right),  \quad & K<0.
  \end{array}
 \right.$$
 \end{enumerate}
\end{theorem} 

This, together with the earlier lemmas and its proof lead to our final theorem on strong 1-completeness. 
\begin{theorem}\label{theorem-strong}
 Under one of the conditions from \underline {\bf C5(a)} to \underline {\bf C5(d)}, the $h$-Brownian motion is complete
 and for any $t>0$ $$ \sup_{s\le t}\E[ d^p(x_s, y_0)]\le  c_1(p)[d^p(x_0, y_0)+t]e^{c_2(p)t}.$$
 \begin{enumerate}
 \item 
The gradient SDE is strongly 1-complete if furthermore \underline {\bf C5(e)} holds, in which case $\E[|T_xF_t|]$ is also finite.
\item Strong 1-completenss also holds under  \underline {\bf C5(d)} +\underline {\bf C5(f)} in which case there exists a number $\alpha_1>0$ s.t. for any compact set $D$,
for any $t>0$ and for any $\theta$ satisfying $\theta t< \alpha_1$, $$\sup_{s\le t}\sup_{x_0\in D} \E \left(e^{\theta d^2(x_s, x_0)} \right) <\infty.$$
 \end{enumerate}
In both cases $|T_xF_t|$ has finite expectation. 
\end{theorem}
\begin{proof}
The explosion problem is discussed earlier. We use Proposition 8.5 in \cite{Li-flow} which states that 
 $$\sup_{x\in D} \E \exp\left(\f 12 \int_0^t \overline H_1(x_s)ds\right)<\infty$$
 for all compact subset $D$ implies that $\E|v_t|$ is finite and  the gradient SDE is strongly 1-complete.
 If $\overline H_1\le c(1+\ln r)$ and $\sup_{x \in D} \E [r(x_t)]^p$ is finite for all $p$, the inequality indeed holds.
 Similarly it holds if $\overline H_1\le c(1+r^{q-1})$ and $\sup_{x\in D} \E e^{\delta r^2(x_t)}<\infty$ for some $\delta>0$.
\end{proof}
We will eventually assume that $\rho^h$ is bounded from below, in which case $|W_t|$ is bounded. Otherwise,
if for example $\rho^h\ge -C(1+\ln r)$ where $C\ge 0$ and $\partial_r h\le c(1+r)$, then $r(x_t)$ has moments of all orders.
From $|W_t|^2\le e^{-\int_0^t \rho^h(x_s) ds}$,
we see that for all $p\ge 1$,
$$\sup_{x_0\in D}\sup_{s\le t}\E |W_s|^p\le  \sup_{x_0\in D}  \f 1t \int_0^t \E e^{-\f 12 pt\rho^h(x_u) } du,$$
the right hand side is finite for any compact set $D$ and $t>0$.  

\section{Doubly damped stochastic parallel translations} 
\label{section-doubly-damped}
If  $(x_t, t\le T)$ is a continuous stochastic process on $M$,  $\F_t^{x_\cdot}$ its  filtration and
 $\F=\vee_t \F_t^{x_\cdot}$, augmented as usual,
 a $TM$-valued stochastic process $(V_t, t\le T)$ is said to be along  $(x_t)$ if the projection of $V_t$  to $M$ is $x_t$. 

\begin{definition}
\cite[Def 3.3.2]{Elworthy-LeJan-Li-book}
%Let $\tau$ be a $\F_\cdot^{x_\cdot}$ predictable process, by which we mean there exists an increasing  sequence of stopping times
%converging to it.
Given a stochastic parallel translation $\parals_t$ along the stochastic process $(x_t)$,  we say that 
a stochastic process $(\bar V_t)$ with values in $TM$ is a local conditional expectation of $V$
with respect to the $\sigma$-algebra $\F^{x_\cdot}$
if there exists a $\F^{x_\cdot}$ measurable real valued process $(\alpha_t, t\le T)$
and a family of $\F_{t-}^{x_\cdot}$-stopping times $\tau_n$ increasing to $T$ such that
$\parals_{t\wedge \tau_n}^{-1}V_{t\wedge \tau_n} \alpha_{t\wedge \tau_n}$
has finite expectation and 
$$\E \{ \parals_{t\wedge \tau_n}^{-1}V_{t\wedge \tau_n} \alpha_{t\wedge \tau_n}|\F^{x_\cdot} \}= \parals_{t\wedge \tau_n}^{-1}\bar V_{t\wedge \tau_n} \alpha_{t\wedge \tau_n}.$$
\end{definition}

 By Corollarys 3.3.4  in \cite{Elworthy-LeJan-Li-book}, if $|V_t|\in L^1$ then the local conditional expectation is just the conditional expectation.

In this section $F_t(x)$ denotes the solution to the gradient SDE (\ref{gradient}), and
  $TF_t(v)$ its derivative flow. Let $x_t=F_t(x_0)$ and define the $\sigma$-algebra $\F_t^{x_0}=\sigma\{x_s: s\le t\}$,
  the $\sigma$-algebra generated by the solution flow with initial value  $x_0$, of the gradient SDE,
 up to  $t$. 
 \subsection{Doubly damped  equation and a primitive formula}
 Let $j$ be a parallel field, with  $j(0)=v_2$,  along the normalised geodesic $\gamma$ 
with initial condition $x_0$ and initial velocity $\dot \gamma(0)=v_1$. 
If $W_t(x_0)$ is the stochastic damped parallel translation along the paths of $\{x_t\}$, we differentiate it w.r.t. the initial value of the
path and take conditional expectations as following:
$$V_t:=\E \left\{\f {D}{ds}|_{s=0} W_t(j(s))\; \Big |\; \F_t^{x_0}\right\}.$$

\begin{lemma}\label{second-diff-lemma-1}
Suppose that $\Ric-2 \Hess (h)$ is bounded from below.
\begin{enumerate}
\item[(a)] the gradient SDE  is strongly 1-complete.
\item [(b)] for every $s$, $\E |T_{\gamma(s)}F_t|$ and $\E |\nabla _{TF_t(\gamma(s))}W_t|_{\gamma(s)}$ are finite.
\item [(c)] $s\mapsto \E\{ \f D {ds}W_t(\gamma(s))\big| \; \F_t^{\gamma(s)}\}$ is continuous in $L^1(\Omega)$;
\end{enumerate}
Then  for all $f\in BC^2$,
\begin{equation*}
\Hess (P^h_tf)(v_2, v_1) =\E \left[\nabla df(W_t(v_2), W_t(v_1))\right]+\E \left[df \left(V_t\right)\right].
\end{equation*}
\end{lemma}
\begin{proof}
%Let $j$ be a parallel field with $ j(0)=v_2$ along the normal geodesic $\gamma(s)$ determined by  $ \gamma(0)=x_0$ and $\dot \gamma(0)=v_1$. 
Since $dP_t^hf=e^{\f 12 \Delta^h} (df)$,
$$dP^h_tf (j(s))=\E df(W_t(j(s))).$$
We observe that $W_t(j(s))\in T_{F_t(\gamma(s))}M$ is a function of $F_t(\gamma(s))$. 
By the strong 1-completeness we see
that both $F_t(\gamma(s))$ and $W_t(j(s))$ are differentiable in $s$, and the conditions of the theorem ensure that we may change the order of taking expectation with differentiation with respect to $s$.
In fact,
$$\begin{aligned}&\nabla d(P^h_tf)(v_1, v_2)
=\lim_{\epsilon \to 0} \f 1 \epsilon\int_0^\epsilon \E  \f {d} {ds} \left( df(W_t(j(s))\right) ds\\
=&\lim_{\epsilon \to 0} \f 1 \epsilon
 \int_0^\epsilon \left(    \E\left [\nabla df (TF_t(\dot \gamma(s)), W_t(j(s))) \right]-\E\left [df\left(\f D {ds}W_t(j(s))\right)\right]\right)\\
 =&\lim_{\epsilon \to 0} \f 1 \epsilon
 \int_0^\epsilon   \E\left [\nabla df (W_t(\dot \gamma(s)), W_t(j(s))) \right]ds\\
 &-\lim_{\epsilon \to 0} \f 1 \epsilon
 \int_0^\epsilon \E\left [df\left( \E\left\{\f D {ds}W_t(j(s))\;\Big| \F_t^{\gamma(s)} \right\}   \right) \right]ds.
%&=\E\left[ (\nabla df)\left(TF_t(v_2), W_t(v_1)\right)\right]+ \E\left[ df\left(\f D{ds}|_{s=0}W_t(j(s))\right)\right].
\end{aligned}$$
We have used the fact that $|TF_t|$ in integrable so $\E\{TF_t(v_2) |\F_s^{x_0}\}=W_t(v_2)$,
and that  $|\f D {ds}W_t(j(s))|$ is integrable. Since the SDE is strongly 1-complete, 
$$\lim_{s\to 0} \nabla df (W_t(\dot \gamma(s)), W_t(j(s))) = \nabla df (W_t(v_2), W_t(v_1)) $$
and $\f D {ds}W_t(j(s))$ converges to $\nabla_{TF_t(v_2)}(v_1)$ as $s\to 0$. Since $\rho^h$ is bounded from below,
$|W_t|_x$ is uniformly bounded, and the first limit converges in $L^2$.
The  $L^1$ continuity of $s\mapsto \f D {ds}W_t(j(s))$ and the fact that $f\in BC^2$ lead to the required conclusion.
 \end{proof}

\begin{lemma}\label{lemma-doubly-damped}
Suppose that the gradient SDE  is strongly 1-complete.
Given $W_t(v_1)$, let $W_t^{(2)}(v_1, v_2)$ (abbreviated as $W_t^{(2)}$) denote the solution to the following covariant differential equation
\begin{equation}\label{Wt2}
\begin{aligned}
 {D} W_t^{(2)}(v_1, v_2)
=&\left(- \f 12 {\Ric}^{\sharp}   +(\nabla^2 h)^\sharp\right)  \left(W_t^{(2)} (v_1, v_2)\right)dt +\f 12 
\Theta^h(W_t(v_2))( W_t(v_1)) dt\\
&+\Rc( d\{x_t\}, W_t(v_2)  )W_t(v_1),\\
W_0^{(2)}(v_1, v_2)=&v_2,
\end{aligned}
\end{equation}
where $\Rc$ is the curvature tensor, $\{x_t\}$ denotes the martingale part of $x_t$.
Then $W_t^{(2)}(v_1, v_2)$ is the local conditional expectation of $\f {D}{ds}|_{s=0} W_t(j(s))$.
If the latter is integrable, then
$$W_t^{(2)}(v_1, v_2)= \E \left\{\f {D}{ds}|_{s=0} W_t(j(s))\; \Big |\; \F_t^{x_0}\right\}.$$
 \end{lemma}

\begin{proof}
Let $\gamma$ be the normal geodesic defined above. We differentiate the vector field $W_t$ along the surface $F_t(\gamma(s), \omega)$
in the $s$ direction followed by stochastic covariant differential in $t$ direction, using strong 1-completeness.   Thus,
 \begin{equs}
& D \left(\f {D} {ds} |_{s=0}W_t(j(s))\right)
=\f {D} {ds} |_{s=0} \left( {D}  W_t(j(s))\right)+\Rc\left( X(x_t)\circ dB_t, TF_t(v_2)  \right)W_t(v_1)\\
=&\f {D} {ds} |_{s=0} \left( ( -\f 12 {\Ric}^{\sharp} +(\nabla^2 h)^\sharp) (W_t(j(s)) dt\right)\\&+\Rc( X(x_t)\circ dB_t+\nabla h(x_t)dt, TF_t(v_2)  )W_t(v_1).
\end{equs}
 The curvature term results from exchanging the derivative of $W_t(j(s))$ in the direction of  the stochastic differential ( because of $W_t$ is also a function of $ (F_t(\gamma(s)))$) in $t$ and 
the derivative of $F_t(j(s))$ in $s$. We also used that  $\f {D}{ds} j(s)=0$ and that the differential of $W_t(F_t(\gamma(s)))$ in $t$  satisfies the stochastic damped parallel translation equation.
Let us compute the term involving $\nabla^2h$:
$$\f {D} {ds} |_{s=0}(\nabla^2 h)^\sharp (W_t(j(s))
=\nabla_{TF_t(v_1)} (\nabla^2h)^\sharp (W_t(v))+(\nabla^2h)^\sharp  \left(\f {D} {ds} |_{s=0}W_t(j(s))\right).$$
The first term can also be written as 
$ \nabla^3h (\nabla_{TF_t(v_1)}, W_t(v), \cdot)$, and consequently,
%If $U, V$ be vector field with $\nabla U$ vanishes at $x_0$ and with $V(x_0)=v_1$,  then
%$$\nabla_{TF_t(v_1)} \<(\nabla^2h)^\sharp (W_t(V)), U\>=\left\< \left(\nabla_{TF_t(v_1)} (\nabla^2h)^\sharp (W_t(v))\right), U \right\>
%+\left\< (\nabla^2h)^\sharp \left( (\nabla_{v_2} W_t)(v))\right), U \right\>
% $$
%and $\nabla_{TF_t(v_1)} \nabla^2 h(W_t(V), U\>=\nabla^3(TF_t(v_1), W_t(v_1), u) $
\begin{equs}
& {d}  \parals_t^{-1} \left(\f {D} {ds} |_{s=0}W_t(j(s))\right)\\
=& \parals_t^{-1} \left( - \f 12\nabla_{TF_t(v_2)}  {\Ric}^{\sharp}( W_t(v_1))
+\nabla^2(\nabla h)\left(TF_t(v_2), W_t(v_1)\right) \right) dt
\\&+\left(- \f 12\parals_t^{-1}  {\Ric}^{\sharp}+(\nabla^2 h)^\sharp\right) \left(\f {D} {ds} |_{s=0} \left(W_t(j(s))\right) \right)dt\\
&+\parals_t^{-1} \Rc( X(x_t)\circ dB_t, TF_t(v_2)  )W_t(v_1)+\parals_t^{-1} \Rc( \nabla h(x_t), TF_t(v_2)  )W_t(v_1)dt.
\end{equs}
The notation $\nabla^2(\nabla h)$ should be clear, note the function $\<\nabla^2( \nabla h) (v,u), w\>$ is symmetric in $(u,w)$.
We unravel the Stratonovich stochastic integral into one with respect to $d\{x_t\}$, the martingale part of $x_t$:
\begin{equs}
&\parals_t^{-1} \Rc( X(x_t)\circ dB_t, TF_t(v_2)  )W_t(v_1)\\
=&\parals_t^{-1} \Rc( d\{x_t\}, TF_t(v_2)  )W_t(v_1)
+\f 12 \trace (\nabla_{\cdot}\Rc) (\cdot, TF_t(v_2)  )W_t(v_1) dt,
\end{equs}
where we have used the following fact for the gradient system: at any point $x\in M$ either $\nabla X_i(x)$ vanishes or $X_i(x)$ vanishes.
We use the following formula,  see e.g. \cite{Elworthy-book}, for an o.n.b. $\{f_i\}$ of $T_xM$, and $v_i\in T_xM$,
$$\sum_{i=1}^n\left\<\nabla_{f_i} R( f_i, v_2)v_1, v_3\right\>
=(\nabla_{v_3} {\Ric})(v_1, v_2)-(\nabla_{v_1}{ \Ric})(v_3, v_2).
$$
We combine the three terms involving the covariant derivative of the Ricci curvature, and set:
$$\begin{aligned}
&\left\< \Theta(v_2) (v_1), v_3\right\>=  \left( \nabla_{v_3}  {\Ric}^{\sharp}\right) (v_1, v_2) -  \left( \nabla_{v_1}  {\Ric}^{\sharp}\right) (v_3, v_2)
 -   \left( \nabla_{v_2}  {\Ric}^{\sharp}\right) (v_1, v_3),\end{aligned}$$
 which is symmetric in $(v_1, v_2)$. Setting, 
$$\Theta^h(v_2) (v_1)=\f 12 
\Theta(v_2, v_1)+  \nabla^2 (\nabla h)( v_2,v_1)
+ \Rc( \nabla h, v_2)(v_1),$$
to see that
\begin{equs}
 &{d}  \parals_t^{-1} \left(\f {D} {ds} |_{s=0}W_t(j(s))\right)\\
=&  \f 12\parals_t^{-1} \Theta^h(TF_t(v_2)) (W_t(v_1))dt
+\left(- \f 12\parals_t^{-1}  {\Ric}^{\sharp}+\nabla^2 h \right) \left(\f {D} {ds} |_{s=0} \left(W_t(j(s))\right) \right)dt\\
&+\parals_t^{-1} \Rc( X(x_t)\circ dB_t, TF_t(v_2)  )W_t(v_1).
\end{equs}

With this preparation we condition the above stochastic equation with respect to $\F_t^{x_0}$ and use the fact that
$W_t(v_2)$ is the local conditional expectation of $TF_t(v_2)$ and apply Lemma 3.3.1 in \cite{Elworthy-LeJan-Li-book} to obtain:
%$\E\{TF_t(v_2)|\F_t^{x_0}\}=W_t(v_2)$. 
%
% \begin{equs}
%&{d} \parals_t^{-1}  W_t^{(2)}\\
%=&\left( - \f 12 \parals_t^{-1} {\Ric}^{\sharp}  + \parals_t^{-1}\nabla^2 h\right)  \left(W_t^{(2)} \right)dt\\
%&+\parals_t^{-1} \left( - \f 12 \left(\nabla_{W_t(v_2)}  {\Ric}^{\sharp} \right)( W_t(v_1))
%+ \nabla^2 (\nabla h)( W_t(v_2), W_t(v_1)) \right)dt\\
%&+\parals_t^{-1} \Rc( d\{x_t\}, W_t(v_2)  )W_t(v_1)
%+\f 12 \sum_{i=1}^m\parals_t^{-1}  (\nabla_{X_i}\Rc) (X_i(x_t), W_t(v_2)  )W_t(v_1) dt\\
%&+\parals_t^{-1} \Rc( \nabla h(x_t), W_t(v_2)  )W_t(v_1)dt.
%\end{equs}
 \begin{equs}
d\parals_t^{-1}  W_t^{(2)}
=&\left( - \f 12 \parals_t^{-1} {\Ric}^{\sharp}  + \parals_t^{-1}\nabla^2 h\right)  \left(W_t^{(2)} \right)dt
+\f 12\parals_t^{-1} \Theta^h (W_t(v_2)) (W_t(v_1))dt\\
&+\parals_t^{-1} \Rc( d\{x_t\}, W_t(v_2)  )W_t(v_1).
\end{equs}
We have used the fact that $\int_0^s \parals_t^{-1} \Rc( d\{x_t\}, \cdot )$ is adapted to the filtration of $\{x_\cdot\}$, see \cite[Theorem 3.1.2]{Elworthy-LeJan-Li-book}. The conclusion about the local expectation follows.
The rest follows from Corollarys 3.3.4  in \cite{Elworthy-LeJan-Li-book}.
\end{proof}

Let $p>1$ and $r$ denotes the Riemannian distance from a given point.  Observe that for any $t>0$and for any compact set $D$,
$\sup_{x\in D}\sup_{s\le t} \E(|T_xF_{s}| ^p)$ is finite if $\overline H_p\le c_1+c_2r^2$  where $c_2$, depending possibly on $t$, is sufficiently small. See (\ref{criterion}).
Below let $c>0$  be a constant and $\alpha_2>0$ be a sufficiently small constant.
\begin{lemma}\label{moments-10}
Suppose   $\rho^h\ge -K$, $\|\Theta^h\| \le ce^{C r}$, $\|\Rc\| \le ce^{C r}$, where $C\le \alpha_2$.
\begin{enumerate}
\item [(a)] then
$\sup_{x_0 \in D} \sup_{s\le t} \E \|W_t^{(2)}\|^2$ is finite for any compact subset $D$.
\item [(b)] If furthermore  $\int_0^t (\E [|T_xF_{s}|^{2p})^{\f 1p}ds<\infty$ for some $p>1$.
then $\E\left( |V_t|^2\right)<\infty$.
\end{enumerate}
\end{lemma}
\begin{proof}
Take unit vectors $v_1, v_2$ in $T_{x_0}M$. For part (1) we observe that
\begin{equs}
&{}d \left|W_t^{(2)}(v_1,v_2)\right|^2
=\left(-\Ric+2 \Hess h\right)(W_t^{(2)}(v_1,v_2), W_t^{(2)}(v_1,v_2))dt\\
&+\left<W_t^{(2)}(v_1,v_2),  \Theta^h\left(W_t(v_2)\right)  ( W_t(v_1)) \right\>dt\\
&
+2\left\<  \Rc( d\{x_t\}, W_t(v_2)  )W_t(v_1), W_t^{(2)}(v_1,v_2))\right\>
+\sum_{i=1}^n\left |  \Rc(u_te_i, W_t(v_2)  )W_t(v_1)\right|^2 dt.
\end{equs}
The penultimate term, which we denote by $(M_t)$, in the above equation is a local martingale.  
By taking suitable stopping times $\{\tau_k\}$, this term vanishes. Since $\rho^h\ge - K$,  $|W_t|^2\le e^{-Kt}$ for any $t\ge 0$, 
and we obtain the following estimates:
 \begin{equs}
& \E\left|W_{t\wedge \tau_k}^{(2)}(v_1,v_2)\right|^2\\
\le &
 |v_2|^2+\E \int_0^{t\wedge \tau_k}\left(\f 12 + \rho^h(x_s)\right)\left|W_{s}^{(2)}(v_1,v_2)\right|^2ds
\\ &+\E \int_0^{t\wedge \tau_k} \f 12\left|  \Theta^h(W_s(v_2))W_s(v_1)\right|^2ds
 +\E  \int_0^{t\wedge \tau_k}\|\Rc(\cdot, W_s(v_2))W_s(v_1)\|^2 ds\\
\le &
 |v_2|^2+\E \int_0^{t\wedge \tau_k}\left(\f 12 +K\right)\left|W_{s}^{(2)}(v_1,v_2)\right|^2ds\\
 &+
\E \int_0^{t\wedge \tau_k} \f 12  \|\Theta^h\|^2_{x_s} e^{2Ks} |v_1|^2|v_2|^2\;ds+\E  \int_0^{t\wedge \tau_k}
 \| \Rc\|^2_{x_s}  e^{2Ks} |v_1|^2|v_2|^2\; ds.
  \end{equs}
By Lemma \ref{lemma:strong-1}, $\sup_{s\le t}\E e^{\alpha_1 d^2(x_s, x_0)}$ is uniformly bounded in $x_0$.
Take $\alpha_2=\f 12 \alpha_1$ to see both $\E [\|\Theta^h\|_{x_s}^2  ]$ and $ \E[\| \Rc\|_{x_s}^2  ]$ are finite.
 We apply Gronwall's lemma  followed by taking $k\to \infty$ and using Fatou's lemma to see that
$ \E\sup_{ r\le t} \left|W_r^{(2)}(v_1,v_2)\right|^2
\le C(t) $ where $C(t)$ is a constant depending on $t$, and is locally uniform w.r.t. $x_0$.

For part (2) we use the following equation from the proof of the last lemma.
\begin{equs}
 {d}  \parals_t^{-1} V_t
= &\left(- \f 12\parals_t^{-1}  {\Ric}^{\sharp}+\nabla^2 h \right) \left(V_t \right)dt
+ \f 12\parals_t^{-1} \Theta^h( TF_t(v_2))(W_t(v_1))dt\\
&+\parals_t^{-1} \Rc( X(x_t)\circ dB_t, TF_t(v_2)  )W_t(v_1).
\end{equs}
We may proceed as before, but replacing $W_t(v_2)$ be $TF_t(v_2)$. We finally have to take care of the following terms
$\E[\|\Theta^h\|^2 |TF_t|^2]\le  \left(\E \|\Theta^h\|^{2q}\right)^{\f 1q} ( \E  |TF_t|^{2p})^{\f 1p}$.
If $2qC<\alpha_1$, $  \E \|\Theta^h\|^{2q}$ is finite. The term involving $\Rc$ can be treated similarly. 
 \end{proof}

The following proposition follows immediately from the two pervious lemmas. 
Denote by $c, \delta$ positive constants,  $K$ a constant, and $\alpha_2$ a sufficiently small constant.
\begin{proposition}\label{second-diff-formula}
Suppose that \begin{enumerate}
\item [(a)]  Suppose   $\rho^h\ge -K$, $\|\Theta^h\| \le ce^{C r}$, $\|\Rc\| \le ce^{C r}$, where $C\le \alpha_2$.
\item [(b)]   the gradient SDE  is strongly 1-complete and 
 $\int_0^t (\E [|T_xF_{s}|^{2p})^{\f 1p}ds<\infty$ for any $x$ and for some $p>1$.
\end{enumerate}
 Then  for all $f\in BC^2$,
\begin{equation}\label{second-diff-5}
\Hess (P^h_tf)(v_2, v_1) =\E \left[\nabla df(W_t(v_2), W_t(v_1))\right]+\E \left[df (W_t^{(2)}(v_1, v_2))\right].
\end{equation}
\end{proposition}
%\begin{proof}
%Conditions (a) and (b) imply that the conditions of  Lemma \ref{second-diff-lemma-1} hold.
%\end{proof}

An immediate consequence are the following primitive $L_\infty$ estimates:
\begin{corollary} Suppose the conclusions of Proposition \ref{second-diff-formula}.
Let $f\in BC^2$. Then
$$\left| \Hess (P^h_tf)(v_2, v_1)\right|
\le |\nabla df|_\infty  \E \left[e^{-\rho^h(x_s)ds} \right] +|df|_\infty  \E \left|W_t^{(2)}(v_1, v_2)\right|.$$
\end{corollary}

\subsection{Exponential integrability}
 Let us fix a $h$-Brownian motion, i.e. fix a probability space and its sample paths,
this could be the solution of a gradient SDE, or the projection of the canonical horizontal SDE, or any other pathwise representation.
We define $(W_t, W_t^{(2)})$ to be the solution to the following system of equations along the chosen sample paths 
$\{x_\cdot(\omega)\}$, so we have a triple of stochastic processes $( x_t, W_t, W_t^{(2)})$.
\begin{equation}\label{system-equs}
\begin{aligned}
%dx_t&=\sum_{i=1}^m X_i(x_t) \circ dB_t^i +\nabla h(x_t) dt,\\
\f{D} {dt }{W}_t=&-\f 12 {\Ric}_{x_t}^{\#}({W}_t)-(\nabla^2h)^\sharp_{x_t}(W_t), \\
 W_0=&v_1
\\
W_t \;d \left( W_t^{-1}W_t^{(2)}\right)
=&\f 12\Theta^h (W_t(v_1)(W_t(v_2))dt+ \Rc( \{dx_t\}, W_t(v_2)  )W_t(v_1), \\
 W_0^{(2)}=&v_2.
 \end{aligned}
\end{equation}

We next investigate the exponential integrability of $|W_t^{(2)}|$. In this, we assume that the curvature operator is bounded for 
the simplicity of the exposition. In the lemma below, $\|\Rc\|_{\infty}$ and $T$ are positive constants.

\begin{lemma}\label{exponential-estimates}
Suppose condition $\underline{\bf C2}$. Set $C_1(T,0)=1$,
$$C_1(T,K)= \sup_{0< s\le 3KT}\f 1{s}{(e^s-1)}, \quad \alpha_2(T, K, \|\Rc\|_\infty)=\f  1{49 n^2\|\Rc\|_\infty^2 C_1(T,K)}.$$
Then 
there exists a universal constant $c$ such that for $v_1, v_2\in ST_{x_0}M$,
and for any $\alpha \le  \alpha_2(T, K, \|\Rc\|_\infty)$, 
$$\begin{aligned} \E \exp\left(\alpha  \gamma |W_t^{(2)}(v_1, v_2)|^2\right)
& \le c e^{2\alpha \gamma} 
\sqrt{\E \exp\left(4t \;\gamma \alpha   \int_0^t e^{3 Ks} \|\Theta^h\|^2_{x_s} ds\right)}
\\
&\le c e^{ \f  {2 \gamma}{49 n^2\|\Rc\|_\infty^2}} \sqrt{ \E \exp\left(\f{4t\gamma}{49n^2\|\Rc\|_\infty^2C_1(t,K)}  \int_0^t e^{3 Ks} \|\Theta^h\|^2_{x_s} ds\right)}.\end{aligned}$$
\end{lemma}
 \begin{proof}
Take $v_1, v_2\in ST_xM$, and write $W_t^{(2)}
=W_t^{(2)} (v_1, v_2)$ for simplicity.  We first observe that
 $$W_t  d\left( W_t^{-1}W_t^{(2)}\right)=D W_t^{(2)}+
\left(- \f 12 {\Ric}^{\sharp} +(\nabla^2 h)^\sharp\right)   \left(W_t^{(2)} \right) dt.$$
By the definition of $W_t^{(2)}$, we see that
\begin{equs}
\begin{aligned}&W_t^{-1}W_t^{(2)}-v_2\\
=&\f 12 \int_0^t W_{s}^{-1}  (\Theta^h(W_s(v_1))(W_s(v_2)) )ds+\int_0^t W_{s}^{-1} \Big(\Rc( u_s dB_s, W_s(v_2)  )W_s(v_1)\Big).
\end{aligned}
\end{equs}
Denote by $A_t$ and $M_t$ respectively the first and the second term on the right hand side. If $\alpha$ is a positive number, we use elementary and Cauchy-Schwartz inequalities to obtain the following brutal estimate:
$$\begin{aligned}\E \exp\left(\alpha |W_t^{(2)}|^2\right)
&=\E\exp \left(\alpha ( v_2+ A_t+M_t)\right)^2\\
&\le  e^{2\alpha |v_2|^2} \E\exp \left(4\alpha  A_t^2+4\alpha M_t^2\right)
\le  e^{2\alpha |v_2|^2} \sqrt{\E e^{ 8 \alpha A_t^2}} \sqrt{\E e^{8 \alpha M_t^2}}.
\end{aligned}$$

Let us consider the $j$th component of the $\R^n$-valued local martingale $(M_t)$, where $\R^n$ is identified with $T_{x_0}M$,  
$$M_t^j:= \int_0^t \left\<W_{s}^{-1} \Big(\Rc( u_s dB_s, W_s(v_2)  )W_s(v_1)\Big), e_j\right\>.$$ 
This has quadratic variation 
$$f^j_t=\sum_{i=1}^n\int_0^t \left\<W_{s}^{-1} \Big(\Rc( u_s e_i, W_s(v_2)  )W_s(v_1), e_j\right \>^2 ds.$$
Since $\Rc$ is bounded, $\rho^h$ is bounded, the quadratic variations are uniformly bounded by a constant,
$$f_t^j\le n\|\Rc\|_\infty^2 \int_0^t e^{3Ks} ds \le  3tn\|\Rc\|_\infty^2  \f {e^{3Kt}-1}{3Kt}.$$
Thus $M_t^j$ is a time changed Brownian motion of the form $\beta^j_{f_t^j}$ where each $\beta^j$ is a one dimensional Brownian motion.
Furthermore for a universal constant $c(n)$ depending on $n$,
  $$e^{8\alpha |M_t|^2}=e^{8\alpha \sum_{j=1}^n (M_t^j)^2}\le c(n)  \sum_{j=1}^n e^{8\alpha n (M_t^j)^2}.$$
  In particular, 
  $\E e^{8\alpha |M_t|^2}$ is bounded by a universal constant if  $8\alpha  nf_t^j < \f 12$. Letting $C_1(t, K)= \sup_{0< t\le T} \f {e^{3Kt}-1}{3Kt}$,  
  it is sufficient to take $\alpha< \f  1{48 n^2\|\Rc\|_\infty^2 } \f 1 {C_1(t, K)}$.
Since $\alpha_2(t)=\f  1{49 n^2\|\Rc\|_\infty^2 C_1(t,K)}$,  for any $\alpha\le \alpha(t)$,
 $\E e^{8\alpha |M_t|^2}$  is bounded by a universal constant.

We consider the first term $A_t$, wich is easy to estimate. In fact,  $$|A_t|^2 \le  \f t 2\int_0^t \left| W_{s}^{-1}  (\Theta^h(W_s(v_1))(W_s(v_2)) )\right|^2 ds \le \f t 2\int_0^t e^{3 Ks} \|\Theta^h\|^2_{x_s} ds.$$
By part (2) in Lemma \ref{lemma:strong-1}, it is 
exponentially integrable to any given exponent if  $\|\Theta^h\|^2\le c+\delta  r^2$ where $\delta$ is sufficiently small (this is part of the assumption \underline{\bf C2}).
Putting everything together we see that there exists a universal constant $c$ such that for $\alpha\le\alpha_2(t)$, $\gamma\in (0,1)$,
$$\E \exp\left(\alpha\gamma  |W_t^{(2)}|^2\right) \le c e^{2\gamma\alpha}
\sqrt{ \exp\left( 4t\gamma \alpha \int_0^t e^{3 Ks}\E \|\Theta^h\|^2_{x_s} ds \right)}.$$
For the maximum value of $\alpha$,
$$\begin{aligned}
\E \left[e^{8\gamma \alpha_2 ( t ,K, \|\Rc\|_\infty) |A_t|^2}\right]
&\le \E \exp\left(4t \;\gamma \alpha_2( t ,K, \|\Rc\|_\infty)   \int_0^t e^{3 Ks} \|\Theta^h\|^2_{x_s} ds\right)\\
&\le \E \exp\left(\f{4t\gamma}{49n^2\|\Rc\|_\infty^2C_1(t,K)}  \int_0^t e^{3 Ks} \|\Theta^h\|^2_{x_s} ds\right).
\end{aligned}$$
Putting everything together we have, for any $\gamma \in (0,1)$, 
$$\E \exp\left(\alpha \gamma |W_t^{(2)}|^2\right) \le c e^{2\gamma\alpha} \sqrt{ \E \exp\left(\f{4t\gamma}{49n^2\|\Rc\|_\infty^2C_1(t,K)}  \int_0^t e^{3 Ks} \|\Theta^h\|^2_{x_s} ds\right)}.$$

 To conclude we use the equivalence of the two norms $ |W_t^{-1}W_t^{(2)}|$ and 
$ |W_t^{(2)}| $, which follows from the assumption that $\rho^h$ is bounded and hence  $e^{-Kt}  |W_t^{-1}W_t^{(2)}|\le |W_t^{(2)}| \le e^{Kt} |W_t^{-1}W_t^{(2)}| $. \end{proof}
 
For further estimates we use the following basic lemma \cite[Lemma~6.45]{Stroock2000}).
\begin{lemma}\label{Stroock's-Lemma}
Let $\phi$ be a non-negative measurable function on $\Omega$ s.t. $\E \phi\not =0$. If $\Psi$ is a measurable function on $\Omega$ such that $\phi \Psi$, $e^{-\Psi}$  and $e^\Psi$ are integrable then
\begin{equation}\label{Stroock}
-\mathbb{E}\left[\phi \log \f \phi {\E \phi} \right]+(\E \phi) \log \mathbb{E}(e^{-\Psi})
\le \mathbb{E}\left[\phi \Psi\right]
 \leq \mathbb{E}\left[\phi \log \f \phi {\E \phi} \right] +(\E \phi) \log \mathbb{E}(e^\Psi).
\end{equation}
%In particular if both $(\E \phi) \log \mathbb{E}(e^{-\Psi})\ge -c$ and $(\E \phi) \log \mathbb{E}(e^{\Psi})\le c$, then
%$$|\E \left[\phi \Psi\right] |\le c+\mathbb{E}\left[\phi \log \phi \right] -(\E \phi) \log (\E \phi).$$
\end{lemma}
\begin{proof}
The second inequality follows from Jensen's inequality,
\begin{equs}
\log \E [e^\Psi] &=\log \left[ \f {\E[  (e^\Psi \phi^{-1} \phi) ]}{\E \phi} \E \phi\right]
\ge \f { \E \left[ \phi \log (e^\Psi \phi^{-1} )\right] }{\E \phi}+ \log (\E \phi)\\
&= \f { \E \left[  \phi(\Psi -\log  \phi)  \right] }{\E \phi}+ \log (\E \phi).
\end{equs}
The first inequality follows from taking $-\Psi$ in place of $\Psi$.
\end{proof}

%Define $$\alpha_2(t)=\f  1{7n^2\|\Rc\|_\infty^2 C_1(t,K)}, \quad  C_1(t,K)= \sup_{0< s\le 3KT}\f 1{s}{(e^s-1)}.$$
%Let $W_t^{(2)}(v_1, v_2)$ denote the doubly damped stochastic transport from Lemma \ref{doubly-damped}.
%We use condition \underline{\bf C2'}: $|\rho^h| \le K$,  $\|\Rc_x\|\le \|\Rc\|_{\infty}$,  and $\|\Theta^h\|\le \|\Theta^h\|_\infty$.

\begin{lemma}\label{estimate-10}
Suppose  \underline{\bf C2}.
Let $f$ be a non-negative bounded and Borel measurable function normalised such that $P_t^hf(x_0)=1$. Let $v_1, v_2\in ST_{x_0}M$ and $(x_t)$ an $h$-Brownian motion.
 \begin{enumerate}
\item Let $c_1(n)=14\sqrt 2 n{ \sqrt {C_1(t/2,K)}} \|\Rc\|_\infty $. Then there exist  numbers $c_2$ and $c_3$, which depend on $n, K, \|\Rc\|_\infty, \Theta$ and are given explicitly in the proof, s.t.
$$\begin{aligned}
&\left| \E \left (  {f(x_t)}\cdot  \f 2 t \int_0^{\f t 2} \< u_s dB_s, W_s^{(2)}(v_1, v_2)\>\right)\right|\\
&\le  \f 1 { \sqrt t} \left( c_1\E\left [ (f\log f)(x_t)\right] + c_2+c_3A(t/2, K, \|\Rc\|_\infty, \|\Theta^h\|)
\right),\end{aligned}$$
where the constant $A(t/2, \Theta^h)$ is finite and is given by $$A(t/2, K, \|\Rc\|_\infty, \|\Theta^h\|)=\log \left(\E \exp\left(2 \alpha_2(t/2)  \int_0^{\f t 2} e^{3 Ks} \|\Theta^h\|^2_{x_s} ds\right)\right).$$
\item For a number $\delta_0>0$ set  $C_2(t,K)={(4+\delta_0)}\sup_{0<s\le t} \left(\f {e^{sK}-e^{\f {sK} 2}} {Kt}\right)$. There exists a number $c_0(\delta)$ depending on $\delta_0$ such that for $N_t$ defined in Definition~\ref{def-nt},
$$\left| \E \left (  {f(x_t)}\cdot  N_{t}\right)\right|
\le\f 1 {t} C_2(t,K)\E (f\log f)(x_t)+\f 1 {t} C_2(t,K) c(\delta_0).$$\end{enumerate}
\end{lemma}

\begin{proof}
For $v_1, v_2$ fixed we write $W_t^{(2)}$ for $W_t^{(2)}(v_1, v_2)$ to ease notation.
Let $\alpha_2(t)$ be the  number in Lemma \ref{exponential-estimates}:
$$ \alpha_2(T, K, \|\Rc\|_\infty)=\f  1{49 n^2\|\Rc\|_\infty^2 C_1(T,K)}.$$

For any number $\gamma\not =0$, we apply Lemma \ref{Stroock's-Lemma} to 
$\Psi= \gamma \sqrt{\f {\alpha_2(t)} {2t}}  \int_0^{\f t 2} \< u_s dB_s, W_s^{(2)}\>$ and to the mean $1$ random variable
$\phi=f(x_t) $ to obtain:
 $$\begin{aligned}
&\gamma\E \left (  {f(x_t)}\cdot   \f 2 t \int_0^{\f t 2} \< u_s dB_s, W_s^{(2)}\>\right)\\
&\le  \f 1 {\sqrt t } \f  {2 \sqrt 2}{\sqrt{ \alpha_2(\f t 2)} }\left[ \E[ ( f\log f )(x_t)]
 +\log \E\exp \left(  \gamma \sqrt{\f {\alpha_2(\f t 2)} {2t}} \left ( \int_0^{\f t 2} \< u_s dB_s, W_s^{(2)}\>\right)\right)\right].
 \end{aligned}$$
 Observe that $$c_1\equiv \f  {2 \sqrt 2}{\sqrt{ \alpha_2(\f t 2)} }\le 14\sqrt 2 n\|\Rc\|_\infty\sqrt{C_1(t/ 2, K)}.$$
 Next we use the following inequality, for a real valued continuous local martingale $M_t$ and real number $c$,
$\E e^{c M_t}\le \sqrt{ \E e^{2c^2\<M, M\>_t}}$. Let $\gamma$ be a number with $0<|\gamma|\le 1$. 
Note that $\alpha_2(t)\equiv \alpha_2(t, K, \|\Rc\|_\infty) $ is a monotone decreasing function.
By Lemma \ref{exponential-estimates} we have the following estimate 
 \begin{equation*}
 \begin{aligned}
& \log \E\exp \left( \gamma \sqrt{\f {\alpha_2(t/2)} {2t}} \left ( \int_0^{\f t 2} \< u_s dB_s, W_s^{(2)}\>\right)\right)\\
&\le \f  12 \log  \E e^{  \f {\gamma^2\alpha_2(t/2)} {t} \left ( \int_0^{\f t 2}  |W_s^{(2)}|^2 ds\right)}
\le  \f  12 \log \left( \f 2 t\int_0^{\f t 2} \E e^{  \f  12{\gamma^2\alpha_2(t/2)}  |W_s^{(2)}|^2}ds \right)\\
&\le\f  12 \log \left(\f 2 t\int_0^{\f t 2}  c(n) e^{\gamma^2\alpha_2(t/2)} 
\sqrt{\E \exp\left(2\gamma^2s \alpha_2(t/2)  \int_0^s e^{3 Kr} \|\Theta^h\|^2_{x_r} dr\right)}\right)\\
&\le \f 12 \log c(n)+\f 12 \gamma^2\alpha_2(t/2)+
 \f  14 \log \left( 
 \E \exp\left(2\gamma^2t \alpha_2(t/2)  \int_0^{t} e^{3 Kr} \|\Theta^h\|^2_{x_r} dr\right)\right).
\end{aligned}
 \end{equation*}
 Since $C_1(t,K)\ge 1$, $\alpha_2(s, K, \|\Rc\|_\infty^2)\le \f {1} {49n^2\|\Rc\|_\infty^2}$,
we have,
$$
 \begin{aligned}
& \f {2 \sqrt 2}  {\sqrt{t\alpha_2(t/2)}}  \log \E\exp \left( \gamma \sqrt{\f {\alpha_2(t/2)} {2t}} \left ( \int_0^{\f t 2} \< u_s dB_s, W_s^{(2)}\>\right)\right)\\
\le& \f 1{\sqrt t} \f{\sqrt 2 \log c(n)}{\sqrt{\alpha_2(t/2)}} +  \f 1{\sqrt t} \sqrt 2 \gamma^2\sqrt{\alpha_2(t/2)} 
\\&+   \f 1{\sqrt t}  \f { \sqrt 2}  {2\sqrt{\alpha_2(t/2)}}     \log \left( 
 \E \exp\left(2\gamma^2t \alpha_2(t/2)  \int_0^{t} e^{3 Kr} \|\Theta^h\|^2_{x_r} dr\right)\right)\\
\le &  \f 1{\sqrt t}  c_2+ c_3   \f 1{\sqrt t}  
  \log \left(\E \exp\left(2 \alpha_2(t/2)  \int_0^{\f t 2} e^{3 Ks} \|\Theta^h\|^2_{x_s} ds\right)\right).\end{aligned}$$
  where $$c_2=  7  n\log c(n)\|\Rc\|_\infty \sqrt{2C_1(t/2, K) }
+  \f {\sqrt 2\gamma^2}{7n\|\Rc\|_\infty }, \quad c_3=\f{7} {\sqrt 2} n\|\Rc\|_\infty \sqrt{C_1(t/2, K) }.$$
  To see that $A(t/2, \Theta^h)=\log \left(\E \exp\left(2 \alpha_2(t/2)  \int_0^{\f t 2} e^{3 Ks} \|\Theta^h\|^2_{x_s} ds\right)\right)$
  is finite we use Jensen's inequality to see that
  $$A(t/2, K, \|\Rc\|_\infty, \|\Theta^h\|)\le \log \left(\f 2 t  \int_0^{\f t 2}  \E \left[\exp \left(t \alpha_2(t/2) e^{3 Ks} \|\Theta^h\|^2_{x_s} \right) ds \right]   \right).$$
  
 Since $\rho^h$ is bounded above, by Lemma \ref{lemma:strong-1},   $\alpha r^2(x_t)$ is exponentially integrable 
for a small number $\alpha$. By assumption, $\|\Theta^h\|^2_{x_s} \le c+ \delta r^2(x_s)$ and by a sufficiently small $\delta$ we meant that
 $\delta 2 \alpha_2(t/2)  \int_0^{\f t 2} e^{3 Ks} <\alpha$.
 This completes the proof for part (1).
  For the second part recall that
$$ N_{t} =\f 4{t^2 }\int_{\f t 2}^t \<u_s dB_s, W_s(v_1)\>\int_0^{\f t2 } \<u_s dB_s, W_s(v_2)\>.$$
Set $c_1(t, K)=\f {e^{tK}-1} {Kt}$ and  $\alpha_3(t)=\f 1 {4+\delta_0} \left(\sup_{0<s\le t} c_1( s /2, K)e^{\f {Ks} 2}\right)^{-1}$. 
 By Lemma~\ref{Stroock's-Lemma}, for any number $\gamma$ with $0<|\gamma|\le 1$,
 $$\gamma \E \left (  {f(x_t)}\cdot  N_{t}\right)
 \le  \f 1 {t} \f 1 {\alpha_3(t)}\E (f\log f)(x_t)+\f 1 {t} \f 1 {\alpha_3(t)} \log\E e^{\gamma t \alpha_3(t)N_{t} }.$$
To estimate the exponential term we first apply Cauchy Schwartz's inequality to obtain
 \begin{equation*}
 \begin{aligned}
 \E e^{ \gamma t \alpha_3(t) N_{t} }&\le\sqrt{ \E e^{  \f{2 |\gamma| \alpha_3(t)}{t} \left(\int_{\f t 2}^t \<u_s dB_s, W_s(v_1)\>\right)^2}}
 \sqrt{\E e^{ \f{2|\gamma|  \alpha_3(t)}{t}\left(\int_0^{\f t2 } \<u_s dB_s, W_s(v_2)\>\right)^2}}.
 \end{aligned}
 \end{equation*}
Denote by $\Gamma_1$ and $\Gamma_2$ respectively the quadratic variations of the local martingales $\int_{\f t 2}^\cdot \<u_s dB_s, W_s(v_1)\>$ and $\int_0^{\cdot } \<u_s dB_s, W_s(v_2)\>$. Then  
$$\Gamma_1(t) \le \int_{\f t 2}^t e^{Ks} ds=\f {e^{Kt} -e^{\f {Kt} 2}} {K}, \quad \Gamma_2(t/2)\le \f {e^{\f {Kt}2}-1} K.$$
These local martingales are time changed one dimensional Brownian motions, and so
their exponentials are integrable if
$$ {2 |\gamma| \alpha_3(t)} \f{ \Gamma_i(t)}{t}<\f 12.$$
These hold by construction, hence $\E e^{ t \alpha_3(t) N_{t} }$ is finite and bounded by a universal constant (depending on $\delta_0$ and $|\gamma|$).
 $$\gamma \E \left (  {f(x_t)}\cdot  N_{t}\right)\le\f 1 {t} \f 1 {\alpha_3(t)}\E (f\log f)(x_t)+\f 1 {t} \f 1 {\alpha_3(t)} c(\delta_0, |\gamma|).$$
Finally we observe that $\f 1 {\alpha_3(t)}$ is locally bounded to complete the proof.
\end{proof}

\section{2nd order Feynman-Kac formulas}
\label{section-Hessian}

%The tilde sign above a vector field indicates the horizontal lift of the vector field in question.

Let $B_t=(B_t^1, \dots, B_t^m)$. 
 Set $u_s=TF_s(v_1)$ and $v_s=TF_s(v_2)$. 
In \cite[Theorem 2.3]{Elworthy-Li}, the following formula was given for the second order derivative of the
heat semigroup $P_tf$ where $f$ is a $BC^2$ function, and $(x_t)$ is the solution to a gradient Brownian system.
\begin{equs}
\Hess P_tf(x_0)(v_1,v_2)&=\f 4 { t^2}\E\left\{f(x_t)\int_{\f t2}^t \<Y(x_s)u_s, dB_s\>
\int_0^{\f t2}\<Y(x_s)v_s, dB_s\>\right\}\\
&+{\f 2 t}\E\left\{f(x_{t})\int_0^{\f t 2}
\<DY(x_s)(u_s)(v_s), dB_s\>\right\}\\
&+ {\f 2t}\E\left\{ f(x_t)\int_0^{\f t2} 
\<Y(x_s)\nabla TF_s(v_1,v_2), dB_s\>\right\}.
\end{equs}
This was proved using the martingale method developed in \cite[X.-M. Li]{Li-thesis}, using heat semigroup on differential forms. In this formula
the derivatives of the initial function $f$ is not involved, emphasizing the smoothing property of the Bismut-Witten Laplacian.

The conditions 
 on the driving vector fields and their derivatives are presented in  \cite[Theorem 2.3]{Elworthy-Li}.
Suppose that the curvature and the shape operator of the manifold and their
first order derivatives are bounded, it is not difficult to see that these conditions 
are satisfied. It is possible to use the technique of filtering to remove the redundant noise in the formula for $\Nabla d P_tf$ leading to an intrinsic formula. These computations are lengthy and will involve the second order derivatives of the solution to SDE with respect to the initial data,
 and so involving the derivatives of the driving vector fields up to order $3$. 
 For $\Delta P_tf$, the intrinsic formula is proven to hold  under conditions on the Ricci curvature not their derivatives, 
which does not seem to be the case for the Hessian of $P_tf$ (see \cite{Elworthy-Li-form}),   a local Hessian formula is  presented in \cite{ArnaudonPlankThalmaier}.

Let $(x_t, t\ge 0)$ be a $h$-Brownian motion. If $f$ is a smooth function with compact support we define
$M_t^{df}:=f(x_t)-f(x_0)-\int_0^t \Delta^h f(x_s) ds$.
\begin{definition}[\cite{Elworthy-LeJan-Li-book2}]
\label{def-martingale-part}
If $\alpha_t$ is a predictable process with $\alpha_t\in T^*_{x_t}M$, there is a unique process
which we denote by $\int_0^t \alpha_s d\{x_s\}$ satisfying that
$$\<\int_0^\cdot \alpha_s d\{x_s\}, M_t^{df}\>_t=\int_0^t \alpha_s( \nabla f(x_s)) ds.$$
We say that $\int_0^\cdot \alpha_s d\{x_s\}$ is the integral of $\alpha_s$ w.r.t. the martingale part of $x_s$.
\end{definition}

The $h$-Brownian motion has two notable representations. The first is given by the gradient SDE of an isometric embedding $i: M\to \R^m$:
$$dx_t=\sum _{i=1}^mX_i(x_t)\circ dB_t^i +\nabla h(x_t) dt,$$
see \ref{section-gradient} for detail.
For the gradient SDE, $$\int_0^t \alpha_s d\{x_s\}=\sum_{i=1}^m\int_0^t \alpha_s ( X_i(x_s)) dB_s^i.$$

Let $OM$ denote the orthonormal frame bundle over $M$ and $\pi:OM\to M$ the canonical projection taking
$u: \R^n\to T_xM$ to $x$.
Let $\{e_i\}$ be an o.n.b. of $\R^n$ and $\{H_i\}$ the corresponding fundamental horizontal vector fields on $OM$,
so if $\pi$ is the natural projection taking a frame $u: \R^n \to T_xM$ to the base point $x$, $H_i(u)$ projects to $ue_i$,
and $H_i(u)$ is the horizontal lift of $ue$ from $T_{\pi(u)}M$ to $T_uOM$. We consider
the canonical horizontal SDE:
\begin{equation}\label{canonical-1} du_t=\sum_{i=1}^n H_i(u_t) \circ dB_t^i+\h_{u_t} ({\nabla h}(\pi(u_t))  dt,
\end{equation}
whose initial value $u_0$ is taken from $\pi^{-1}x_0$ and where $\h_{u}(v)$ denotes the horizontal lift of a vector on $T_{\pi(u)}M$ to $T_uOM$,
 through $u_0\in \pi^{-1}(x_0)$.
The solutions $(u_t)$ are called  horizontal Brownian motions.
Then 
$\pi(u_t)$ is a $h$-Brownian motion with initial value $x_0$. This is the canonical representation of the $h$-brownian motion.
Furthermore, $u_tu_0^{-1}$ is the stochastic parallel translation along $(x_t)$,
and
 $$\int_0^t \alpha_sd\{x_s\}=\int_0^t \alpha_s(u_s dB_s).$$

 In the rest of this section,  $(x_t)$ denotes a $h$-Brownian flow, given by one of the above two representations,
  and  $(x_t, W_t, W_t^{(2)})$ solves
equation (\ref{system-equs}). 

\begin{theorem}\label{hessian-lemma}
Suppose \underline{\bf C1}. Then for any $f\in \B_b$, any $t>0$,
\begin{equs}
\Hess P_t^hf(v_1, v_2)
=&
\f 4 {t^2} \E\left[ f(x_t)\int_{t/2}^t \<d\{x_s\}, W_s(v_1)\>\int_0^{t/2} \<d\{x_s\}, W_s(v_2)\>\right]\\
&+\f 2 t
\E \left[ f(x_t)\int_0^{t/2} \<d\{x_s\},   W_s^{(2)}\> \right],
\end{equs}
If $x_t=\pi(u_t)$,  see (\ref{canonical-1}), then
 $d\{x_s\}=u_s dB_s$.
\end{theorem}
 
 \begin{proof}
 It is sufficient to prove this for $f\in C_K^\infty$, followed by a smooth approximation of  the bounded measurable function by a uniformly bounded sequence of functions in $C_K^\infty$ and use the upper bound on the stochastic damped parallel translation and the $L^2$ boundedness of $\|W_s^{(2)}\|$ to pass the limit in a neighbourhood of 
 of a given point. We will also need the gradient formula: for all $0< t \leq T$ and $v\in T_{x_0}M$,
$$
dP^{h}_Tf(v) =\text{ } \frac{1}{t}\E \left[  f(x_T) \int_0^t \< W_s(v),u_s dB_s\>\right]$$
 holds if $\Ric-2\Hess(h)$ is bounded from below. See e.g.  \cite{Elworthy-Li, Elworthy-LeJan-Li-book, Thalmaier-Wang, Arnaudon-Driver-Thalmaier,Li-Thompson}.

Let $f\in C_K^\infty$. We only need to prove this for $x_t=\pi(u_t)$, where $u_t$ solves (\ref{canonical-1}).
An application of Weitzenb\"ock formula $\Delta= \trace \nabla^*\nabla - \Ric^{\sharp}$
shows that $dP_t^hf=e^{\f 12 \Delta^h}(df)$ solves the heat equation on differential 1-forms with initial value~$df$:
$$\f {d}{dt}dP_t^h f=\f 12 \trace (\nabla^{h,*}\nabla) (dP_t^hf)
+\f 12 \left( - {\Ric}^{\sharp}+2(\nabla ^2 h)^\sharp)\right)(dP_t^hf),$$
where $\nabla^{h,*}$ is the adjoint of $\nabla$ with respect to $e^{2h} dx$.
Applying It\^o's formula to $(u_t, W_t)$ and $f\circ \pi$, a routine computation shows that
$$d(P_{T-t}^hf)(W_t(v_1))=d(P_t^hf)(v_1)
+\int_0^t \nabla_{u_s dB_s} (d P^h_{T-s} f)(W_s(v_1)).$$
 Taking $t\uparrow T$ we see that, 
$$(df)(W_T(v_1))=d(P_T^hf)(v_1)
+\int_0^T \nabla_{u_s dB_s} (d(P^h_{T-s} f)(W_s(v_1)).$$
If $\rho^h\ge K$ then $|W_t|^2\le e^{-Kt}$. Consequently we may use It\^o's isometry to obtain:
\begin{equation}\label{2nd-order-1}
\begin{aligned}
&\E (df)(W_T(v_1))\int_0^T \<u_s dB_s, W_s(v_2)\>
=\E \left[\int_0^T \Hess (P^h_{T-s} f)(W_s(v_2),W_s(v_1) )ds\right].
\end{aligned}
\end{equation}
Since $f$ satisfies (\ref{second-diff-5}),
\begin{equation*}
\Hess (P_t^hf)(v_2, v_1) =\E\left[ \Hess f(W_t(v_2), W_t(v_1))\right]+\E df (W_t^{(2)}),
\end{equation*}
 we replace $t$ by $s$ and $f$ by $P^h_{T-s} f$ which is easily seen to be also a $BC^2$ function:
 $$\E \left[\Hess P^h_{T-s}f(W_s(v_1), W_s(v_2))\right]=\Hess (P^h_s(P^h_{T-s}f))(v_1, v_2) -\E d(P^h_{T-s}f) (W_s^{(2)}).$$
 With this, we  return to equation (\ref{2nd-order-1})
\begin{equs}
&\E \left[(df)(W_T(v_1))\int_0^T \<u_s dB_s, W_s(v_2)\>\right]\\
=&T\Hess P^h_Tf(v_1, v_2)
+\E \left[\int_0^T  dP^h_{T-s}f (W_s^{(2)}) \;ds\right].
\end{equs}
Now $T\Hess P^h_Tf(v_1, v_2)$ is given by
\begin{equs}
\E \left[(df)(W_{T}(v_1))\int_0^T \<u_s dB_s, W_s(v_2)\>\right] -
\E \left[\int_0^T  dP^h_{T-s}f (W_s^{(2)}) \;ds\right].
\end{equs}
Setting $T=t/2$ and replacing $f$ by $P^h_{t/2}f$, we see that
\begin{equs}
&(\f t 2) \Hess P^h_tf(v_1, v_2)\\
=&
\E\left[ (dP^h_{t/2}f)(W_{t/2}(v_1))\int_0^{t/2} \<u_s dB_s, W_s(v_2)\>\right]-
\E \left[\int_0^{t/2}  dP^h_{t-s}f (W_s^{(2)}) \;ds\right]\\
=&\f 2 t
\E\left[ f(x_t)\int_{t/2}^t \<u_s dB_s, W_s(u_1)\>\int_0^{t/2} \<u_s dB_s, W_s(v_2)\>\right]\\
&+
\E \left[ f(x_t)\int_0^{t/2} \<u_s dB_s,  W_s^{(2)}\> \right].
\end{equs}
In the above we have used the Markov property and the first order derivative formula for the first term which is obtained from considering $P^h_{\f t 2, t-s}(f\circ \pi)(\psi_{\f t 2,s} (u_{\f t 2}))$ where $\psi_{s,t}$ denotes the solution flow for the canonical SDE begins with $s$ and ending at $t$ and $\psi_{0,t}=u_t$. Apply It\^o's formula to it, followed by taking $s$ to $t$ we see the following,
$$f(x_t)=P^h_{\f t 2, t}f(x_{\f t 2})+\int_{\f t 2}^t \left( dP^h_{\f t 2, t-r} f\right)(u_rdB_r).$$
Multiply both sides by the suitable martingale we see
$$\E \left [f(x_t) \int_{t/2}^t \<u_s dB_s, W_s(u_0)\>\right]
=\E\left[ \int_{ t/ 2}^t \left( dP^h_{\f t 2, t-r} f\right)(W_r(v_1))dr\right]=\f t 2dP^h_{\f t 2} f(v_1).$$
For the second term we observe that
$$P^h_{t/2} f(x_{t/2})=P^h_tf(x_0)
+\int_0^{t/2} d(P^h_{t-s}f )_{x_s}(u_s dB_s)$$
and multiply both sides by $\int_0^{t/2} \<u_s dB_s,  W_{s}^{(2)}\>$. Since $\int_0^t |W_s^{(2)}|^2 ds$ is finite, we obtain $\E \left[P^h_{\f t 2}f(x_{\f t2}) \int_0^{t/2} \<u_s dB_s,  W_{s}^{(2)}\>\right]=\E\left[ \int_0^{t/2} d(P^h_{t-s}f )(W_s^{(2)}) \right]ds$ and the desired formula follows. \end{proof}
 
%This formula is similar
%to that given in D. W. Stroock \cite{Stroock2000}.
%
%We first review the two versions of first order Feynman-Kac formula obtained in \cite{Li-Thompson}, see also \cite{Li-zhao}
%for KPP equations.
%Assume that $\rho^h$ is bounded from below and $V$ is a bounded H\"{o}lder continuous function.
%Then for all $f\in L_\infty$ and $v\in T_{x_0}M$,
%$$\begin{aligned}
%(dP^{h,V}_tf)(v)= &\frac{1}{t} \E \left[ f(x_t) \int_0^t \< W_s(v),u_s dB_s\>\right]\\
%&+ \E \left[   f( x_t) \left(\int_0^t e^{-\int_{t-s}^t V(x_u)du}
%\f {V (x_{t-s})} {t-s}\int_0^{t-s} \< W_r(v),u_r dB_r\>\right)ds\right].\end{aligned}$$
%
%Also,
%\begin{equation}
%\begin{aligned}
%dP^{h,V}_Tf(v) =\text{ }& \frac{1}{t}\E \left[ \V_T f(x_T) \int_0^t \< W_s(v),u_s dB_s\>\right]\\
%&-\frac{1}{t} \E \left[\V_Tf(x_T)\int_0^t \int_0^r  dV(W_s(v)) \,ds\,dr\right].
%%\frac{1}{t}\int_0^t(t-s) \E \left[ \V_s dV(W_s(v)) f(x_T)e^{-\int_s^T V(x_r) dr}\right]\,ds.
%\end{aligned}\nonumber
%\end{equation}
%The Feynman-Kac kernel $P_t^{h,V}$ is a smoothing operator and have a first order formula for $p_t^V$. 
\begin{definition}\label{def-nt}
Given $v_1, v_2\in T_xM$ we define
$$N_{t} =\f  4{t^2}\int_{\f t 2}^t \<d\{x_s\}, W_s(v_1)\>\int_0^{\f t2} \<d\{x_s\}, W_s(v_2)\>,$$
where $d\{x_s\}$ stands for the martingale part of $x_s$.
\end{definition}
In case $x_t$ is given by the projection of $u_t$, 
\begin{equation}
N_{t} =\f 4 {t^2} \int_{\f t 2}^t \<u_s dB_s, W_s(v_1)\>\int_0^{\f t 2} \<u_s dB_s, W_s(v_2)\>.
\end{equation}

\begin{lemma}\label{V-integral}
Assume $\rho^h$ is bounded from below.
If $V:M\to \R$ is bounded and H\"older continuous with $V(x_0)=0$ and $f:M\to \R$ is a bounded measurable function, then
$$\int_0^t \left( \E\left[  P_{t-r}^h (VP_r^h f)N_{ t-r}\right]\right)^{1+\epsilon}dr$$
is finite for some number $\epsilon>0$.
\end{lemma}
\begin{proof}
It is sufficient to show that there exist positive numbers $c$ and $\delta$ such that $$ \E\left[  P_{t-r}^h (VP_r^{h,V} f)N_{ t-r}\right]\le \f c {(t-r)^{1-\delta}}.$$ 
By Cauchy-Schwarz inequality and the bound on $|W_t|$ using the lower bound on $\rho^h$, and Burkholder-Davies-Gundy inequality 
we see that for any $q>0$,
  $\|N_{ t-r}\|_{L^q}$ is of the order of $\f 1{t-r}$, and so
  $$ \E\left[  P_{t-r}^h (VP_r^{h,V} f)N_{ t-r}\right]\le C\f 1 {t-r}  \left(\E\left[  P_{t-r}^h (VP_r^{h,V} f)\right]^2\right)^{\f 12}.$$
To treat $ P_{t-r}^h (VP_r^{h,V} f)$, we simply drop $P_r^{h,V} f$ which can be seen easily, by the Feynman-Kac's formula, to be bounded. We have,
$$\E[P_{t-r}^h (VP_r^h f)]^2\le (|f|_\infty)^2\E V^2(x_t).$$
If $V$ is globally H\"older continuous of order $\alpha$, $$ \E V^2(x_{t-r})\le \E d^{2\alpha}(x_{t-r}, x_0)\le c {(t-r)^\alpha}.$$ giving
enough to ensure the required integrability. We have used part (2) of Lemma \ref{lemma:strong-1}.
  
  Now we suppose that $V$ is H\"oder continuous of order $\alpha$ on a small geodesic ball $B(a)$ of radius $a$ around $V_0$. Denote by $\tau$ 
  the first exit time of $x_t$ from $U$.  Then for a constant $c$ depending on $B(a)$,
 $$\begin{aligned} 
  \E V^2(x_{t-r})\le & V^2(x_{t-r}) \1_{t-r\le \tau} + \E V^2(x_{t-r}) \1_{t-r\ge \tau}\\
  \le & c(B(a)) (t-r)^\alpha+ |V|_\infty \PP(\tau \le t-r)\\
  \le & c(B(a)) (t-r)^\alpha+ |V|_\infty  \f {\E \sup_{s\le t-r} d^2(x_s, x_0)}{a^2}\\
   \le & c(B(a)) (t-r)^\alpha+ |V|_\infty  \f {t-r}{a^2}.
  \end{aligned}$$
We used in line 3 Chebyshev inequality and in line 4 an estimate from  Lemma \ref{lemma:strong-1}. The proof is complete.
\end{proof}
We present an elementary lemma.
\begin{lemma}
\label{Fubini}
Let $g: [0, \infty)\times M\to \R$ be a Borel measurable function. Define $Tg(t)=\int_0^t g(r,x) dr$.
Then $d(Tg)(t)(v)=\int_0^t dg(r,\cdot)(v)$ for any $v\in T_xM$ provided that
for any normal geodesic $(\gamma(t), t\in [0, a])$ with initial value $x$ and initial velocity  $\dot \gamma(0)$,
$dg(r,  \gamma(s)(\dot \gamma(s))\in L^1([0,t]\times [0, a])$, and
 $s\mapsto \int_0^t dg(r,  \gamma(s)) (\dot  \gamma(s))$ is continuous in $L^1$.
\end{lemma}
\begin{proof}
We turn the differential into an integral and use Fubini's theorem to exchange the order of integration and obtain
 $d(Tg)(t)(v) =\lim_{\epsilon \to 0} \f 1 \epsilon\int_0^\epsilon \int_0^t  dg(r,  \gamma(s))(\dot \gamma(s)) ds dr$.
Finally we use the $L^1$ continuity to conclude.
\end{proof}

For $a<t$ and $x_0$ fixed we define 
\begin{equation}
\V_{t-r, t}=(V  (x_{t-r})-V(x_0))e^{-\int_{t-r}^t  [V(x_s)-V(x_0)]ds}.
\end{equation}
Thus we first remover $V(x_0)$ and then put it back.
 %and write $\V_t=V(x_t) e^{-\int_{t-r}^t  V(x_s)ds}$.

\begin{theorem}[Second Order Feynman-Kac Formula]
\label{second-order}
Assume \underline{\bf C1}. Let  $V$ be a bounded H\"{o}lder continuous function.  Then
for any $f\in \B_b(M;\R)$,
\begin{equs}
 \Hess P_t^{h,V}f (v_1, v_2)
=&  e^{-V(x_0)t}\E\left[ f(x_t)N_{ t}\right]
+ e^{-V(x_0)t}\E \left[ f(x_t)\f 2 t\int_0^{t/2} \<d\{x_s\},  W_{s}^{(2)}(v_1, v_2)\> \right]\\
&\quad+ e^{-V(x_0)t}\int_0^t  
\E \left[f (x_t)\f {2 \V_{t-r, t} } {t-r}\int_0^{(t-r)/2} \<d\{x_s\}, W_{s}^{(2)}(v_1,v_2)\> \right] dr\\
&\quad+e^{-V(x_0)t} \int_0^t  \E\left[ f (x_t)  \V_{t-r, t}N_{ t-r}\right]dr.
\end{equs}
\end{theorem}

\begin{proof}
Let us assume that $V(x_0)=0$. 
If $V(x_0)\not =0$, we  shift it to be zero at $x_0$. If $U=V-V(x_0)$ then
$P_t^{h,V}=e^{-V(x_0)t} P_t^{h, U}$, and
$$dP_t^{h,V}f=e^{-V(x_0)t} dP_t^{h, U}f, \quad \nabla dP_t^{h,V}f=e^{-V(x_0)t}  \nabla d P_t^{h, U}f.$$

Let $f\in \B_b$ we may differentiate both sides of the variation of constant formula
and obtain the following formula where the differentiation is w.r.t. the first variable at the point $x_0$ and $V(x_0)=0$:
\begin{equation}
\label{basic-Hessian-potential}
\Hess P_t^{h,V}f (v_1, v_2)=\Hess P_t^hf(v_1, v_2)+\int_0^t \Hess P^h_{t-r} (V P_r^{h,V}f)(v_1, v_2) \,dr.
\end{equation}
To justify the exchange of integral and differential we first check that
$$d P_t^{h,V}f (v)=d P_t^hf(v)+\int_0^t d P^h_{t-r} (V P_r^{h,V}f)(v) \,dr$$
holds for any $v\in T_{x_0}M$.
Indeed let $x_t=\pi(u_t)$ where $(u_t)$ is the canonical process, we apply the formula,
$dP_t^f f=\f 12 \E f(x_t)\int_0^t\<u_r dB_r, W_r(\dot \sigma(s))) \>$, which holds since $\rho^h$ is bounded from below. Then,
$$\int_0^t d P^h_{t-r} (V P_r^{h,V}f)(\dot \sigma(s))
=\int_0^t \f 1 {t-r} \E \left[(V P_{t-r}^{h,V}f)(F_t(\sigma(s))) \int_0^{t-r} \<u_r dB_r, W_r(\dot \sigma(s))) \>\right].$$
Let $\sigma: [0,a]\to M$ be a normal geodesic with initial conditions $x_0$ and $v$. Since $V, f$ are bounded, 
$$\left| \E \left[(V P_{t-r}^{h,V}f)(F_t(\sigma(s))) \int_0^{t-r} \<u_r dB_r, W_r(\dot \sigma(s))) \>\right]\right|
\le C \f 1 {\sqrt{t-r}}.$$
By Lemma \ref{Fubini} the exchange of order of integration and differentiation is justified.

For the second order derivative, we apply Theorem \ref{hessian-lemma} to its integrand. Formally we have,
\begin{equs}
&\int_0^t \Hess P^h_{t-r} (V P_r^{h,V}f) (v_1,v_2)\,dr\\
&=\int_0^t \f 4 {(t-r)^2} \E\left[ (V P_r^{h,V}f) (x_{t-r})\int_{(t-r)/2}^{t-r} \<u_s dB_s, W_s(v_1)\>\int_0^{(t-r)/2} \<u_s dB_s, W_s(v_2)\>\right]\,dr\\
&+\int_0^t\f 2 {t-r}
\E \left[ (V P_r^{h,V}f) (x_{t-r})\int_0^{(t-r)/2} \<u_s dB_s, W_{s}^{(2)}(v_1, v_2)\> \right]\,dr.
\end{equs}
The integrand has an obvious singularity at $r=t$. 
Since $|W_s^{(2)}|^2$ has second moment which is locally bounded, the norm of the integrand in the second integral is of order $\sqrt{t-r}$ and is integrable. For the first integral,
$$\int_{(t-r)/2}^{t-r} \<u_s dB_s, W_s(v_1)\>\int_0^{(t-r)/2} \<u_s dB_s, W_s(v_2)\>,$$
we use Lemma \ref{V-integral} to see the first integrand is uniformly integrable and hence we may exchange the order of integration and differentiation.

Finally, by the zero order Feynman-Kac formula, the first term on the right hand side~is:
$$\int_0^t dr\;\E\left[ V  (x_{t-r})f (x_t)e^{-\int_{t-r}^t V(x_r) dr}N_{t-r}\right]dr,$$
while the second term becomes,
$$
-\int_0^t  dr \;\f 2 {t-r}
\E \left[ V  (x_{t-r})f (x_t)e^{-\int_{t-r}^t V(x_r) dr}\int_0^{(t-r)/2} \<u_s dB_s,  W_{s}^{(2)}\> \right].$$
This concludes the proof.
\end{proof}
An immediate consequence of (\ref{basic-Hessian-potential}) is that the Hessian of the Feynman-Kac kernel $p_t^{h,V}$ can be expressed by
$p_t^h$ and its derivatives. 
 \begin{corollary} We assume $V(x_0)=0$ for simplicity and the conditions of the earlier theorem.
$$\Hess p_t^{h,V}(x_0,y)=\Hess p_t^h(x_0,y)+
\int_0^t \int_M V(z) \Hess p_{t-r}^h(x_0,z) p_r^{h,V} (z,y) dz dr;$$
$$\Hess p_t^{h,V}(x_0,y)=\Hess p_t^h(x_0,y)+
\int_0^t \int_M V(z) \Hess p_{t-r}^h(x_0,z) p_r^{h} (z,y)\E[ e^{-\int_0^r V(Y_s^{r,z,y})}] dz dr,$$
where $Y_s^{r,z,y}$ is the $h$-Brownian bridge with terminal value $r$, initial value $z$ and terminal value $y$.
\end{corollary}
The proof for these are straightforward. The integrals make sense by the previous estimates.
\section{Hessian estimates for Schr\"odinger semi-groups and Feynman-Kac kernels}
\label{section-Hessian-estimates}

If $\phi$ is a random function, set $\HH(\phi)=\phi\log \phi$.
The following first order estimates are given in \cite{Li-Thompson} for a positive function normalised such that $P_t^{h,V}f(x_0)=1$:
$$| \nabla P^{h,V}_t f|_{x_0} \leq \sqrt{ \frac{C_1(t,K)}{t}}\Big(\E\left[\left(\HH\left( f(x_t)e^{\int_0^t V(x_s) ds}\right)\right)^+\right] \Big)^{\frac{1}{2}}+t|\nabla V|_{\infty} C_2(t,K),$$
Where $C_1, C_2$ are explicit constants.
Choosing $f$ to be the Feynman-Kac kernel in the above we have the following  kernel estimates:
$$| \nabla \log p^{h,V}_t |_{x_0} \leq \frac{\sqrt{2C_1 }}{\sqrt{t}}\left(\Big(\sup_{y\in M} \log  \frac{p^h_t(y,y_0)}{p^h_{2t}(x_0,y_0)}+ 2t(\sup V-\inf V)    \Big)^+\right)^{\frac{1}{2}}+t|\nabla V|_{\infty} C_2.$$
Since 
$$\nabla d \log P_t^{h,V}f=\f {\nabla d P^{h,V}_tf}{P_t^{h,V}f}- \nabla \log P_t^{h,V}f \otimes \nabla \log P_t^{h,V}f,$$
for Hessian estimates on $\log P_T^{h,V}$ it is sufficient to estimate the first term in the identity.
 
%\begin{lemma}\label{estimate-11}
% For the Hessian estimate we examine the terms in the Hessian formula involving $V$. Suppose  \underline{\bf C2}.  Let $V$ be a bounded H\"older continuous non-negative function. Suppose that $f$ is a non-negative Borel measurable function. For simplicity we assume that $V(x_0)=0$. Set $\phi_{t-r,t}=P^h_{t-r} (VP_r^{h,V}f) (x_{t-r})$.
%If $f, V$ are not equivalent to the zero function, then $\E \phi_{t-r,t} =P_{t}^h (VP_r^{h,V}f)$ does not vanish, which we assume.
% In Lemma \ref{estimate-10}, taking $f=\f {\phi_{t-r,t}}{ \E \phi_{t-r,t}}$ and taking the time parameter to be $t-r$, w obtain
%$$\begin{aligned} &P_{t-r}^h (VP_r^{h,V}f)\E \left[ \f {\phi_{t-r,t} } {\E \phi_{t-r,t} } \f 2{t-r}\int_0^{(t-r)/2}\<u_s dB_s, W_s^{(2)} (v_1,v_2)\> \right]\\
% & \le\f {P_{t-r}^h (VP_r^{h,V}f)} { \sqrt {t-r}} \left( c_1\E\left [  \phi_{t-r,t}\log ( \phi_{t-r,t})\right] + c_2+c_3A((t-r)/2, K, \|\Rc\|_\infty, \|\Theta^h\|)\right). \end{aligned}
%$$
%Since, ${\E \phi_{t-r,t} }\le V_\infty  {P_t^h f(x_0)}$,
%$$\begin{aligned} & \f 1 {P_t^h f(x_0)}\E \left[  {f(x_t) \V_{t-r,t} } \f 2{t-r}\int_0^{(t-r)/2}\<u_s dB_s, W_s^{(2)} (v_1,v_2)\> \right]\\
%& \le\f  {1}  { \sqrt {t-r}} V_\infty   \left( c_1\E \left[ \f {\phi_{t-r,t} } {\E \phi_{t-r,t} }\log(\f {\phi_{t-r,t} } {\E \phi_{t-r,t} } )\right]+ c_2+c_3A((t-r)/2, K, \|\Rc\|_\infty, \|\Theta^h\|)\right). \end{aligned}
%$$
%%\end{lemma}
%However this method does not work well for the second term involving $V$.

We make an Hesian estimate for the Schr\"odinger operator at the point where the potential function vanishes. (Shift the potential to zero otherwise).
\begin{theorem}\label{Hessian-estimates-thm1}
Assume  \underline{\bf C2}+ \underline{\bf C1(c)}.  Let $V$ be a bounded H\"older continuous non-negative function with $V(x_0)=0$. If $f$ is a non-negative non-trivial function, then
$$\begin{aligned}
\left|\f{\nabla dP_t^{h,V} f}{P_t^h f}\right|_{x_0}
 \le{}& \f 1 {\sqrt t } \left( c_1\E \left[\f {f(x_t)} {P_t^h f(x_0)}\log \left(\f {f(x_t)} {P_t^h f(x_0)}\right)\right] +c_2+c_3A\right)\\
 &+\f 1 t c_2(t,K) \E \left[\f {f(x_t)} {P_t^h f(x_0)}\log \left(\f {f(x_t)} {P_t^h f(x_0)}\right)\right]+\f 1t C_2(t,K) c(\delta_0)\\
 &+ c|f|_\infty (|V|_\infty)^{\f 12} \int_0^t  \f {\sqrt{P_{t-r}|V|(x_0)} }{t-r}dr+c |f|_\infty |V|_\infty \sqrt t\sup_{s \le t}\E | W_{s}^{(2)}(v_1,v_2)|^2 \end{aligned}$$ 
 Here  the constants may depend on the time and on the potential $V$, the dependence are locally bounded.
 \end{theorem}
\begin{proof}
Under the conditions of the theorem, the Hessian formula in Theorem \ref{second-order} holds.
For the first two terms on the right hand side, we divide it by $P_t^h f(x_0)$ and apply  Lemma \ref{estimate-10},
$$\begin{aligned}&\E\left[ f(x_t)N_{ t}\right]
+ \E \left[ f(x_t)\f 2 t\int_0^{t/2} \<d\{x_s\},  W_{s}^{(2)}(v_1, v_2)\> \right]\\
\le{}& \f {P_t^hf(x_0)} {\sqrt t } \left( c_1\E \left[\f {f(x_t)} {P_t^h f(x_0)}\log \left(\f {f(x_t)} {P_t^h f(x_0)}\right)\right] +c_2+c_3A\right)\\
 &+\f {{P_t^hf(x_0)} } t \left(c_2(t,K) \E \left[\f {f(x_t)} {P_t^h f(x_0)}\log \left(\f {f(x_t)} {P_t^h f(x_0)}\right)\right]+\f 1t C_2(t,K) c(\delta_0)
 \right).
\end{aligned}
$$ 
These two terms are of order $\f 1t $ and $\f 1 {\sqrt t}$ (small time) respectively.
For the fourth term, we must use regularity of $V$ to compensate for the singularity in time. Let us apply H\"older and Burkholder-Davis-Gundy inequality
to see that
$$\left|\int_0^t  \E\left[ f (x_t)\V_{t-r, t}N_{ t-r}\right]dr\right|
\le  c|f|_\infty (|V|_\infty)^{\f 12} \int_0^t  (|V  (x_{t-r})|_{L^1(\Omega)}) ^{\f 12}\f {1} {t-r}dr.$$
for some constant $c$, and also,
$$\begin{aligned}
&\int_0^t  
\E \left[f (x_t)\f {2 \V_{t-r, t} } {t-r}\int_0^{(t-r)/2} \<d\{x_s\}, W_{s}^{(2)}(v_1,v_2)\> \right] dr\\
&\le c |f|_\infty |V|_\infty \sqrt t\sup_{s \le t}\E | W_{s}^{(2)}(v_1,v_2)|^2.
\end{aligned}
$$ 
The proof is complete.
\end{proof}

This leads to estimates on the Hessian of the Feynman-Kac kernel.   \begin{corollary}
 \label{corollary-kernel-estimates-1}
 Suppose  \underline{\bf C2} and  \underline{\bf C1(c)}.  Then,
$$\begin{aligned}
\left|\f{\nabla d p_{2t}^{h} }{p_{2t}^h }\right|_{x_0}
\le{}& \f 1 {\sqrt t } \left( c_1 \sup_{y\in M} \left( \log  \frac{p^{h}_t(y,y_0)}{p^{h}_{2t}(x_0,y_0)}\right)
 +c_2+c_3A\right)\\
 &+\f 1 t c_2(t,K)\sup_{y\in M} \left( \log  \frac{p^{h}_t(y,y_0)}{p^{h}_{2t}(x_0,y_0)}\right).
\end{aligned}
$$
\end{corollary}

\begin{proof}
In the theorem above, take $V=0$ and $f=p_t^h(x,y_0)$ where $y_0$ is a fixed point.  Set
${\mathcal H_t} (f,x_0) $
Note that
 $P_t^h (p_t^h(\cdot, y_0))(x)=p_{2t}^h(x,y_0)$ we see that
$$\begin{aligned}
\f{\left|\nabla dp_{2t}^{h} (\cdot, y_0)\right|_{x_0}}{p_{2t}(x_0, y_0)}
 \le{}& \f 1 {\sqrt t } \left( c_1\E \left[\f {p_t^h(x_t, y_0)} {p_{2t}^h(x_0, y_0)}\log \left(\f {p_t^h(x_t, y_0)} {p_{2t}^h(x_0, y_0)}\right)\right] +c_2+c_3A\right)\\
 &+\f 1 t c_2(t,K) \E \left[\f {p_t^h(x_t, y_0)} {p_{2t}^h(x_0, y_0)}\log \left(\f {p_t^h(x_t, y_0)} {p_{2t}^h(x_0, y_0)}\right)\right]. \end{aligned}$$ 
Observe that,
\begin{equation}\label{ineq-hbndestproof}
\begin{aligned}
&{\mathcal H_t} (p_t^h(\cdot, y_0),x_0)
\equiv \E \left[\f {p_t^h(x_t, y_0)} {p_{2t}^h(x_0, y_0)}\log \left(\f {p_t^h(x_t, y_0)} {p_{2t}^h(x_0, y_0)}\right)\right] \\
& \leq \sup_{y\in M} \left( \log  \frac{p^{h}_t(y,y_0)}{p^{h}_{2t}(x_0,y_0)}\right) 
\E \left[\f {p_t^h(x_t, y_0)} {p_{2t}^h(x_0, y_0)}\right]
\leq \sup_{y\in M} \left( \log  \frac{p^{h}_t(y,y_0)}{p^{h}_{2t}(x_0,y_0)}\right).
\end{aligned}
\end{equation}
This completes the proof.
\end{proof}

If the fundamental solution $p_t^h$ satisfies an on diagonal upper bound for the kernel of the following form:  $p_t^h(y, y_0) \le \lambda(t)$,
where $\lambda(t)$ satisfies the doubling condition $\f {\lambda (t)}{\lambda(2t)} \le C$, 
 and a lower bound of the form
 $p_t^h(x_0, y_0) \ge  c_1\lambda(t) e^{-d^2(x_0, y_0)/c_2t}$, then for a constant $c$, we  have
 $$ \sup_{y\in M} \left( \log  \frac{p^{h}_t(y,y_0)}{p^{h}_{2t}(x_0,y_0)}\right)\le c+ c\f{d^2(x_0, y_0)}{t},$$ 
and the familiar type estimates
$$\f{ \left|\nabla d p_{2t}^{h} (x_0, y_0)\right|}{p_{2t}^h (x_0, y_0)}\le c \f{d^2(x_0, y_0)} {t} +c \f 1 t.$$
Such inequalities for strict elliptic operators on $\R^n$ are known to Aronson \cite{Aronson}, and to manifolds
  \cite[Thm. 2.2]{Li-Yau}. See also \cite{Varopoulos,Fabes-Stroock}. In this paper we do not elaborate this in detail
  and  we refer to the following papers on heat kernel and other related estimates:
\cite{LDUP,Davies-Gaussian-bound,Norris97-long,Grigoryan99,Shubin-99,Hebisch-Saloff-Coste,Bakry-Qian,Arnaudon-Thalmaier-Wang,XDLi-Liouville,Norris-Stroock,Zhang-01}.  See also \cite{Li-Xu} for differential {H}arnack inequalities 
and \cite{Barlow-Grigoryan-Kumagai} for the study in the context of Dirichlet form.

\section{Manifolds with a pole and semi-classical Riemannian bridges}
 \label{section-pole}
Let $y_0$ be a pole for $M$. We denote by $J_{y_0}$ 
the Jacobian determinant of the exponential map $\exp_{y_0}: T_{y_0}M\to M$ at $y_0$:
$$J_{y_0}= | \det D_{\exp_\pole^{-1}(y)} \exp_\pole|.$$
 
The subscript $y_0$ will be omitted from time to time.
Set $\Phi(y)=\frac{1}{2}J_{y_0}^{\frac{1}{2}}(y)\Delta J_{y_0}^{-\frac{1}{2}}(y)$.  For $T>0$ fixed, a semi-classical bridge (also called semi-classical Riemannian bridge) $\tilde{x}_s$
is a time dependent diffusion with generator $\frac{1}{2}\triangle + \nabla \log k_{T-s}(\cdot, y_0)$  where, for $d$  the Riemannian distance function,
 $$k_t(x_0, y_0):=(2\pi t)^{-\f n 2}e^{-\f {d^2(x_0, y_0)}{2t}}J^{-\f 12}(x_0).$$
 The radial process $r_t=d(\tilde x_t, y_0)$ of the semi-classical bridge  is the $n$-dimensional  Bessel bridge. On $\R^n$, the semi-classical  bridge agrees with the Brownian motion conditioned to be at $y_0$ at time $T$. The semi-classical bridge are introduced in \cite{Elworthy-Truman-81}, K. D. Elworthy and A. Truman who proved the following formula, \begin{equation}\label{elaform}
p_T^V(x_0,y_0) = k_T(x_0,y_0) \E \left[ e^{\int_0^T \left(\Phi-V\right)(\tilde{x}_s)\,ds} \right],
\end{equation}
under the condition that $\Phi-V$ is bounded from above. 
Their consideration comes from classical mechanics. Their method to overcome the singularity at time $T$ 
is to overshoot the target by a drift $\nabla \log k_{T+\epsilon -t}$. 
In \cite{Li-ibp-sc, Li-Thompson}, we considered the semi-classical bridge on $[0, t]$ where $t<T$ on which the
distributions of the semi-classical bridge process and the $h$-Brownian motion are equivalent,
and use Girsanov transform to  prove this formula
and a gradient formula for the Feynman-Kac semi-groups. In those cases it is not difficult to take the limit $t$ to $T$. 

Let us define
 $$\Phi^h= - \f 1 2 |\nabla h|^2-\f 1 2\Delta h+\Phi.$$
Let $\tilde{u}_s$ denote the solution to the following canonical horizontal SDE on the orthonormal frame bundle
\begin{equation*}
d \tilde{u}_s = \sum_{i=1}^n  H_i(\tilde{u}_s)\circ dB^i_s+\h_{\tilde u_s} \left({ \nabla \log k_{T-s}}(\pi(\tilde{u}_s))\right)\,ds\nonumber
\end{equation*}
with $\tilde{u}_0 =u_0$ and $\widetilde{\nabla \log k_{T-s}}$ indicates the horizontal lift of  $\nabla \log k_{T-s}$, then $\tilde{x}_s = \pi(\tilde{u}_s)$ is a semi-classical bridge with initial value $x_0=\pi(u_0)$. 
%We note that the martingale part of $d \tilde{u}_s$ is $\sum_{i=1}^n  H_i(\tilde{u}_s) dB^i_s$. 
Let $(u_t)$ be the solution to the canonical horizontal SDE (\ref{canonical-1}).

\begin{lemma}[Lemma 3.2 \cite{Li-Thompson}]\label{lemma-Girsanov}
Suppose that $h\in C^2(M; \R)$ and that the $h$-Brownian motion $x_t$
is complete.
 Fix $t \in \left[0,\term\right)$. Then the probability distributions of $u_t$ and $\tilde u_t$ are equivalent  on $\F_t$. If $F$ is a function on the path space $C([0,t];M)$ then $\E F(u_\cdot)=\E F(\tilde u_\cdot M_t)$ where
 \begin{equation}
 \label{cm}
M_t:=e^{h(\tilde x_t)-h(x_0)}\f {k_\term (x_0, y_0) } {k_{\term-\tm} (\tilde{x}_\tm, y_0) }\exp\left[  \int_0^\tm \Phi^{h}(\tilde{x}_\stm) \,d\stm\right].
\end{equation}
\end{lemma}

This result follows from the identity $\frac{1}{2}\Delta^h +\nabla \log (k_{T-s}e^{-h})=\frac{1}{2}\Delta +\nabla (\log k_{T-s})$,   Girsanov's theorem for SDEs states that \begin{equation*}
M_t = \exp \left[ - \sum_{i=1}^m\int_0^\tm  \<\nabla \log (k_{\term-\stm} e^{-h}) (\tilde x_\stm), \tilde  u_\stm e_i\> dB_\stm^i -\f 12 \int_0^\tm | \nabla \log (k_{\term-\stm}e^{-h}) |_{(\tilde x_\stm)}^2 d\stm\right],
\end{equation*}
which reduced to (\ref{cm}) by an application of  It\^o's formula for $\log {(e^{-h}k_{\term-\tm})(\tilde{x}_{\tm}) }$ and
the following identities:
 \begin{equation*}
\begin{aligned}
&\f 1 2|\nabla \log (k_{\term-\stm}e^{-h})|^2 + \left( \f {\partial} {\partial s} +\f 12\Delta^h +Z\right)\log( k_{T-s}e^{-h})\\
 &=\f 18 |\nabla \log J|^2+\f 1 2 |\nabla h|^2-\f 1 4 \Delta(\log J)-\f 1 2\Delta h-|\nabla h|^2+\<Z, \nabla \log (k_{T-s} e^{-h})\>\\
 &=\f 18 |\nabla \log J|^2-\f 1 4 \Delta(\log J)-\f 1 2\Delta h-\f 12 |\nabla h|^2+\<Z, \nabla (\log k_{T-s}e^{-h})\>.
\end{aligned}
\end{equation*}

\subsection{Basic Estimates}
\label{section-estimate-2}

 We intend to give derivative versions of the elementary formula (\ref{elaform}), aiming to obtain Hessian estimates
 for $\log p_t^h$ and $p_t^h$. For this we set $\f {D}{dt}\tilde W_t := \paral_t \frac{d}{dt}  \paral_t^{-1} \tilde W_t$, where $\parals_t$ indicates stochastic parallel transport along $\tilde x_\cdot$.  If $\tilde x_t=\pi(\tilde u_t)$ then $\parals_t=\tilde u_t \tilde u_0^{-1}$. If $\tilde x_t$ is defined to be the solution of the
 gradient SDE, then if we identify $\R^n$ with $T_{x_0}M$, $\parals_t$ is given by the solution of the horizontal lift of the gradient SDE to the orthonormal frame bundle.
 
 We denote by $\tilde W_t(v_0): T_{x_0}M\to T_{\tilde x_t}M$ the solutions to the following equation along the semi-classical bridge $\tilde x_t$,
\begin{equation}\label{damped-semi}
\f {D} {dt}\tilde W_t=-\f 12  {\Ric}_{\tilde x_t}^{\#}(\tilde W_t)+(\nabla^2h)^\sharp _{\tilde x_t}(\tilde W_t), \quad \tilde W_0 = {\id}_{T_{x_0}M}
\end{equation}
% The following was proved in \cite{Li-Thompson}.
% \begin{theorem}\label{thm:fof}
%Assume $y_0$ is a pole, $V \in C^{1,\alpha}\cap BC^1$, $\Phi^h-V$ is bounded above,and $\underline\Ric-2 \Hess h\ge K$. Then
%\begin{equation}
%\<\nabla \log p^{h,V}_T(\cdot,y_0), v\>
%= \frac{1}{T}\E \left[ Z^h \int_0^T \< \tilde{W}_r(v),\tilde{u}_r d\tilde{B}_r-(T-r)\nabla Vdr\>\right],\nonumber
%\end{equation}
%where $Z_T ^h=\frac{\exp\left(  \int_0^t (\Phi^h-V)(\tilde{x}_s)ds\right)}
%{\E \exp\left(  \int_0^t (\Phi^h-V)(\tilde{x}_s)ds\right)}$.
%\end{theorem}

To obtain a formula for the Hessian, let us consider the triple, where $\tilde x_t$ is a semi-classical Riemannian bridge and $(\tilde x_t, \tilde W_t, \tilde W_t^{(2)})$ is the solution to the following system of equations:
\begin{equation}\label{system-equs-2}
\begin{aligned}
\f {D} {dt }\tilde{W}_t=&\left( -\f 12  {\Ric}_{\tilde x_t}^{\#}+(\nabla^2 h)^\sharp\right)(\tilde{W}_t), \quad \tilde W_0 = v_1.
\\
{D} \tilde W_t^{(2)}
=&\left(- \f 12 {\Ric}_{\tilde x_t}^{\sharp}   +(\nabla^2 h)^\sharp_{\tilde x_t}\right) \left(\tilde W_t^{(2)} \right)dt- \f 12 \Theta^h\left( \tilde W_t(v_2)\right) (\tilde  W_t(v_1))dt\\
&+ R( d\{\tilde x_t\}, \tilde W_t(v_2)  )\tilde W_t(v_1), \quad \tilde W_0^{(2)}=v_2
\end{aligned}
\end{equation}
 where $\{\tilde x_s\}$ is the martingale part of $\tilde x_t$.
 The  tilde signs over $W_t$ and $W_t^{(2)}$ indicate that these transport processes are along the semi-classical bridge.
 We use $\tilde W_t(v_1)$ and $W_t^{(2)}(v_1, v_2)$ for the solution with initial value $v_1, v_2\in T_{x_0}M$,
 and $W_t, W_t^{(2)}$ are linear maps from $T_{x_0}M$ to $T_{\tilde x_t}M$, where $\tilde x_0=x_0$.
Occasionally, when there is no risk of ambiguity,  the solutions are also denoted as $W_t, W_t^{(2)}$ (the the initial values are supressed) for simplicity.

Before proceeding to prove the Hessian formula, we make elementary estimations and justify taking the limit $t$ to $T$ in integrals, appearing in the Hessian formula,  of functionals of $\tilde x_t$ from $0$ to $t$. The following elementary inequalities will be used frequently.
\begin{lemma}
\label{elementary-lemma}
Let $p>0$. For a constant $c(p)$, the following estimates hold.
\begin{equation*}
\begin{aligned}
&\sup_{0\le t \le T}\f 2t \int_0^{\f t 2} \f {s^{\f p2} }{(T-s)^{\f p2}T^{\f p 2}}ds \le c(p)T^{-\f p2}, \;
\sup_{0\le t \le \f T 2}\f 2t \int_{\f t 2}^t \f {s^{\f p2} }{(T-s)^{\f p2}T^{\f p 2}}ds \le c(p)T^{-\f p2},\\
%&\sup_{\f T 2\le t \le T}\f 2t \int_0^{\f t 2}\f {s^{\f p2} }{(T-s)^{\f p2}T^{\f p 2}}ds
%\le c(p)T^{-p/2},\\
&\sup_{\f T 2\le t \le T}\f 2t \int_{\f t 2}^t \f {s^{\f p2} }{(T-s)^{\f p2}T^{\f p 2}}ds
\le c(p)T^{-p/2},\quad \hbox{ for } p<2.
\end{aligned}
\end{equation*}
\end{lemma}

%  \begin{equation*}
%\sup_{0\le t \le T}\f 2t \int_0^{\f t 2} \f {s^{\f p2} }{(T-s)^{\f p2}T^{\f p 2}}ds \le  \f 2t  \f {2^{\f p2} }{T^p}\int_0^{\f t 2} s^{\f p2} ds 
%\le \f {2}{p+2} T^{-\f p2};
%\end{equation*}
% $$\sup_{0\le t \le \f T 2}\f 2t \int_{\f t 2}^t \f {s^{\f p2} }{(T-s)^{\f p2}T^{\f p 2}}ds 
%\le \f 2 t \f {2^{\f p2}} {T^p}    \int_{\f t 2}^t s^{\f p2} ds
%\le \f 2 {p+2} T^{-\f p2};$$
% \begin{equation*}
%\begin{aligned}
%&\sup_{\f T 2\le t \le T}\f 2t \int_{\f t 2}^t \f {s^{\f p2} }{(T-s)^{\f p2}T^{\f p 2}}ds
%\le \f 4 T \int_{\f t 2}^t \f 1 {(T-s)^{\f p2}} ds
%=\f 4 T \f 2 {2-p}\left [ (T-\f t 2)^{1- \f p 2}-(T-t)^{1- \f p 2}\right]
%\\
%&\le \left\{\begin{array}{lll}
%\f 8 {2-p} \f 1 T (T-\f t 2)^{1-p/2}\le \f 8 {2-p} T^{-\f p2},\quad &p<2\\
%{ \f 4T} \ln (\f{T-\f t 2}{T-t}), &p=2,\\
%\f 8{p-2} \f 1 T (T-  t)^{1-\f p2}, &p>2.
%\end{array} \right.
%\end{aligned}
%\end{equation*}

\begin{lemma}\label{basic-lemma}
\begin{enumerate}
\item [(1)] For every $p\ge 1$ and $T$ there exists a constant $\delta>0$ such that  if $\Phi^h \le C+\delta d^2(\cdot, y_0)$ for some $C$, then  $\left( e^{ \int_0^t \Phi^{h}(\tilde{x}_s) \,ds}, t\le T\right)$ is $L^p$ bounded.
\item [(2)] For some number $c(p,n)$, depending only on $p$ and $n$,
   $$\E d^p(\tilde x_s, y_0)\le c(p)(T-s)^p\left( \f {d(x_0, y_0) }{T}\right)^p+c(p,n) \left(\f {s(T-s)}T\right)^{\f p 2},$$ 
For any $\delta>0$ and $s\in [0,T)$, $$\begin{aligned}
&\E \left|\nabla \log k_{T-s} (\tilde x_s,y_0)\right|^2\\
& \le  \f 14(1+\delta^{-1}) \E|\nabla \log J(\tilde x_s)|^2
+(1+\delta) \left( \f {d^2(x_0, y_0) }{T^2}+ \f {ns}{T(T-s)}\right).
\end{aligned}$$
%\item [(3)]For any $p\in (1,2)$ there exists constants $c(p)$ and $c(p,n)$ s.t. for $s\in [0,T)$,
%$$\begin{aligned}
%\E \left|\nabla \log k_{T-s} (\tilde x_s,y_0)\right|^p
% \le&  \E|\nabla \log J(\tilde x_s)|^p+c(p)
%\f {d^p(x_0, y_0) }{T^p}\\
%&{}+c(p,n)  \left(\f s {T(T-s)}\right)^{\f p 2}.
%\end{aligned}$$

\item [(3)]  Set $c_{p,J}=\sup_{0\le s \le T} |\nabla \log J^{- \f 12}(\tilde x_s)|_{L^p(\Omega)}$. There exist constants $c(p)$ and $c(p, n)$ such that 
$$\begin{aligned}&  \E \left(  \f 1{t-r}\int_r^t |\nabla  \log( k_{T-s}(\tilde x_s, y_0)|^p ds\right) 
\le   (2C_{p,J})^p +c(p)\f {d^p(x_0, y_0) }{T^p}
+  c(p,n) F(p,r,t),
\end{aligned}$$
where  
\begin{equation}\label{basic}
F(p,r,t):=\f 1{t-r}\int_r^t  \f {s^{\f p2} }{(T-s)^{\f p2}T^{\f p 2}}ds.\end{equation}
Moreover for any $p>0$,  \begin{equation}
\label{integrated-6.6}
\sup_{0\le t \le T}\E\left( \f 2 t \int_0^{\f t 2}  \left|\nabla \log k_{T-s} (\tilde x_s,y_0)\right|^p ds \right)
 \le  (2C_{p,J})^p   +c(p)\f {d^p(x_0, y_0) }{T^p}+c(p,n)  T^{- p/2},
\end{equation}
and for $p<2$,
 $$\sup_{0\le t\le T}\E \left(\f 2 t \int_{\f t 2}^t  \left|\nabla \log k_{T-s} (\tilde x_s,y_0)\right|^p ds\right)
\le (2C_{p,J})^p +c(p)\f {d^p(x_0, y_0) }{T^p}+c(p,n)  T^{- p/2}.$$

%\item[(4)] For $p<2$, there exist  constants $c(p)$ and $c(p,n)$ s.t. 
%$$\begin{aligned} 
%   \int_0^t  \E \left(  \f 1{t-r}\int_r^t |\nabla  \log( k_{T-s}(\tilde x_s, y_0)| ^pds\right)dr
%   \le&  t\left( 2^pC(J) +c(p)\f {d^p(x_0, y_0) }{T^p}\right)\\
%&{}+  c(p,n)\ln T\cdot T^{1-\f p2}. \end{aligned}$$
%Consequently for $p\in (1,2)$, $ \int_0^t   \f 1{t-r}\int_r^t |\nabla  \log( k_{T-s}(\tilde x_s, y_0)| ds\; dr$  is $L^p$ bounded on $[0,T]$.
\end{enumerate}
\end{lemma}
\begin{proof}
 The $\R^n$ valued Brownian bridge is a Gaussian process,  the exponential of its square distance function is integrable in $L^a$ for $a<1/2$ (Fernique's theorem).
Since $d(\tilde x_s, y_0)$ is the $n$-dimensional Bessel bridge, $\sup_{s\le t}\E e^{Cd^2(\tilde x_s, y_0)}$ is finite for  $C$ sufficiently small, and any $p\ge 0$ and $T$,  by choosing sufficiently small $\delta$ we see that $\Phi^h$
$\left( e^{ \int_0^t \Phi^{h}(\tilde{x}_s) \,ds}, t\le T\right)$ is $L^p$ bounded, concluding part (1).

Let $c(p,n)$ denote a constant depending only on $p$ and $n$, varying from line to line. Using the explicit law of the $n$-dimensional Bessel bridge, we see that for some number $c(p,n)$,
   $$\E d^p(\tilde x_s, y_0)\le c(p)\left( \f {(T-s)d(x_0, y_0) }{T}\right)^p+c(p,n) \left(\f {s(T-s)}T\right)^{\f p 2}.$$ 
 Furthermore $\E d^2(\tilde x_t, y_0)=\f {(T-t)^2}{T^2} d^2(x_0, y_0)+\f {nt(T-t)} T$, so the second part of statement (2) follows from applying standard inequalities to the following identity
$$\left|\nabla \log k_{T-t}(\tilde x_t, y_0)\right |
=\left| \nabla \log J^{-\f 12}(\tilde x_t)+\f {(d\nabla d) (\tilde x_t, y_0)}{T-t}\right|.$$
We used the fact that $|\nabla d|=1$.   
 For any $p>0$, 
$$\begin{aligned}
&\E \left|\nabla \log k_{T-s} (\tilde x_s,y_0)\right|^p
\le  2^p \E|\nabla \log J^{-\f 12}(\tilde x_s)|^p
+2^p \f{\E\left|(\nabla d)d(\tilde x_s, y_0)\right|^p} {(T-s)^p}
\\  
&\le  2^p \E|\nabla \log J^{-\f 12}(\tilde x_s)|^p+c(p)
\f {d^p(x_0, y_0) }{T^p}+c(p,n)  \left(\f s {T(T-s)}\right)^{\f p 2},
\end{aligned}$$
and so the first inequality in the statement of part (3) follows:
$$\begin{aligned}& \E \left(  \f 1{t-r}\int_r^t |\nabla  \log( k_{T-s}(\tilde x_s, y_0)|^p ds\right)\\
&\le 2^p\sup_{r\le s \le t} \E|\nabla \log J^{- \f 12}(\tilde x_s)|^p +c(p)\f {d^p(x_0, y_0) }{T^p}
+  c(p,n) \f 1{t-r}\int_r^t  \f {s^{\f p2} }{(T-s)^{\f p2}T^{\f p 2}}ds.
\end{aligned}$$
Replacing $r=\f t 2$ and $t$ by $\f t 2$,  and apply the following inequality from Lemma \ref{elementary-lemma}
to conclude the last two inequalities.
%\begin{equation*}
%\sup_{0\le t\le T} \f 2 t \int_{\f t 2}^t \f {s^{\f p2} }{(T-s)^{\f p2}T^{\f p 2}}ds \le c(p)T^{-\f p2},
%\end{equation*}
%to the last term in the earlier inequality to obtain that  for $p<2$,
%$$\begin{aligned}& 
%\sup_{0\le t \le T}\E \f 2 t \int_{\f t 2}^t  \left|\nabla \log k_{T-s} (\tilde x_s,y_0)\right|^p ds\\
%&\le  2^p\sup_{0\le s \le  T } \E|\nabla \log J^{-\f 12}(\tilde x_s)|^p+c(p)
%\f {d^p(x_0, y_0) }{T^p} +c(p,n)  T^{- p/2},
%\end{aligned}
%$$
%and that for any $p>0$
%$$\begin{aligned}
%&\sup_{0\le t \le T} \E\f 2 t \int_0^{\f t 2}  \left|\nabla \log k_{T-s} (\tilde x_s,y_0)\right|^p ds\\
%&\le  2^p\sup_{0\le s \le T} \E|\nabla \log J^{-\f 12}(\tilde x_s)|^p+c(p)
%\f {d^p(x_0, y_0) }{T^p}+c(p,n)  T^{- p/2}.
%\end{aligned}
%$$
%This completes the proof of part (3).
 \end{proof}
%We return to the first inequality in the statement of part (3) and estimate the integral of its last term. We see that  $$\begin{aligned}
%  &\int_0^t  \f 1{t-r}\int_r^t  \f {s^{\f p2} }{(T-s)^{\f p2}T^{\f p 2}}ds\; dr \le
%  \int_0^t \ln (\f t {t-s}) \f {s^{\f p2} }{(T-s)^{\f p2}T^{\f p 2}}ds  \\
%   \le&\int_0^t \ln (\f t {t-s}){(T-s)}^{-\f p2} ds\le \int_0^t \ln (\f T {T-s}){(T-s)}^{-\f p2} ds\\
%   \le &  c(p)\ln T\cdot \left(T^{1-\f p2} -{(T-t)}^{1-\f p2}\right) - \int_0^t \ln ( {T-s}){(T-s)}^{-\f p2} ds \le c(p)\ln T\cdot T^{1-\f p2}.   \end{aligned}$$
% This gives the required estimate:
%   $$\begin{aligned} 
%   &\int_0^t  \E \left(  \f 1{t-r}\int_r^t |\nabla  \log( k_{T-s}(\tilde x_s, y_0)|^p ds\right)dr\\
%   &\le 2^p\, t\, \sup_{0\le s \le t} \E|\nabla \log J^{- \f 12}(\tilde x_s)|^p +c(p)\,t\,\f {d^p(x_0, y_0) }{T^p}
%+  c(p,n)\ln T\cdot T^{1-\f p2}. \end{aligned}$$
%To see the $L^p$ boundedness
% of $ \f 1{t-r}\int_r^t |\nabla  \log( k_{T-s}(\tilde x_s, y_0)| ds$ on $[\f T 2, T]$,
% we may take $p\ge 1$,
% $$ \E \left(  \f 1{t-r}\int_r^t |\nabla  \log( k_{T-s}(\tilde x_s, y_0))| ds\right) ^p \le \E \left(  \f 1{t-r}\int_r^t |\nabla  \log( k_{T-s}(\tilde x_s, y_0))|^p ds\right).$$

For an orthonormal basis $\{e_i\}_{i=1}^n$ of $\R^n$, we define as following 
a family of independent one dimensional Brownian motions
on the probability space $(\Omega, \F, \F_t, {\mathbb Q})$ where ${\mathbb Q}$ is probability measure that is equivalent to $\mathbb P$  on 
each $\F_t$, where $t<T$, with the density $\f {d{\mathbb Q}}{d\mathbb P}=M_t$ where $(M_t)$ is given in Lemma \ref{lemma-Girsanov}. \begin{definition}\label{constants-1}
Set $\tilde B_t=(\tilde B_t^1, \dots, \tilde B_t^n)$, where
for $1\le i\le n$,
 $$\tilde B_t^i:=B_t^i +\int_0^t d \log( k_{T-s} e^{-h})(\tilde u_s e_i)\, ds.$$ 
 For  a number  $p>0$, the following notation will be used in the next lemma.
  \begin{equation*}
 \begin{array}{lll}
&b_1(2p)=\sup_{0\le s \le T} \left|\nabla \log J^{-\f 12}(\tilde x_s)\right|_{L^{2p}(\Omega)}, 
&b_2(2p)=\sup_{0\le s \le T} \left|\nabla h(\tilde x_s)\right|_{L^{2p}(\Omega)},\\
 &b_3(2)=\sup_{0\le s \le T} \left|\tilde W_s^{(2)}(v_1, v_2)\right|_{L^2(\Omega)}, 
&b_4(2p)= \sup_{0\le s \le T} \left|\tilde W_s^{(2)}(v_1, v_2)\right|_{L^{2p}(\Omega)},\\
&A=( (b_1(2p))^p+(b_2(2p))^p)^{\f 1p}b_4(2p).
 \end{array}
 \end{equation*}

\end{definition}
\begin{lemma}
\label{lemma8.4}
Assume \underline{\bf C4a, C4b}. Let $p\ge 1$. Then for a constant $c(p,n)$, depending only on $p$ and $n$, the following estimate holds for any $t\in (0, T]$,
\begin{equation}\label{pre-estimate}
 \begin{aligned}
 &\left|\f 2 t \int_0^{t/2} \<d\{\tilde x_s\},  \tilde W_{s}^{(2)}(v_1, v_2)\>\right|_{L^p(\Omega)}\\
&\le c(p,n)\left(b_4(2p)\f {d(x_0, y_0) }{T}+ {b_4(2p)} \f 1 {\sqrt T} +{b_3(2)} \f 1{\sqrt{t} } +A\right).
\end{aligned}
\end{equation}
for all unit tangent vectors  $v_1, v_2$. Also  the following holds for $0<t\le T$,
$$\begin{aligned}
&  \left|\int_0^t   \f 2 {t-r}\int_0^{(t-r)/2} \<d\{\tilde x_s\}, \tilde W_{s}^{(2)}(v_1,v_2)\>dr\right|_{L^p(\Omega)}\\
 & \le c(p, n)\left( b_4(2p){d(x_0, y_0) }+[b_4(2p)+b_3 (2)] {\sqrt T}+AT\right).
 \end{aligned}$$
Furthermore, the following stochastic processes,
$$ \exp\left[  \int_0^t \Phi^{h}(\tilde{x}_s) \,ds\right]\cdot \f 2 t\int_0^{t/2} \<d\{\tilde x_s\},  \tilde W_{s}^{(2)}(v_1,v_2)\>$$
$$\exp \left[  \int_0^t \Phi^{h}(\tilde{x}_s) \,ds \right]
\int_0^t  
 \f 2 {t-r}\int_0^{(t-r)/2} \<d\{\tilde x_s\}, \tilde W_{s}^{(2)}(v_1,v_2)\>dr$$
are $L^p$ bounded on $[\f T 2, T]$ and converge in $L^1$ as $t\uparrow T$.

\end{lemma}
\begin{proof}
Let us take the canonical representation of the process  $\tilde x_t$ so integration w.r.t. 
the martingale part of $\tilde x_s$  is the same as w.r.t. $\tilde u_s d\tilde B_s$.
 Since $\|\Theta^h\|+\|\Rc\| \le c e^{\delta_1 d(\cdot, y_0)}$, we apply Lemma \ref{moments-10} to conclude that
$\sup_{s\le t} \E\left( \left |  \tilde W_{s}^{(2)}(v_1, v_2)\right|^q\right)$ is finite. 

Let us  denote $ \tilde W_{s}^{(2)}(v_1, v_2)$ by $\tilde W_s^{(2)}$ to ease the notation.
By It\^o's isometry,  for any $0\le r<t\le T$,
$$\left |\f 1{t-r}\int_r^t \<\tilde u_s dB_s,  \tilde W_{s}^{(2)}\>\right|_{L^2(\Omega)}
\le \f 1{t-r} \sqrt{\int_r^t \E|\tilde W_s^{(2)}| ^2ds} \le  \f 1 {\sqrt{t-r}}\sup_{r\le s \le t}|\tilde W_s^{(2)}|_{L^2(\Omega)}.$$
%and $\left(\f 2 t\int_0^{t/2} \<\tilde u_s d B_s,  \tilde W_{s}^{(2)}, 0\le t\le T\right)$ is uniformly bounded on $[0,T]$.
On the other hand, by the definition of $\tilde B_t$,
$$\begin{aligned}
&\int_r^t\<\tilde u_s d\tilde B_s,  \tilde W_{s}^{(2)}\>
=\int_r^t\<\tilde u_s d B_s,  \tilde W_{s}^{(2)}\> +\int_r^t  d \log( k_{T-s} e^{-h})\left (  \tilde W_{s}^{(2)}\right)ds.
\end{aligned}$$
Below $c(p)$ stands for a constant depending on $p$. By multiple uses of H\"older's inequality and the elementary inequality
$(a+b)^p\le c(p)a^p+c(p)b^p$ we obtain the following for $p\ge 1$,

$$\begin{aligned}&\E\left( \left|\f 2{t-r} \int_0^{\f {t-r} 2} \<\tilde u_s d\tilde B_s,  \tilde W_{s}^{(2)}\>\right|^p\right)
 \le c(p) \E \left(\left|\f 2{t-r} \int_0^{\f {t-r} 2} \<\tilde u_s d B_s,  \tilde W_{s}^{(2)}\>\right|^p\right)\\
 &\qquad \qquad \qquad \qquad +c(p)\E \left( \left|\f 2{t-r} \int_0^{\f {t-r} 2} d \log( k_{T-s} e^{-h})\left (  \tilde W_{s}^{(2)}\right)ds\right|^p\right). \end{aligned}$$
 $$\begin{aligned}&\E\left( \left|\f 2{t-r} \int_0^{\f {t-r} 2} \<\tilde u_s d\tilde B_s,  \tilde W_{s}^{(2)}\>\right|^p\right)
  \le c(p) (t-r)^{-\f p 2}\sup_{0\le s \le T}\left(\E|\tilde W_s^{(2)}|^2\right)^{\f p2}\\
 &\qquad \qquad \qquad +c(p)\f 1{t-r} \int_r^t  \E \left(  \left| \nabla \log( k_{T-s} e^{-h})(\tilde x_s)\right|^p \left |  \tilde W_{s}^{(2)}\right|^p\right) ds.
 \end{aligned}$$
Since   $\sup_{s\le t}\E e^{\delta_2 d(\tilde x_s, y_0)}<\infty$, and since
 $|\nabla h|$ and $|\nabla \log J|$ are bounded by $\delta_2 d(\cdot, y_0)$,  $|\nabla h(\tilde x_s)|$ and  $|\nabla \log J(\tilde x_s)|$ are 
 $L^p$ bounded on $[0,T]$ for any $p\ge 1$.  Since,
$$\left| \nabla \log( k_{T-s} e^{-h})(\tilde x_s)\right|^{2p}
\le c(p)   \left| \nabla \log( k_{T-s} )(\tilde x_s)\right|^{2p}
+c(p) \left| \nabla h(\tilde x_s)\right|^{2p},$$
 we obtain
$$ \begin{aligned}
&\f 1{t-r} \int_r^t  \E \left(  \left| \nabla \log( k_{T-s} e^{-h})(\tilde x_s)\right|^p \left |  \tilde W_{s}^{(2)}\right|^p\right) ds\\
  \le&  \left(\f 2{t-r} \int_0^{\f {t-r} 2}  \E\left( { \left| \nabla \log( k_{T-s} )(\tilde x_s)-\nabla h(\tilde x_s)\right|^{2p}} \right) \right)^{\f 1 2}
\left(\f 2{t-r} \int_0^{\f {t-r} 2}  \E\left( \left |  \tilde W_{s}^{(2)}\right|^{2p}  \right)ds\right)^{\f 1{2}}\\
 \le&(b_4(2p))^p\left(\f 2{t-r} \int_0^{\f {t-r} 2}  \E\left( \left| \nabla \log( k_{T-s} )(\tilde x_s)\right|^{2p}\right)+(b_2(2,p))^p  \right)^{\f 1 2}.
 \end{aligned}$$
On the other hand, by Lemmas \ref{basic-lemma} and \ref{elementary-lemma} we have:
$$\begin{aligned}&
\left(\f 2{t-r} \int_0^{\f {t-r} 2} \E \left|\nabla \log k_{T-s} (\tilde x_s)\right|^{2p}  ds\right)^{\f 1 2}\\
&\le \left( c(p)\sup_{0\le s \le T} \E|\nabla \log J^{-\f 12}(\tilde x_s)|^{2p} +
c(p)\f {d^{2p}(x_0, y_0) }{T^{2p}}+c(p, n) \f 2{t-r}\int_0^{\f {t-r} 2}  \f {s^p }{(T-s)^{p}T^{p}}ds\right)^{\f 12}\\
&\le c(p)(b_1(2p))^p +c(p)
\f {d^{p}(x_0, y_0) }{T^p}+c(p,n)  T^{- p/2}.\end{aligned}
$$
 Finally we obtain
$$\begin{aligned}
&\E\left( \left|\f 2{t-r} \int_0^{\f {t-r} 2} \<\tilde u_s d\tilde B_s,  \tilde W_{s}^{(2)}\>\right|^p\right)\\
 \le & {}  c(p) (t-r)^{-\f p2}(b_3(2))^p
+c(p)(b_4(2p))^2
\left[(b_1(2p))^p  +(b_2(2p))^p
 +\f {d^{p}(x_0, y_0) }{T^p}+c(p,n)  T^{- p/2}\right]\\
 =&{}c(p, n)\left(
  (t-r)^{-\f p2}(b_3(2))^p+A^p+(b_4(2p))^p\f {d^{p}(x_0, y_0) }{T^p}+(b_4(2p))^p  T^{- p/2}\right) \end{aligned}$$
For $r=0$, this is (\ref{pre-estimate}), as required.
For $p\ge 1$  integrating both sides of the above inequality from $0$ to $t$ with respect to $r$ leads to
$$\begin{aligned}
& \E \left|\int_0^t   \f 2 {t-r}\int_0^{(t-r)/2} \<\tilde u_s d\tilde B_s, \tilde W_{s}^{(2)}(v_1,v_2)\>dr\right|^p\\
 &\le c(p,n) t^{\f p2}(b_3(2))^p +t^pc(p, n)\left(
 A^p+(b_4(2p))^p\f {d^{p}(x_0, y_0) }{T^p}+(b_4(2p))^p  T^{- p/2}\right),
 \end{aligned}$$
giving the second required estimate.
% $$\begin{aligned}
%&\E\left( \left|\f 2 t \int_0^{t/2} \<\tilde u_s d\tilde B_s,  \tilde W_{s}^{(2)}\>\right|^p\right)\\
% &\le    c(p, \alpha) \left(c_1+c_2
% + t^{-\f p2}c_3\right)
% +c(p, \alpha)c_4\left( c_1+\f {d^{p}(x_0, y_0) }{T^p}+c(p,n, \alpha)  T^{- p/2}\right),\end{aligned}$$
%Here $c_1=\sup_{0\le s \le T} \E|\nabla \log J^{-\f 12}(\tilde x_s)|^p $, $c_2=\sup_{0\le s \le T} \left(\E|\nabla h(\tilde x_s)|^{p\alpha} \right)^{\f 1 \alpha}$, $c_3=\sup_{0\le s \le T} \left(\E|W_s^{(2)}|^{p/2} \right) $, 
%$c_4= \sup_{0\le s \le T} \left(\E|W_s^{(2)}|^{\alpha' p/2} \right) ^{\f 1 {\alpha'}}$. 

From Lemma \ref{basic-lemma},
 $ e^{ \int_0^t \Phi^{h}(\tilde{x}_s) \,ds}$ is also $L^p$ bounded for any $p\ge 1$, from which we see that
  $ \exp\left[  \int_0^t \Phi^{h}(\tilde{x}_s) \,ds\right]\cdot \f 2 t\int_0^{t/2} \<\tilde u_s d\tilde B_s,  \tilde W_{s}^{(2)}(v_1,v_2)\>$ is $L^p$ bounded.
 The $L^p$ boundedness of the second stochastic process follows by the same argument,  completing the proof of the lemma.
 \end{proof}

\begin{definition}\label{constants-2}
For $0< t \le T$ we set
\begin{equation}\label{Nt} \tilde N_{t}=\f {4}{t^2} \int_{\f t 2}^{t} \<d\{\tilde x_s\}, \tilde W_s(v_1)\>\int_0^{\f t2} \<d\{\tilde x_s\},  \tilde W_s(v_2)\>,
\end{equation}
and
$C_1(K,T)= \max_{ 0\le r \le t \le T}\f{e^{Kt}-e^{Kr}}{K(t-r)}$. And for $\alpha, p$ positive we define
\begin{equation*}
\begin{array}{lll}
&b_1(\alpha p)=\sup_{0\le s \le T}\left|\nabla \log J^{-\f 12}(\tilde x_s)\right|_{L^{p\alpha}(\Omega)}, \quad
b_2(\alpha p)=\sup_{0\le s \le T} \left|\nabla h(\tilde x_s)\right|_{L^{p\alpha}(\Omega)},\\
&a_1(\alpha, p,T)=e^{ |K| T}\left((b_1(\alpha' p))^2+ b_2(\alpha'  p))^2+(b_1(\alpha p))^2+ (b_2(\alpha  p))^2\right), \\
&a_2(K,T)={C_1(K,T)} +e^{ |K| T}.
\end{array}
\end{equation*}
\end{definition}

\begin{lemma}\label{lemma8.5}
Assume \underline{\bf C4a, C4b}.
 Let $\alpha, p\ge 1$ be real numbers such that $\alpha p<2$.
For a constant $c(p)$ depending only on $p$ and $n$, the following holds for $t\in (0,T]$,
\begin{equation*}
\left|\tilde N_{t}\right |_{L^p(\Omega)}
\le c(p,n)\left(a_1(\alpha, p,T)+e^{ |K| T}\f {d^{2}(x_0, y_0)}{T^{ 2}} +a_2(K,T)  \f 1 T\right).
\end{equation*}
Furthermore,
$$\sup_{\f T 2 \le t \le T}\left(\E\left|\int_0^t \tilde N_{t-r}dr\right |^p\right)^{\f 1 p}
\le  c(p, n)  \left(  a_1(\alpha, p,T)T+   e^{ |K| T}\f {d^{2}(x_0, y_0)}{T} +a_2(K,T)\right).$$
%Let $\phi$ be a smooth function with compact support and such that $\phi(y_0)=1$.  
As a consequence,  for any number $p\in [1,2)$, $(\tilde N_{t})$ is $L^p$ bounded on  $ [\f T 2, T] $; and so are
the  stochastic processes:
$$e^{  \int_0^t \Phi^{h}(\tilde{x}_s)ds}  \tilde N_{t}; \quad \hbox{ and } \quad 
  e^{  \int_0^t \Phi^{h}(\tilde{x}_s)ds} \int_0^t |  \tilde N_{ t-r} |dr,
$$
both converge in $L^1$ as $t\uparrow T$.
\end{lemma}

\begin{proof}
Let us take $v_1, v_2$ to be unit vectors and take the canonical representation of $\tilde x_t$ so $d\{\tilde x_s\}=\tilde u_s d\tilde B_s$.   We first compute $\E |\tilde N_{t-r}|^p$  for which it is sufficient to split the products. For
$\alpha, \alpha'>1$ conjugate, 
$\f 1 \alpha +\f 1 {\alpha'}=1$,  we have
\small
$$
\E|\tilde N_{t-r}|^p
\le \left(  \E\left( \left| \f {2}{t-r} \int_{\f {t-r} 2}^{t-r} \<\tilde u_s d\tilde B_s, \tilde W_s(v_1)\>\right|^{p\alpha}\right)\right)^{\f 1 \alpha}
\left( \E\left( \left| \f 2 {t-r} \int_0^{\f {t-r}2 }\<\tilde u_s d\tilde B_s,  \tilde W_s(v_2)\>\right|^{p\alpha'}\right)\right)^{\f 1 {\alpha'}}.$$
\normalsize
Below $c$ denotes a constant whose value may change from line to line.
Using the lower bound $\rho^h\ge -K$ and Burkholder-Davis-Gundy inequality we see that, for any $0\le r<t\le T$,
$$\E \left |\f 1{t-r}\int_r^t \<\tilde u_s dB_s,  \tilde W_{s}(v_i)\>\right|^p
\le \f {c(p)}{(t-r)^p}\E \left(\int_r^t |\tilde W_s(v_i)| ^{2}ds\right)^{\f p2} \le  \f {c(p)C_1(K,T)^{\f p2}} {(t-r)^{\f p 2}},$$
where $C_1(K,T)= \max_{ 0\le r \le t \le T}\f{-e^{Kr}+e^{Kt}}{K(t-r)}$ if $K\not =0$, and $C_1(K,T)=1$ for $K=0$. 
By the definition of $\tilde B$, the estimate above, and
repeated use of the elementary inequality $(a+b)^p\le c(p)a^p+c(p)b^p$ we obtain the following estimates:

\begin{equation*}
\begin{aligned}
& \E\left( \left| \f {1}{t-r} \int_r^{t} \<\tilde u_s d\tilde B_s, \tilde W_s(v_1)\>\right|^{\alpha p}\right)\\
 \le &c(p, \alpha)\E\left( \left| \f {1}{t-r} \int_r^{t} \<\tilde u_s d B_s, \tilde W_s(v_1)\>\right|^{\alpha p}\right)
 \\&+c(p, \alpha )\E\left( \left| \f {1}{t-r} \int_r^{t} \<\nabla \log (k_{T-s}e^{-h}, \tilde W_s(v_1)\>\right|^{\alpha p}\right)\\
\le &c(p,\alpha)\f {(C_1(K,T))^{\f {\alpha p} 2} }  {(t-r)^{\f {\alpha p} 2} }
+c(p,\alpha)\E\left( \f {1}{t-r} \int_r^{t}  \left|\nabla \log (k_{T-s}e^{-h})\right|^{\alpha p} \left| \tilde W_s(v_1)\right|^{\alpha p}ds  \right)\\
\le &c(p, \alpha)\f {(C_1(K,T))^{\f {\alpha p} 2} }  {(t-r)^{\f {\alpha p} 2} }\\
&+c(p, \alpha)e^{\f 12|K|p\alpha T}\E\left( \f {1}{t-r} \int_r^{t} 
 \left|\nabla \log k_{T-s}(\tilde x_s)\right|^{\alpha p} ds+\f {1}{t-r} \int_r^{t}  \left|\nabla h(\tilde x_s)\right|^{\alpha p} ds   \right).
 \end{aligned}
\end{equation*}
We apply Lemma \ref{basic-lemma} to obtain that, \begin{equation}\label{estimate-6.8}
\begin{aligned}
& \E\left( \left|  \f {2}{t-r} \int_{\f {t-r} 2}^{t-r} \<\tilde u_s d\tilde B_s, \tilde W_s(v_1)\>\right|^{\alpha p}\right)
\le c(\alpha, p)\f {(C_1(K,T))^{\f {\alpha p} 2} }  {(t-r)^{\f {\alpha p} 2} }\\
&+c(\alpha, p)e^{\f 12 |K| \alpha pT}\left(  (b_1(\alpha,p))^{ \alpha p}+ (b_2(\alpha,p))^{ \alpha p}+\f {d^{\alpha p}(x_0, y_0)}{T^{\alpha  p}} +F(\alpha p,(t-r)/2, {t-r})\right),
 \end{aligned}
\end{equation}
where for $F(\alpha p, r,t)=\f 1{t-r}  \int_r^t \left( \f  {s}{(T-s)T}\right) ^{\f 12 \alpha p}ds$.
Consequently,  the $p^{th}$ moment of $|\tilde N_{t-r}|^p$ has an upper bound of the same form:
\small
\begin{equation*}
\begin{aligned} 
& \E|\tilde N_{t-r}|^p\\
 \le& \left| \f {2}{t-r} \int_{\f {t-r} 2}^{t-r} \<\tilde u_s d\tilde B_s, \tilde W_s(v_1)\>\right|^p_{L^{p\alpha}(\Omega)}
\left| \f 2 {t-r} \int_0^{\f {t-r}2 }\<\tilde u_s d\tilde B_s,  \tilde W_s(v_2)\>\right|^p_{L^{p\alpha'}(\Omega)}\\
\le &c(\alpha, p)\left( \f {(C_1(K,T))^{\f { p} 2} }  {(t-r)^{\f p 2} }
+e^{\f 12 |K| pT}\left( (b_1(\alpha p))^{p}+ (b_2(\alpha p))^{  p}+\f {d^{p}(x_0, y_0)}{T^{ p}} +\left( F(\alpha p, (t-r)/2, t-r)\right)^{\f 1 \alpha}\right)\right)
\\&\times \left( \f {(C_1(K,T))^{\f { p} 2} }  {(t-r)^{\f p 2} }
+e^{\f 12 |K| pT}\left( (b_1(\alpha' p))^{p}+ (b_2(\alpha' p))^{  p}+\f {d^{p}(x_0, y_0)}{T^{ p}} +\left( F(\alpha' p, 0, t-r)\right)^{\f 1 \alpha'}\right)\right).
\end{aligned}
\end{equation*}
\normalsize
If $\alpha p<2$ we apply Lemma \ref{elementary-lemma},
$$\sup_{0\le t \le T} \left( F(\alpha p, (t-r)/2, t-r))\right)^{\f 1 \alpha} \le c(p, \alpha) T^{-p /2}, \; \; 
\sup_{0\le t \le T} \left( F(\alpha' p, 0,t-r)\right)^{\f 1 \alpha'} \le c(p, \alpha) T^{-p /2}.$$
Thus, for $b_3:=(b_1(\alpha' p))^2+ b_2(\alpha' p))^2+(b_1(\alpha p))^2+ (b_2(\alpha p))^2$,
\begin{equation*}
\left( \E|\tilde N_{t-r}|^p\right)^{\f 1 p}
\le c(\alpha, p)\left( {C_1(K,T)} \f 1 {t-r}
+e^{ |K| T}\left( b_3+\f {d^{2}(x_0, y_0)}{T^{ 2}} +\f 1 T\right)\right).
\end{equation*}

 Set
$b_4:=(b_1(\alpha' p))^{2p}+ b_2(\alpha' p))^{2p}+(b_1(\alpha p))^{2p}+ (b_2(\alpha p))^{2p}$.
Since $L^p$ norms in a probability space increases with $p$ we may assume that $p>1$
and apply H\"older's inequality to obtain:
\begin{equation*}
\begin{aligned}
&\E\left|\int_0^t \tilde N_{t-r}dr\right |^p\\
&\le t^{p-1} \E \left(   \int_0^t \left| \f {1}{t-r} \int_r^{t} \<\tilde u_s d\tilde B_s, \tilde W_s(v_1)\> 
\cdot \f 1r \int_0^{r} \<\tilde u_s d\tilde B_s,  \tilde W_s(v_2)\>\right|^p dr \right)\\
&\le t^{p-1}   \int_0^t  \left| \f {2}{t-r} \int_{\f {t-r} 2}^{t-r} \<\tilde u_s d\tilde B_s, \tilde W_s(v_1)\>\right|^p_{L^{p\alpha}(\Omega)}
\left| \f 2 {t-r} \int_0^{\f {t-r}2 }\<\tilde u_s d\tilde B_s,  \tilde W_s(v_2)\>\right|^p_{L^{p\alpha'}(\Omega)}
 dr\\
&\le  c(\alpha, p) t^{p-1} \int_0^t \left( \f {(C_1(K,T))^p } {(t-r)^p }
+e^{ |K| pT}\left( b_4+\f {d^{2p}(x_0, y_0)}{T^{ 2p}} +T^{-p}\right)\right)dr\\
&\le  c(\alpha, p)  \left(  {(C_1(K,T))^p } 
+T^p e^{ |K| pT}\left( b_4+\f {d^{2p}(x_0, y_0)}{T^{ 2p}} +T^{-p}\right)\right).
%&\le t^{p-1}\left(  \int_0^t \E\left( \left| \f {1}{t-r} \int_r^{t} \<\tilde u_s d\tilde B_s, \tilde W_s(v_1)\>\right|^{p\alpha}\right)\right)^{\f 1 \alpha}
%\left(\int_0^t  \E\left( \left| \f 1r \int_0^{r} \<\tilde u_s d\tilde B_s,  \tilde W_s(v_2)\>\right|^{p\alpha'}\right)\right)^{\f 1 {\alpha'}}dr.
\end{aligned}
\end{equation*}
This completes the proof.
\end{proof}
\subsection{Hessian formula in terms of the semi-classical bridge}

We are ready to prove a formula for $\nabla d p_T^{h,V}$ in terms of the semi-classical bridge.
 The idea is to use Grisanov transform to transform our Hessian formula
in terms of the $h$-Brownian motion to the semi-classical bridge with terminal time $T$, which is valid on any interval $[0,t]$ where $t<T$.
The previous estimates allow us to take $t\to T$ in the formula.
Set  $$\tilde \V_{a,t}=(V(\tilde x_a)-V(x_0)) e^{-\int_{a}^t (V(\tilde x_s) -V(x_0))ds}.$$

\begin{theorem}\label{2nd-order-kernel}
Assume \underline{\bf C4}. Let $V$ be a bounded H\"{o}lder continuous function.  
%\begin{enumerate}
%\item [(a)] $\rho^h\ge K$, $\|\Theta^h\|+\|\Rc\|+|\nabla \log J|^2+|\nabla h|^2\le c e^{c d(\cdot, y_0)}$.
%\item[(b)]  $\Phi^h\le c+cd(\cdot, y_0)$. 
%\item[(c)]
%Assume the gradient SDE is strongly 1-complete or more generally the conclusion of Theorem \ref{second-order}. 
%\end{enumerate}
 Then the following formula holds where $T>0$:
\begin{equs}
&\f{\nabla d p_T^{h,V}(v_1, v_2)}{k_T(x_0, y_0)}e^{h(x_0)-h(y_0)+V(x_0)T} \\
=& \E\left[ e^{\int_0^T \Phi^h(\tilde{x}_s)ds}\tilde N_{T}\right]
+\f 2 T 
\E \left[e^{\int_0^T \Phi^h(\tilde{x}_s)ds} \int_0^{T/2} \<d\{\tilde x_s\}, \tilde W_s^{(2)}(v_1,v_2)\> \right]\\
&+\int_0^T  
\E \left[ \tilde \V_{T-r, T} \; e^{\int_0^T \Phi^h(\tilde{x}_s)ds}\f 2 {T-r}\int_0^{(T-r)/2} \<d\{\tilde x_s\},  
 \tilde W_{s}^{(2)}(v_1,v_2)\> \right] dr\\
&+\int_0^T  \E\left[  \tilde \V_{T-r, T} \;e^{\int_0^T \Phi^h(\tilde{x}_s)ds}\tilde N_{T-r}\right] 
dr.
\end{equs}
\end{theorem}
\begin{proof}
We may assume $V(x_0)=0$, and work with $V-V(x_0)$ otherwise.
We may also take $x_t=\pi(u_t)$ and  $\tilde x_t=\pi(\tilde u_t)$ so $d\{x_s\}=u_s dB_s$.
In the formula of Theorem \ref{second-order}, we take $f=k_{T-\epsilon} (\cdot, y_0)\phi$ where $\epsilon$ is a positive number smaller than $T$, $\phi$ is a smooth function with compact support such that $\phi(\pole)=1$, thus
\begin{equs}
&\nabla dP_t^{h,V}\left( k_{T-\epsilon} (\cdot, y_0)\phi\right)(v_1, v_2)\\
=& \E\left[k_{T-\epsilon} ( x_t, y_0)\phi (x_t)N_{t}\right]
+\E \left[ k_{T-\epsilon} ( x_t, y_0)\phi (x_t)\f 2 t\int_0^{t/2} \< u_s dB_s,  W_{s}^{(2)}(v_1,v_2)\> \right]\\
&+\int_0^t  
\E \left[V  ( x_{t-r})k_{T-\epsilon} (x_t, y_0)\phi (x_t)e^{-\int_{t-r}^t V( x_s) ds}\f 2 {t-r}\int_0^{(t-r)/2} \<u_s dB_s, W_{s}^{(2)}(v_1,v_2)\>dr \right] \\
&+\int_0^t  \E\left[V  ( x_{t-r})k_{T-\epsilon} (x_t, y_0)\phi (x_t)e^{-\int_{t-r}^t V(x_s) ds}N_{t-r}dr\right].
\end{equs}
Take $\epsilon\uparrow t$ in the above identity, followed by taking $t\uparrow T$. The left hand side is:
\begin{equs}
&\lim_{t\uparrow T}\lim_{\epsilon\uparrow t}\nabla dP_t^{h,V}\left( k_{T-\epsilon} (\cdot, y_0)\phi\right)(v_1, v_2)\\
&=\lim_{t\uparrow T}\lim_{\epsilon\uparrow t}\nabla d \left(\int_M p_t^{h,V}(\cdot, z)  k_{T-\epsilon} (z, y_0)\phi(z) dz  \right)(v_1, v_2)\\
&=\lim_{t\uparrow T}\lim_{\epsilon\uparrow t} \int_M \nabla d p_t^{h,V}(\cdot, z)(v_1, v_2)  k_{T-\epsilon} (z, y_0)\phi(z) dz\\
&=\nabla d p_T^{h,V}(\cdot, y_0)(v_1, v_2).
\end{equs}
In the last step we use the fact that $k_{T-t}$ is an approximation of the delta measure at $y_0$, 
% we made a change of variables and transfer the integration to $T_{x_0}M$,  where $(\f 1{2\pi t})^{-\f n 2} e^{ -\f {|\exp^{-1}_{y_0}(z)|^2} 2} $ is an approximation of the delta measure at $0$, 
$J(y_0)=1$, and $\phi$ has compact support, so we may exchange the order of taking limits 
and integration.

On the right hand side, we note that $k_{T-\epsilon}(\cdot, y_0)\phi$ is bounded uniformly in $\epsilon$ for $\epsilon$ small, $V$ is bounded. Also, by Lemma \ref{moments-10},
the stochastic integrals on the right hand side are $L^p$ integrable for any $p\ge 1$. We may therefor interchange the order of taking limits and 
taking expectations. The right hand side, after taking $\epsilon\uparrow t$, is:
\begin{equs}
& \E\left[k_{T-t} ( x_t, y_0)\phi (x_t)N_{t}\right]
+\E \left[ k_{T-t} ( x_t, y_0)\phi (x_t)\f 2 t\int_0^{t/2} \< u_s dB_s,  W_{s}^{(2)}(v_1,v_2)\> \right]\\
&+\int_0^t  
\E \left[V  ( x_{t-r})k_{T-t} (x_t, y_0)\phi (x_t)e^{-\int_{t-r}^t V( x_s) ds}\f 2 {t-r}\int_0^{(t-r)/2} \<u_s dB_s, W_{s}^{(2)}(v_1,v_2)\> \right]dr \\
&+\int_0^t  \E\left[V  ( x_{t-r})k_{T-t} (x_t, y_0)\phi (x_t)e^{-\int_{t-r}^t V(x_s) ds}N_{t-r}\right]dr.
\end{equs}
To the above, we make a Girsanov transform and apply Lemma \ref{lemma-Girsanov} to  transfer the $x_t$ process to the $\tilde x_t$ process
and obtain:
\begin{equs}
& \E\left[ M_tk_{T-t} (\tilde  x_t, y_0)\phi (\tilde x_t) \tilde N_{ t}\right]
+\E \left[ M_tk_{T-t} (\tilde  x_t, y_0)\phi (\tilde x_t)\f 2 t\int_0^{t/2} \<\tilde u_s d\tilde B_s,  \tilde W_{s}^{(2)}(v_1,v_2)\> \right]\\
&+\int_0^t  
\E \left[ M_t V  (\tilde x_{t-r})k_{T-t} (\tilde  x_t, y_0)\phi (\tilde x_t) e^{-\int_{t-r}^t V(\tilde x_s) ds}\f 2 {t-r}\int_0^{(t-r)/2} \<\tilde u_s d\tilde B_s, \tilde W_{s}^{(2)}(v_1,v_2)\> \right]dr \\
&+\int_0^t  \E\left[ M_t V  (\tilde x_{t-r})k_{T-t} (\tilde  x_t, y_0)\phi (\tilde x_t)e^{-\int_{t-r}^t V(\tilde x_s) ds}\tilde N_{t-r}\right]dr.
\end{equs}
The $k_{\term-\tm} (\tilde{x}_\tm, y_0)$ term in the above integrals cancels with that in $M_t$, where
$$M_t=e^{h(\tilde x_t)-h(x_0)}\f {k_\term (x_0, y_0) } {k_{\term-\tm} (\tilde{x}_\tm, y_0) }\exp\left[  \int_0^\tm \Phi^{h}(\tilde{x}_\stm) \,d\stm\right].$$ Set
$A_t=e^{h(\tilde x_t)-h(x_0)} {k_\term (x_0, y_0) } \exp\left[  \int_0^\tm \Phi^{h}(\tilde{x}_\stm) \,d\stm\right]$, the right hand side simplifies to
the following expression:
\begin{equs}
& \E\left[ A_t\phi (\tilde x_t) \tilde N_{ t}\right]
+\E \left[ A_t\phi (\tilde x_t)\f 2 t\int_0^{t/2} \<\tilde u_s d\tilde B_s,  \tilde W_{s}^{(2)}(v_1,v_2)\> \right]\\
&+\int_0^t  
\E \left[ A_t V  (\tilde x_{t-r})\phi (\tilde x_t) e^{-\int_{t-r}^t V(\tilde x_s) ds}\f 2 {t-r}\int_0^{(t-r)/2} \<\tilde u_s d\tilde B_s, \tilde W_{s}^{(2)}(v_1,v_2)\>\right]dr \\
&+\int_0^t  \E\left[ A_t V  (\tilde x_{t-r})\phi (\tilde x_t)e^{-\int_{t-r}^t V(\tilde x_s) ds}\tilde N_{t-r}\right]dr.
\end{equs}
We are now free to take the limit $t\uparrow T$. Since $\phi(\tilde x_t)\to 1$ and $\Phi^h$ does not depend on time,  $A_t\phi(\tilde x_t)$ converges almost surely to $$A_T\phi(\tilde x_T)=e^{h(y_0)-h(x_0)} {k_\term (x_0, y_0) } \exp\left[  \int_0^T \Phi^{h}(\tilde{x}_\stm) \,d\stm\right].$$
Since $\phi$ has compact support and $ \int_0^T \Phi^{h}(\tilde{x}_\stm) $ is $L^p$ bounded $L^p$ for all $p$, the above convergence
holds in $L^p$ for any $p$. By Lemma \ref{lemma8.5}, we may exchange the order of taking the limit $t \to T$ and taking expectation the convergence of the earlier long expression to the following :
\begin{equs}
&\nabla d p_T^{h,V}(\cdot, y_0)(v_1, v_2)\\
&= \E\left[ A_T\tilde N_{T}\right]
+\E \left[ A_T\f 2 T\int_0^{T/2} \<\tilde u_s d\tilde B_s,  \tilde W_{s}^{(2)}(v_1,v_2)\> \right]\\
&+\int_0^T 
\E \left[ A_T V  (\tilde x_{T-r}) e^{-\int_{T-r}^T V(\tilde x_s) ds}\f 2 {T-r}\int_0^{(T-r)/2} \<\tilde u_s d\tilde B_s, \tilde W_{s}^{(2)}(v_1,v_2)\> \right] dr\\
&+\int_0^T  \E\left[ A_T V  (\tilde x_{T-r})e^{-\int_{T-r}^T V(\tilde x_s) ds}\tilde N_{T-r}\right]dr.
\end{equs} 
To summarise this is exactly $\nabla d p_T^{h,V}(\cdot, y_0)(v_1, v_2)$ and the proof is complete.
\end{proof}

 \section{Exact Gaussian bounds for weighted Laplacian}
 \label{section-exact-Gaussian}
Let $T>0$ and $x_0\in M$.  Set $\beta_T^h=e^{\int_0^T \Phi^h(\tilde{x}_s)ds} $   and 
 $Z_T=\f {e^{\int_0^T \Phi^h(\tilde{x}_s)\,ds}} {\E \left[ e^{\int_0^T \Phi^h(\tilde{x}_s)\,ds} \right]}$.
 Let $ST_{x_0}M$ denote the unit sphere in the tangent space $T_{x_0}M$.

\begin{corollary}
 Assume \underline{\bf C4}. Let $q\ge 1$. Then there exists a constant $c$, explicitly given in the proof,
 s.t. for  any $v_1, v_2\in ST_{x_0}M$,
 \begin{equs}
& \left|\nabla d p_T^h(v_1, v_2\right|
 \le c\;T^{-\f n 2}e^{-\f {d^2(x_0, y_0)}{2T}}  J_{y_0}^{-\f 12}(x_0) {e^{h(y_0)-h(x_0)} }|\beta_T^h|_{L^{q}(\Omega)}\left( \f {d^2}{T^2} +\f 1 { T}+1\right)
\end{equs}
Also,  the following formula holds:
\begin{equation}
\label{Hessian-kernel-5}
\f{\nabla d p_T^h(v_1, v_2)} {p_T^h(x_0, y_0)}
= \E\left[ Z_T\tilde N_{ T}\right]
+\E \left[Z_T \f 2 T \int_0^{T/2} \<d\{\tilde x_s\}, \tilde W_s^{(2)}(v_1,v_2)\> \right].
\end{equation}
Furthermore, there exists a function $C$  locally bounded in $t$ and explicitly given in the proof,
such that $$\left| \nabla d\log  p_T^h \right|\le C(T, K, q,\delta_1, \delta_2) \; |Z_T|_{L^{q}(\Omega)}\left( \f {d^2}{T^2} +\f 1 { T}+1\right)
+ |Z_T|^2_{L^{2} (\Omega)} \f 1 {T} \left(\f {e^{KT}-1}{KT}\right).$$
\end{corollary} 

\begin{proof}
Let $v_1, v_2\in T_{x_0}M$. In Theorem \ref{2nd-order-kernel} where we take $V=0$
\begin{equs}
&\f{\nabla d p_T^{h}(v_1, v_2)}{k_T(x_0, y_0)}e^{h(x_0)-h(y_0)} \\
=& \E\left[ e^{\int_0^T \Phi^h(\tilde{x}_s)ds}\tilde N_{T}\right]
+\f 2 T 
\E \left[e^{\int_0^T \Phi^h(\tilde{x}_s)ds} \int_0^{T/2} \<d\{\tilde x_s\}, \tilde W_s^{(2)}(v_1,v_2)\> \right],\end{equs}
to see that for constant $p\ge 1$ and $p'$ satisfying $\f 1 p+\f 1 {p'}=1$,
$$\begin{aligned}
&\f{|\nabla d p_T^{h}(v_1, v_2)|}{k_T(x_0, y_0)}e^{h(x_0)-h(y_0)}\\
&\le \left|e^{\int_0^T \Phi^h(\tilde{x}_s)ds}\right|_{L^{p'}(\Omega)} \left(|N_T|_{L^p(\Omega)}+
\left|\f 2 T  \int_0^{T/2} \<d\{\tilde x_s\}, \tilde W_s^{(2)}(v_1,v_2)\>\right|_{L^p(\Omega)}
\right).
\end{aligned}
$$
 The first required estimate follows from  the following estimates in Lemma \ref{lemma8.5} and  Lemma \ref{lemma8.4}: 
\begin{equation*}
\begin{aligned}&\left|\tilde N_{T}\right |_{L^p(\Omega)}
\le c(p,n)\left(a_1(\alpha, p,T)+e^{ |K| T}\f {d^{2}(x_0, y_0)}{T^{ 2}} +a_2(K,T)  \f 1 T\right),\\
 &\left|\f 2 T \int_0^{T/2} \<d\{\tilde x_s\},  \tilde W_{s}^{(2)}(v_1, v_2)\>\right|_{L^p(\Omega)}
\le c(p,n)\left(b_4(2p)\f {d(x_0, y_0) }{T}+( {b_4(2p)} +{b_3(2)})\f 1 {\sqrt T} +A\right).
\end{aligned}
\end{equation*}
We take $q=p'$ and set $$S_T= a_1(\alpha, p,T)+A+ e^{ |K| T}\f {d^{2}(x_0, y_0)}{T^{ 2}} +
b_4(2p)\f {d(x_0, y_0) }{T}+\f {a_2(K,T)}{T}+ \f{{b_4(2p)} +b_3(2)} {\sqrt T}.$$
Then there exists a universal constant $c(p,n)$ s.t. \begin{equs}
& \left|\nabla d p_T^h(v_1, v_2|\right|
 \le c(q,n)T^{-\f n 2}e^{-\f {d^2(x_0, y_0)}{2T}}  J_{y_0}^{-\f 12}(x_0) {e^{h(y_0)-h(x_0)} }|\beta_T^h|_{L^{q}(\Omega)} S_T.
\end{equs}

Next by the Girsanov transform in Lemma \ref{lemma-Girsanov}, we see that
$$p_T^h(x_0,y_0) = k_T(x_0,y_0)e^{h(y_0)-h(x_0)} \E \left[ e^{\int_0^T \Phi^h(\tilde{x}_s)\,ds} \right],$$
formula (\ref{Hessian-kernel-5}) follows immediately.
To formula (\ref{Hessian-kernel-5}) we apply H\"older's inequality to see 
$$\begin{aligned}\left |\f{\nabla d p_T^h(v_1, v_2)} {p_T(x_0, y_0)}\right|
 \le & |Z_T|_{L^{p'}(\Omega)} 
c(p,n)\left(b_4(2p)\f {d(x_0, y_0) }{T}+( {b_4(2p)} +{b_3(2)}) \f 1 {\sqrt T}+A\right)\\
&{}+ |Z_T|_{L^{p'}(\Omega)} c(p,n)\left(a_1(\alpha, p,T)+e^{ |K| T}\f {d^{2}(x_0, y_0)}{T^{ 2}} +a_2(K,T)  \f 1 T\right).
\end{aligned}$$

Finally, the following formula holds,
$$
dp^{h}_T(\cdot,y_0)= \frac{1}{T}e^{h(y_0)-h(x_0)}k_T(x_0, y_0)\E \left[ e^{\int_0^T \Phi^h(\tilde x_s) ds}
 \left(\int_0^T \< \tilde{W}_r(\cdot),\tilde{u}_r d\tilde{B}_r\> \right)\right], $$
under our assumptions. This was proved in  \cite{Li-Thompson} under the assumption that $\Phi^h$ is assumed to be bounded from below.
It can be seen from the following formula
 $$dP^{h,V}_Tf(v) =\frac{1}{T}\E \left[  f(x_T) \int_0^T \< W_s(v),u_s dB_s\>\right],$$
holds for $f$ bounded under the conditions that $\rho^h$ bounded from below.
 The estimates in this section and the proof of Theorem \ref{2nd-order-kernel} allow us to conclude this
 under  \underline{\bf C4(b)} on the growth of $\Phi^h$, and use the exponential integrability of  the distance function from $y_0$, c.f. Lemma \ref{lemma:strong-1}.
In particular we obtain
$$
|d \log p^{h}_T(\cdot,y_0)|^2=\left| \E \left[ Z_T
 \left(\f 1 T\int_0^T \< \tilde{W}_r(\cdot),\tilde{u}_r d\tilde{B}_r\> \right)\right]\right|^2
 \le \E (|Z_T|^2) \f 1 {T} \left(\f {e^{KT}-1}{KT}\right).$$
 Putting them together using $\nabla d\log p_t=\f {\nabla d p_t} {p_t}-\nabla \log p_t\otimes \nabla \log p_t$,
 $$\begin{aligned}\left |{\nabla d \log p_T^h(v_1, v_2)}\right|
 \le & |Z_T|_{L^{p'}(\Omega)} 
c(p,n)\left(b_4(2p)\f {d(x_0, y_0) }{T}+( {b_4(2p)} +{b_3(2)}) \f 1 {\sqrt T}+A\right)\\
&{}+ |Z_T|_{L^{p'}(\Omega)} c(p,n)\left(a_1(\alpha, p,T)+e^{ |K| T}\f {d^{2}(x_0, y_0)}{T^{ 2}} +a_2(K,T)  \f 1 T\right)\\
&+  |Z_T|^2_{L^{2} (\Omega)} \f 1 {T} \left(\f {e^{KT}-1}{KT}\right).
\end{aligned}$$
This completes the proof of the corollary.
\end{proof}

 \subsection*{Summary}
 To conclude, we obtained 
 exact Gaussian estimates for $\nabla d p_t^h$ on manifolds with a pole under condition \underline{\bf C4};
and formulas for $\Hess P_t^{h,V}f$ and for $p_t^{h,V}$ on general manifolds,
 from which estimates are given under condition \underline{\bf C2}.

\bibliographystyle{alpha}
\bibliography{Riemannian-bridges}

\def\cprime{$'$} \def\cprime{$'$} \def\cprime{$'$} \def\cprime{$'$}
\begin{thebibliography}{CFKS87}

\bibitem[ADT07]{Arnaudon-Driver-Thalmaier}
Marc Arnaudon, Bruce~K. Driver, and Anton Thalmaier.
\newblock Gradient estimates for positive harmonic functions by stochastic
  analysis.
\newblock {\em Stochastic Process. Appl.}, 117(2):202--220, 2007.

\bibitem[Aid00]{Aida}
Shigeki Aida.
\newblock Logarithmic derivatives of heat kernels and logarithmic {S}obolev
  inequalities with unbounded diffusion coefficients on loop spaces.
\newblock {\em J. Funct. Anal.}, 174(2):430--477, 2000.

\bibitem[Aid04]{Aida-Gaussian-estimates}
Shigeki Aida.
\newblock Precise {G}aussian estimates of heat kernels on asymptotically flat
  {R}iemannian manifolds with poles.
\newblock In {\em Recent developments in stochastic analysis and related
  topics}, pages 1--19. World Sci. Publ., Hackensack, NJ, 2004.

\bibitem[Air76]{Airault}
H{\'e}l{\`e}ne Airault.
\newblock Subordination de processus dans le fibr\'e tangent et formes
  harmoniques.
\newblock {\em C. R. Acad. Sci. Paris S\'er. A-B}, 282(22):Aiii, A1311--A1314,
  1976.

\bibitem[AKR12]{Albverio-Kawabi-Rockner}
Sergio Albeverio, Hiroshi Kawabi, and Michael R{\"o}ckner.
\newblock Strong uniqueness for both {D}irichlet operators and stochastic
  dynamics to {G}ibbs measures on a path space with exponential interactions.
\newblock {\em J. Funct. Anal.}, 262(2):602--638, 2012.

\bibitem[AM92]{Airault-Malliavin}
H.~Airault and P.~Malliavin.
\newblock Integration on loop groups. {II}. {H}eat equation for the {W}iener
  measure.
\newblock {\em J. Funct. Anal.}, 104(1):71--109, 1992.

\bibitem[AM05]{Albeverio-Mazzucchi}
S.~Albeverio and S.~Mazzucchi.
\newblock Feynman path integrals for polynomially growing potentials.
\newblock {\em J. Funct. Anal.}, 221(1):83--121, 2005.

\bibitem[APT03]{ArnaudonPlankThalmaier}
Marc Arnaudon, Holger Plank, and Anton Thalmaier.
\newblock A {B}ismut type formula for the {H}essian of heat semigroups.
\newblock {\em C. R. Math. Acad. Sci. Paris}, 336(8):661--666, 2003.

\bibitem[Aro67]{Aronson}
D.~G. Aronson.
\newblock Bounds for the fundamental solution of a parabolic equation.
\newblock {\em Bull. Amer. Math. Soc.}, 73:890--896, 1967.

\bibitem[AS82]{Aizenman-Simon}
M.~Aizenman and B.~Simon.
\newblock Brownian motion and {H}arnack inequality for {S}chr\"odinger
  operators.
\newblock {\em Comm. Pure Appl. Math.}, 35(2):209--273, 1982.

\bibitem[AS05]{Arnaudon-Simon}
Marc Arnaudon and Thomas Simon.
\newblock Concentration of the {B}rownian bridge on {C}artan-{H}adamard
  manifolds with pinched negative sectional curvature.
\newblock {\em Ann. Inst. Fourier (Grenoble)}, 55(3):891--930, 2005.

\bibitem[ATW06]{Arnaudon-Thalmaier-Wang}
Marc Arnaudon, Anton Thalmaier, and Feng-Yu Wang.
\newblock Harnack inequality and heat kernel estimates on manifolds with
  curvature unbounded below.
\newblock {\em Bull. Sci. Math.}, 130(3):223--233, 2006.

\bibitem[Aze74]{Azencott}
Robert Azencott.
\newblock Behavior of diffusion semi-groups at infinity.
\newblock {\em Bull. Soc. Math. France}, 102:193--240, 1974.

\bibitem[BGK12]{Barlow-Grigoryan-Kumagai}
Martin~T. Barlow, Alexander Grigor'yan, and Takashi Kumagai.
\newblock On the equivalence of parabolic {H}arnack inequalities and heat
  kernel estimates.
\newblock {\em J. Math. Soc. Japan}, 64(4):1091--1146, 2012.

\bibitem[Bis84]{Bismut}
Jean-Michel Bismut.
\newblock {\em Large deviations and the {M}alliavin calculus}, volume~45 of
  {\em Progress in Mathematics}.
\newblock Birkh\"auser Boston, Inc., Boston, MA, 1984.

\bibitem[BQ97]{Bakry-Qian}
Dominique Bakry and Zhongmin Qian.
\newblock On {H}arnack estimates for positive solutions of the heat equation on
  a complete manifold.
\newblock {\em C. R. Acad. Sci. Paris S\'er. I Math.}, 324(9):1037--1042, 1997.

\bibitem[CFKS87]{Cycon-Froese-Kirsh-Simon}
H.~L. Cycon, R.~G. Froese, W.~Kirsch, and B.~Simon.
\newblock {\em Schr\"odinger operators with application to quantum mechanics
  and global geometry}.
\newblock Texts and Monographs in Physics. Springer-Verlag, Berlin, study
  edition, 1987.

\bibitem[CGT82]{Cheeger-Gromov-Taylor}
Jeff Cheeger, Mikhail Gromov, and Michael Taylor.
\newblock Finite propagation speed, kernel estimates for functions of the
  {L}aplace operator, and the geometry of complete {R}iemannian manifolds.
\newblock {\em J. Differential Geom.}, 17(1):15--53, 1982.

\bibitem[CKM93]{Cranston-kendall-March}
Michael Cranston, Wilfrid~S. Kendall, and Peter March.
\newblock The radial part of {B}rownian motion. {II}. {I}ts life and times on
  the cut locus.
\newblock {\em Probab. Theory Related Fields}, 96(3):353--368, 1993.

\bibitem[CL14]{Chen-Li-flow}
Xin Chen and Xue-Mei Li.
\newblock Strong completeness for a class of stochastic differential equations
  with irregular coefficients.
\newblock {\em Electron. J. Probab.}, 19:no. 91, 34, 2014.

\bibitem[CLW10]{Chen-Li-Wu}
Xin Chen, Xue-Mei Li, and Bo~Wu.
\newblock A {P}oincar\'e inequality on loop spaces.
\newblock {\em J. Funct. Anal.}, 259(6):1421--1442, 2010.

\bibitem[CLY81]{CLY}
Siu~Yuen Cheng, Peter Li, and Shing~Tung Yau.
\newblock On the upper estimate of the heat kernel of a complete {R}iemannian
  manifold.
\newblock {\em Amer. J. Math.}, 103(5):1021--1063, 1981.

\bibitem[CZ07]{Coulhon-Zhang}
Thierry Coulhon and Qi~S. Zhang.
\newblock Large time behavior of heat kernels on forms.
\newblock {\em J. Differential Geom.}, 77(3):353--384, 2007.

\bibitem[Dav88]{Davies-Gaussian-bound}
E.~B. Davies.
\newblock Gaussian upper bounds for the heat kernels of some second-order
  operators on {R}iemannian manifolds.
\newblock {\em J. Funct. Anal.}, 80(1):16--32, 1988.

\bibitem[DL96]{Driver-Lohrenz}
Bruce~K. Driver and Terry Lohrenz.
\newblock Logarithmic {S}obolev inequalities for pinned loop groups.
\newblock {\em J. Funct. Anal.}, 140(2):381--448, 1996.

\bibitem[Don79]{Donnelly}
Harold Donnelly.
\newblock Spectral geometry for certain noncompact {R}iemannian manifolds.
\newblock {\em Math. Z.}, 169(1):63--76, 1979.

\bibitem[Dri92]{Driver92}
Bruce~K. Driver.
\newblock A {C}ameron-{M}artin type quasi-invariance theorem for {B}rownian
  motion on a compact {R}iemannian manifold.
\newblock {\em J. Funct. Anal.}, 110(2):272--376, 1992.

\bibitem[DS99]{Davies-Safanov}
Brian Davies and Yuri Safarov, editors.
\newblock {\em Spectral theory and geometry}, volume 273 of {\em London
  Mathematical Society Lecture Note Series}.
\newblock Cambridge University Press, Cambridge, 1999.
\newblock Papers from the ICMS Instructional Conference held in Edinburgh,
  March 30--April 9, 1998.

\bibitem[DT01]{Driver-Thalmaier}
Bruce~K. Driver and Anton Thalmaier.
\newblock Heat equation derivative formulas for vector bundles.
\newblock {\em J. Funct. Anal.}, 183(1):42--108, 2001.

\bibitem[Ebe00]{Eberle}
Andreas Eberle.
\newblock Spectral gaps on loop spaces: a counterexample.
\newblock {\em C. R. Acad. Sci. Paris S\'er. I Math.}, 330(3):237--242, 2000.

\bibitem[EL94]{Elworthy-Li}
K.~D. Elworthy and X.-M. Li.
\newblock Formulae for the derivatives of heat semigroups.
\newblock {\em J. Funct. Anal.}, 125(1):252--286, 1994.

\bibitem[EL98]{Elworthy-Li-form}
K.~David Elworthy and Xue-Mei Li.
\newblock Bismut type formulae for differential forms.
\newblock {\em C. R. Acad. Sci. Paris S\'er. I Math.}, 327(1):87--92, 1998.

\bibitem[EL06]{Elworthy-Li-icm}
K.~David Elworthy and Xue-Mei Li.
\newblock Geometric stochastic analysis on path spaces.
\newblock In {\em International {C}ongress of {M}athematicians. {V}ol. {III}},
  pages 575--594. Eur. Math. Soc., Z\"urich, 2006.

\bibitem[EL07]{Elworthy-Li07}
K.~D. Elworthy and Xue-Mei Li.
\newblock It\^o maps and analysis on path spaces.
\newblock {\em Math. Z.}, 257(3):643--706, 2007.

\bibitem[ELJL99]{Elworthy-LeJan-Li-book}
K.~D. Elworthy, Y.~Le~Jan, and Xue-Mei Li.
\newblock {\em On the geometry of diffusion operators and stochastic flows},
  volume 1720 of {\em Lecture Notes in Mathematics}.
\newblock Springer-Verlag, Berlin, 1999.

\bibitem[ELJL10]{Elworthy-LeJan-Li-book2}
K.~David Elworthy, Yves Le~Jan, and Xue-Mei Li.
\newblock {\em The geometry of filtering}.
\newblock Frontiers in Mathematics. Birkh\"auser Verlag, Basel, 2010.

\bibitem[Elw82]{Elworthy-book}
K.~D. Elworthy.
\newblock {\em Stochastic differential equations on manifolds}, volume~70 of
  {\em London Mathematical Society Lecture Note Series}.
\newblock Cambridge University Press, Cambridge-New York, 1982.

\bibitem[Eng06]{Engoulatov}
A.~Engoulatov.
\newblock A universal bound on the gradient of logarithm of the heat kernel for
  manifolds with bounded {R}icci curvature.
\newblock {\em J. Funct. Anal.}, 238(2):518--529, 2006.

\bibitem[ET81]{Elworthy-Truman-81}
David Elworthy and Aubrey Truman.
\newblock Classical mechanics, the diffusion (heat) equation and the
  {S}chr\"odinger equation on a {R}iemannian manifold.
\newblock {\em J. Math. Phys.}, 22(10):2144--2166, 1981.

\bibitem[Fan99]{Fang-99}
Shizan Fang.
\newblock Integration by parts formula and logarithmic {S}obolev inequality on
  the path space over loop groups.
\newblock {\em Ann. Probab.}, 27(2):664--683, 1999.

\bibitem[FF97]{Fang-Franchi-97}
Shizan Fang and Jacques Franchi.
\newblock De {R}ham-{H}odge-{K}odaira operator on loop groups.
\newblock {\em J. Funct. Anal.}, 148(2):391--407, 1997.

\bibitem[Fre85]{Freidlin-book}
Mark Freidlin.
\newblock {\em Functional integration and partial differential equations},
  volume 109 of {\em Annals of Mathematics Studies}.
\newblock Princeton University Press, Princeton, NJ, 1985.

\bibitem[FS86]{Fabes-Stroock}
E.~B. Fabes and D.~W. Stroock.
\newblock A new proof of {M}oser's parabolic {H}arnack inequality using the old
  ideas of {N}ash.
\newblock {\em Arch. Rational Mech. Anal.}, 96(4):327--338, 1986.

\bibitem[GHL08]{LDUP}
Alexander Grigor'yan, Jiaxin Hu, and Ka-Sing Lau.
\newblock Obtaining upper bounds of heat kernels from lower bounds.
\newblock {\em Comm. Pure Appl. Math.}, 61(5):639--660, 2008.

\bibitem[GM98]{Gong-Ma}
Fu-Zhou Gong and Zhi-Ming Ma.
\newblock The log-{S}obolev inequality on loop space over a compact
  {R}iemannian manifold.
\newblock {\em J. Funct. Anal.}, 157(2):599--623, 1998.

\bibitem[Gri99]{Grigoryan99}
Alexander Grigor{\cprime}yan.
\newblock Estimates of heat kernels on {R}iemannian manifolds.
\newblock In {\em Spectral theory and geometry ({E}dinburgh, 1998)}, volume 273
  of {\em London Math. Soc. Lecture Note Ser.}, pages 140--225. Cambridge Univ.
  Press, Cambridge, 1999.

\bibitem[Gro91]{Gross-91}
Leonard Gross.
\newblock Logarithmic {S}obolev inequalities on loop groups.
\newblock {\em J. Funct. Anal.}, 102(2):268--313, 1991.

\bibitem[GRW01]{Gong-Rockner-Wu}
Fuzhou Gong, Michael R{\"o}ckner, and Liming Wu.
\newblock Poincar\'e inequality for weighted first order {S}obolev spaces on
  loop spaces.
\newblock {\em J. Funct. Anal.}, 185(2):527--563, 2001.

\bibitem[G{\"u}n12]{Gueneysu}
Batu G{\"u}neysu.
\newblock On generalized {S}chr\"odinger semigroups.
\newblock {\em J. Funct. Anal.}, 262(11):4639--4674, 2012.

\bibitem[Ham93]{Hamilton}
Richard~S. Hamilton.
\newblock A matrix {H}arnack estimate for the heat equation.
\newblock {\em Comm. Anal. Geom.}, 1(1):113--126, 1993.

\bibitem[HSC01]{Hebisch-Saloff-Coste}
W.~Hebisch and L.~Saloff-Coste.
\newblock On the relation between elliptic and parabolic {H}arnack
  inequalities.
\newblock {\em Ann. Inst. Fourier (Grenoble)}, 51(5):1437--1481, 2001.

\bibitem[Hsu99]{HsuEstimates}
Elton~P. Hsu.
\newblock Estimates of derivatives of the heat kernel on a compact {R}iemannian
  manifold.
\newblock {\em Proc. Amer. Math. Soc.}, 127(12):3739--3744, 1999.

\bibitem[Jos11]{Jost}
J{\"u}rgen Jost.
\newblock {\em Riemannian geometry and geometric analysis}.
\newblock Universitext. Springer, Heidelberg, sixth edition, 2011.

\bibitem[Li92]{Li-thesis}
Xue-Mei Li.
\newblock Stochastic differential equations on noncompact manifolds.
\newblock University of Warwick Thesis, 1992.

\bibitem[Li94a]{Li-moment}
Xue-Mei Li.
\newblock Stochastic differential equations on noncompact manifolds: moment
  stability and its topological consequences.
\newblock {\em Probab. Theory Related Fields}, 100(4):417--428, 1994.

\bibitem[Li94b]{Li-flow}
Xue-Mei Li.
\newblock Strong {$p$}-completeness of stochastic differential equations and
  the existence of smooth flows on noncompact manifolds.
\newblock {\em Probab. Theory Related Fields}, 100(4):485--511, 1994.

\bibitem[Li05]{XDLi-Liouville}
Xiang-Dong Li.
\newblock Liouville theorems for symmetric diffusion operators on complete
  {R}iemannian manifolds.
\newblock {\em J. Math. Pures Appl. (9)}, 84(10):1295--1361, 2005.

\bibitem[Li16a]{XDLi-Hamilton}
Xiang-Dong Li.
\newblock Hamilton's {H}arnack inequality and the {$W$}-entropy formula on
  complete {R}iemannian manifolds.
\newblock {\em Stochastic Process. Appl.}, 126(4):1264--1283, 2016.

\bibitem[Li16b]{Li-ibp-sc}
Xue-Mei Li.
\newblock On the semi-classical brownian bridge measure.
\newblock arxiv:1607.06498, 2016.

\bibitem[LS11]{Lack-flow}
Xue-Mei Li and Michael Scheutzow.
\newblock Lack of strong completeness for stochastic flows.
\newblock {\em Ann. Probab.}, 39(4):1407--1421, 2011.

\bibitem[LT16]{Li-Thompson}
Xue-Mei Li and J.~Thompson.
\newblock First order feynman-kac formula.
\newblock arxiv:1608.03856, 2016.

\bibitem[LX11]{Li-Xu}
Junfang Li and Xiangjin Xu.
\newblock Differential {H}arnack inequalities on {R}iemannian manifolds {I}:
  linear heat equation.
\newblock {\em Adv. Math.}, 226(5):4456--4491, 2011.

\bibitem[LY86]{Li-Yau}
Peter Li and Shing-Tung Yau.
\newblock On the parabolic kernel of the {S}chr\"odinger operator.
\newblock {\em Acta Math.}, 156(3-4):153--201, 1986.

\bibitem[Mol68]{Molchanov}
S.~A. Mol{\v{c}}anov.
\newblock The strong {F}eller property of diffusion processes on smooth
  manifolds.
\newblock {\em Teor. Verojatnost. i Primenen.}, 13:493--498, 1968.

\bibitem[MS96]{Malliavin-Stroock}
Paul Malliavin and Daniel~W. Stroock.
\newblock Short time behavior of the heat kernel and its logarithmic
  derivatives.
\newblock {\em J. Differential Geom.}, 44(3):550--570, 1996.

\bibitem[Ndu91]{Ndumu91}
Martin~Ngu Ndumu.
\newblock The heat kernel formula in a geodesic chart and some applications to
  the eigenvalue problem of the {$3$}-sphere.
\newblock {\em Probab. Theory Related Fields}, 88(3):343--361, 1991.

\bibitem[Ndu09]{Ndumu-09}
Martin~N. Ndumu.
\newblock Heat kernel expansions in vector bundles.
\newblock {\em Nonlinear Anal.}, 71(12):e445--e473, 2009.

\bibitem[Nor93]{Norris-93}
J.~R. Norris.
\newblock Path integral formulae for heat kernels and their derivatives.
\newblock {\em Probab. Theory Related Fields}, 94(4):525--541, 1993.

\bibitem[Nor97]{Norris97-long}
J.~R. Norris.
\newblock Long-time behaviour of heat flow: global estimates and exact
  asymptotics.
\newblock {\em Arch. Rational Mech. Anal.}, 140(2):161--195, 1997.

\bibitem[NS91]{Norris-Stroock}
James~R. Norris and Daniel~W. Stroock.
\newblock Estimates on the fundamental solution to heat flows with uniformly
  elliptic coefficients.
\newblock {\em Proc. London Math. Soc. (3)}, 62(2):373--402, 1991.

\bibitem[Rin04]{Rincon}
L.~A. Rinc{\'o}n.
\newblock Path integral representation for the solution of a stochastic
  {S}chrodinger equation driven by a semimartingale.
\newblock {\em J. Interdiscip. Math.}, 7(2):183--204, 2004.

\bibitem[She91]{Sheu}
Shuenn~Jyi Sheu.
\newblock Some estimates of the transition density of a nondegenerate diffusion
  {M}arkov process.
\newblock {\em Ann. Probab.}, 19(2):538--561, 1991.

\bibitem[Shi86]{Shigekawa}
Ichir{\=o} Shigekawa.
\newblock de {R}ham-{H}odge-{K}odaira's decomposition on an abstract {W}iener
  space.
\newblock {\em J. Math. Kyoto Univ.}, 26(2):191--202, 1986.

\bibitem[Shu99]{Shubin-99}
Mikhail Shubin.
\newblock Spectral theory of the {S}chr\"odinger operators on non-compact
  manifolds: qualitative results.
\newblock In {\em Spectral theory and geometry ({E}dinburgh, 1998)}, volume 273
  of {\em London Math. Soc. Lecture Note Ser.}, pages 226--283. Cambridge Univ.
  Press, Cambridge, 1999.

\bibitem[Sim82]{Simon-82}
Barry Simon.
\newblock Schr\"odinger semigroups.
\newblock {\em Bull. Amer. Math. Soc. (N.S.)}, 7(3):447--526, 1982.

\bibitem[ST98]{Stroock-Turetsky-98}
Daniel~W. Stroock and James Turetsky.
\newblock Upper bounds on derivatives of the logarithm of the heat kernel.
\newblock {\em Comm. Anal. Geom.}, 6(4):669--685, 1998.

\bibitem[Str00]{Stroock2000}
Daniel~W. Stroock.
\newblock {\em An introduction to the analysis of paths on a {R}iemannian
  manifold}, volume~74 of {\em Mathematical Surveys and Monographs}.
\newblock American Mathematical Society, Providence, RI, 2000.

\bibitem[Stu93]{Sturm-93}
Karl-Theodor Sturm.
\newblock Schr\"odinger semigroups on manifolds.
\newblock {\em J. Funct. Anal.}, 118(2):309--350, 1993.

\bibitem[TW98]{Thalmaier-Wang}
Anton Thalmaier and Feng-Yu Wang.
\newblock Gradient estimates for harmonic functions on regular domains in
  {R}iemannian manifolds.
\newblock {\em J. Funct. Anal.}, 155(1):109--124, 1998.

\bibitem[Var90]{Varopoulos}
N.~Th. Varopoulos.
\newblock Small time {G}aussian estimates of heat diffusion kernels. {II}.
  {T}he theory of large deviations.
\newblock {\em J. Funct. Anal.}, 93(1):1--33, 1990.

\bibitem[Wal42]{Walker}
A.~G. Walker.
\newblock Note on a distance invariant and the calculation of {R}use's
  invariant.
\newblock {\em Proc. Edinburgh Math. Soc. (2)}, 7:16--26, 1942.

\bibitem[Wat88]{Watling1}
Keith~D. Watling.
\newblock Formul\ae\ for the heat kernel of an elliptic operator exhibiting
  small-time asymptotics.
\newblock In {\em Stochastic mechanics and stochastic processes ({S}wansea,
  1986)}, volume 1325 of {\em Lecture Notes in Math.}, pages 167--180.
  Springer, Berlin, 1988.

\bibitem[WW09]{Wei-Wylie}
Guofang Wei and Will Wylie.
\newblock Comparison geometry for the {B}akry-{E}mery {R}icci tensor.
\newblock {\em J. Differential Geom.}, 83(2):377--405, 2009.

\bibitem[Yau76]{Yau76}
Shing~Tung Yau.
\newblock Some function-theoretic properties of complete {R}iemannian manifold
  and their applications to geometry.
\newblock {\em Indiana Univ. Math. J.}, 25(7):659--670, 1976.

\bibitem[Zha01]{Zhang-01}
Qi~S. Zhang.
\newblock Global bounds of {S}chr\"odinger heat kernels with negative
  potentials.
\newblock {\em J. Funct. Anal.}, 182(2):344--370, 2001.

\end{thebibliography}

\end{document}